%% file: main.tex
\pdfoutput=1 

\PassOptionsToPackage{linktocpage=true}{hyperref}

\documentclass[notheorems]{colt2023} 

\usepackage{enumitem}
\usepackage{amssymb}
\usepackage{amsmath}

\usepackage{mathrsfs}		
\usepackage{dsfont}

\usepackage{url}					
\hypersetup{colorlinks=true, urlcolor=blue, linktoc = all}
\usepackage{booktabs}
\usepackage{multirow}
\usepackage{hhline}
\usepackage{multirow}

\usepackage[utf8]{inputenc} 
\usepackage[T1]{fontenc}

\usepackage{multicol}
\usepackage{silence}
\usepackage{amsfonts,thmtools}
\usepackage{xparse}
\usepackage{xargs}
\usepackage{enumitem} 
\usepackage{etoolbox}
\usepackage{mathtools}
\usepackage{complexity}
\usepackage{svg}
\usepackage{soul}
\usepackage{tablefootnote}
\usepackage[bb=boondox]{mathalpha} 
\definecolor{lightgrey}{rgb}{0.9,0.9,0.9}
\sethlcolor{lightgrey}

\definecolor{mygray}{rgb}{0.6,0.6,0.6}

\input{make_cleveref_work.tex}

\newtheorem{theorem}{Theorem}
\newtheorem{lemma}[theorem]{Lemma}
\newtheorem{proposition}[theorem]{Proposition}
\newtheorem{remark}[theorem]{Remark}
\newtheorem{corollary}[theorem]{Corollary}
\newtheorem{definition}[theorem]{Definition}

\newtheorem{fact}[theorem]{Fact}

\crefname{equation}{}{}

\crefname{enumi}{Statement}{Statements}

\newtheorem{assumption}[theorem]{Assumption}

\usepackage{tikz}			
\usepackage{times}

\title[Accelerated Methods for Riemannian Min-Max Ensuring Bounded Geometric Penalties]{Accelerated Methods for Riemannian Min-Max Optimization Ensuring Bounded Geometric Penalties}

\usepackage{times}

\coltauthor{%
\Name{David Martínez-Rubio$^\ast$} \Email{\href{mailto:martinez-rubio@zib.de}{martinez-rubio@zib.de}}\\
\addr Zuse Institute Berlin and Technische Universität\\ Berlin, Germany%
\AND
\Name{Christophe Roux$^\ast$} \Email{\href{mailto:roux@zib.de}{roux@zib.de}}\\
\addr Zuse Institute Berlin and Technische Universität\\ Berlin, Germany%
\AND
\Name{Christopher Criscitiello} \Email{\href{mailto:christopher.criscitiello@epfl.ch}{christopher.criscitiello@epfl.ch}}\\
\addr EPFL, Institute of Mathematics\\Lausanne, Switzerland %
\AND
\Name{Sebastian Pokutta} \Email{\href{mailto:pokutta@zib.de}{pokutta@zib.de}}\\
\addr Zuse Institute Berlin and Technische Universität\\ Berlin, Germany%
}

\hypersetup{pdfborder={0 0 0}}

\newcommand{\norm}[1]{\| #1 \|} 
\newcommand{\norml}[1]{\left\| #1 \right\|} 
\newcommand{\abs}[1]{\lvert #1 \rvert}

\newcommand*\circledaux[1]{\tikz[baseline=(char.base)]{
\node[shape=circle,draw,inner sep=0.8pt] (char) {#1};}}

\NewDocumentCommand{\circled}{m o }{%
\IfNoValueTF{#2}{\circledaux{#1}}{\stackrel{\circledaux{#1}}{#2}}%
}

\newcommand{\defi}{\stackrel{\mathrm{\scriptscriptstyle def}}{=}}

\renewcommand*\R{\mathbb{R}}

\let\epsilon\varepsilon

\usepackage{pifont}

\DeclareMathOperator*{\argmax}{arg\,max}                
\DeclareMathOperator*{\argmin}{arg\,min}                

\makeatletter
\renewcommand\paragraph{\@startsection{paragraph}{4}{\z@}%
                                {0ex \@plus0.5ex \@minus.2ex}%
                                {-1em}%
                                {\normalfont\normalsize\bfseries}}
\makeatother

\usepackage{todonotes}

\newif\ifTODO 
 \TODOfalse

\ifTODO
\paperwidth=\dimexpr \paperwidth + 6cm\relax
\oddsidemargin=\dimexpr\oddsidemargin + 3cm\relax
\evensidemargin=\dimexpr\evensidemargin + 3cm\relax
\marginparwidth=\dimexpr \marginparwidth + 3cm\relax
\fi

\newcommand{\innp}[1]{\langle #1 \rangle}
\newcommand{\bigo}[1]{O( #1 )}
\newcommand{\bigol}[1]{O\left( #1 \right)}

\newcommand\blfootnote[1]{%
\begingroup
\renewcommand\thefootnote{}\footnote{#1}%
\addtocounter{footnote}{-1}%
\endgroup
}

\input{algorithm_config.tex}

\input{bibliography_config.tex}
\input{definitions.tex}

\begin{document}

\maketitle

\blfootnote{$^\ast$Equal contribution.}
\blfootnote{Most notations in this work have a link to their definitions, using \href{https://damaru2.github.io/general/notations_with_links/}{this code}. For example, if you click or tap on any instance of $\expon{x}(\cdot)$, you will jump to the place where it is defined as the exponential map of a Riemannian manifold.}

\begin{abstract}
In this work, we study optimization problems of the form $\min_x \max_y f(x, y)$, where $f(x, y)$ is defined on a product Riemannian manifold $\mathcal{M} \times \mathcal{N}$ and is $\mux$-strongly geodesically convex (g-convex) in $x$ and $\muy$-strongly g-concave in $y$, for $\mux, \muy \geq 0$. We design accelerated methods when $f$ is $(\Lx, \Ly, \Lxy)$-smooth and $\mathcal{M}$, $\mathcal{N}$ are Hadamard. To that aim we introduce new g-convex optimization results, of independent interest: we show global linear convergence for metric-projected Riemannian gradient descent and improve existing accelerated methods by reducing geometric constants. Additionally, we complete the analysis of two previous works applying to the Riemannian min-max case by removing an assumption about iterates staying in a pre-specified compact set.
\end{abstract}

\clearpage
\tableofcontents
\clearpage

\section{Introduction}\label{sec:introduction}
A wide array of recently developed machine learning methods can be phrased as min-max optimization problems. This has lead to a renewed interest in optimization methods for min-max algorithms \citep{gidel2019variational,mokhtari2020convergence,mokhtari2020unified,lin2020near,wang20improved}.
Applications include generative adversarial networks \citep{goodfellow2014generative}, and adversarial  as well as distributionally robust classifiers \citep{goodfellow2015explaining,duchi2018learning}, among others.
In this work, we study a class of min-max problems over Riemannian manifolds.

Riemannian optimization, the study of optimizing functions defined over Riemannian manifolds, is motivated by the following two reasons.
First, some constrained optimization problems can be expressed as unconstrained optimization problems over Riemannian manifolds.
And second, some non-convex Euclidean problems such as operator scaling \citep{allen-zhu2018operator}, can be rephrased as geodesically convex (g-convex) problems on Riemannian manifolds, which means global minima can be found efficiently despite nonconvexity.
Geodesic convexity is a generalization of convexity to Riemannian manifolds.
Some examples of machine learning tasks which can be phrased as g-convex, g-concave min-max problems are the robust matrix Karcher mean, constrained g-convex optimization on manifolds via the augmented Lagrangian, and more generally the distributionally robust version of any finite-sum, g-convex optimization problem, cf. \citep{zhang2022minimax,jordan2022first}.
Other related Riemannian min-max problems, which do not satisfy the g-convex, g-concave assumption include projection-robust optimal transport \citep{jiang2022riemannian} and geometry-aware robust PCA \citep{zhang2022minimax}.

Another motivation for studying g-convex settings is its potential to shed light on the non-g-convex case.  Indeed, in Euclidean optimization, a deep understanding of convex problems has led to optimal methods for approximating stationary points. In fact a variety of these algorithms run convex methods as subroutines \citep{jin2017escape,carmon2019convex,li2022restarted}.
 \begin{table}[ht!]
  \centering
  \setlength{\fboxsep}{1pt}
     \caption{Summary of the results in this work and comparisons with previous works. See \cref{sec:preliminaries} for the notation or click on it. $\zeta$ and $\delta$ are constants determined by the curvature and diameter of the domain, and are $1$ in the Euclidean case. We use $\kappa_{\lambda}$ for $1/(\lambda\mu)$, where $\lambda$ is a proximal parameter. In column \textbf{K}, H stands for Hadamard, R stands for Riemannian manifolds. Our contributions are \fcolorbox[HTML]{EFEFEF}{EFEFEF}{in gray.}}
     \label{table:comparisons:riemannian_min_max}
     \begin{tabular}{lccc}
     \toprule
     \textbf{Method}   & \textbf{Complexity} & \textbf{Notes} &\textbf{K}\\
     \midrule
     \midrule
     \multicolumn{4}{c}{\textsc{g-convex}} \\
     \midrule
         (\hyperlink{cite.martinez2022accelerated}{MP22}, \PRGD{})  & $\bigotilde{\zeta[D]\Lsmooth/\mu}$ & small diam. $D$ \& $\nabla \f(x^\ast)=0$ & H \\
     (\hyperlink{cite.martinez2022accelerated}{MP22}, Thm 4)  & $\bigotilde{\zetad\sqrt{\kappa_{\lambda}}}$ & accelerated, g-convex & H \\
     \rowcolor[HTML]{EFEFEF}
     \rule{0pt}{12px}\PRGD{}  & $\bigotilde{\zeta[R]\Lsmooth/\mu }$ & global, $R\defi \Lips(\f, \X)/\Lsmooth$ & H \\
     \rowcolor[HTML]{EFEFEF}
     \cref{alg:riemacon_sc_absolute_criterion}  & $\bigotilde{\sqrt{\zetad \kappa_{\lambda}} + \zetad}$ & accelerated, g-convex & H \\
     \midrule
     \multicolumn{4}{c}{\textsc{min-max}} \\
     \midrule
     (\hyperlink{cite.jordan2022first}{JLV22}, \RCEG{}-\SCSC{})  & $\bigotilde{\sqrt{\zetad/\deltad}\cdot \L/\mu+ 1/\deltad}$ & \cellcolor[HTML]{EFEFEF}  & R \\
   \vspace{1px}(\hyperlink{cite.zhang2022minimax}{ZZS22}, \RCEG{}-\CC{})  & $\bigol{\sqrt{\zetad/\deltad}\cdot \L\DExGr^2/\epsilon }$ & \multirow{-2}{*}{\cellcolor[HTML]{EFEFEF}\shortstack{ We remove strong \\ assumptions (see \cref{sec:extra_gradient})} } &  R\\
         \rule{0pt}{3ex}\cellcolor[HTML]{EFEFEF}{\normalfont\hyperref[alg:ramma]{\normalcolor RAMMA}}-\SCSC{} &\cellcolor[HTML]{EFEFEF} $\bigotilde{\zetad^{4.5}\sqrt{\frac{\Lx}{\mux}{+}\frac{\zetad \L\Lxy}{\mux\muy}{+}\frac{\Ly}{\muy}{+}\zetad^2}}$  &\cellcolor[HTML]{EFEFEF}  & \cellcolor[HTML]{EFEFEF}H \\
         \rule{0pt}{3ex}\cellcolor[HTML]{EFEFEF}{\normalfont\hyperref[alg:ramma]{\normalcolor RAMMA}}-\SCC{} &\cellcolor[HTML]{EFEFEF} $\bigotilde{\zetad^{5.5}\frac{\L\D}{\sqrt{\mux\epsilon}}}$  &\cellcolor[HTML]{EFEFEF}&  \cellcolor[HTML]{EFEFEF}H \\
    \vspace{1px}  \rule{0pt}{3ex}\cellcolor[HTML]{EFEFEF}{\normalfont\hyperref[alg:ramma]{\normalcolor RAMMA}}-\CC{} &\cellcolor[HTML]{EFEFEF} $\bigotilde{\zetad^{4.5}\sqrt{ \frac{\L\D^2}{\epsilon}}{+} \zetad^{5.5}\sqrt{\Lxy\L}\frac{\D^2}{\epsilon}}$ & \cellcolor[HTML]{EFEFEF}\multirow{-4.3}{*}{\shortstack{$(\Lx, \Ly, \Lxy)$-smooth case, \\ where $\L{}\defi{}\max\{\Lx, \Ly, \Lxy\}$}} & \cellcolor[HTML]{EFEFEF}H \\[3px]
       \rule{0pt}{3ex}\cellcolor[HTML]{EFEFEF}{\normalfont\hyperref[def:ramma-wc]{\normalcolor RAMMA-WC}} &\cellcolor[HTML]{EFEFEF} $\bigotildel{\zeta^{4}\frac{\L\Deltazero}{\epsilon^2}\sqrt{\zeta +\frac{\L}{\muy}}}$ & \cellcolor[HTML]{EFEFEF} $\varepsilon$-stationary point (\cref{def:min-stat})& \cellcolor[HTML]{EFEFEF}H\\[3px]
     \bottomrule
 \end{tabular}
 \end{table}
\paragraph{Main results.}
Previous works on Riemannian min-max optimization \citep{zhang2022minimax,jordan2022first} proposed and analyzed a Riemannian version of the extragradient method.  These important works have two major limitations:
\begin{itemize}
\item[(a)] They assume that the smoothness constants in $x$ and $y$ are similar, i.e., they consider the $(\L, \L, \L)$-smooth and $(\mu,\mu)$-strongly convex case.  However, in applications, often the smoothness constants in $x$ and $y$ are quite different, and this can be exploited to achieve faster algorithms.
\item[(b)] They rely on the assumption that the iterates of their algorithms stay in some compact set specified \emph{a priori} but without a mechanism to enforce such constraints\footnote{Note that this assumption is not the same as the iterates staying in \emph{some} compact set a posteriori.}.  This property is key to bound penalties in the convergence rates caused by the geometry, as these penalties grow with increasing distances between the iterates and a saddle point.
\end{itemize}
We address both these limitations, ensuring bounding geometric penalties.
Our main contribution is a Riemannian Accelerated Min-Max Algorithm \newtarget{def:acronym_riemannian_accelerated_min_max_algorithm}{(\RAMMA{})}, which works in the finer setting of $(\Lx, \Ly, \Lxy)$-smoothness and $(\mux, \muy)$-strong g-convexity, addressing limitation (a), and enforces the iterates to stay in a predefined compact set via projections, addressing limitation (b).
\RAMMA{} works by reducing the strongly g-convex, strongly g-concave (\SCSC{}) min-max problem to a sequence of strongly g-convex minimization problems.
Also, via reductions to the \SCSC{} case, \RAMMA{} can be used to solve g-convex, g-concave (\CC{}) and strongly g-convex, g-concave (\SCC{}) min-max problems, obtaining the accelerated rate on $\epsilon$ in the \SCC{} case for the first time on Riemannian manifolds.
Furthermore, using a variation of \RAMMA{}, which we refer to as \RAMMAWC{} (weakly-convex), we can extend our analysis to the non-convex, strongly g-convex (\NCSC{}) setting.

In order to implement subroutines used in \RAMMA{}, we prove linear convergence of Projected Riemannian Gradient Descent (\PRGD) for smooth and constrained strongly g-convex functions defined on Hadamard manifolds, which is a fundamental result and an important piece of our min-max algorithm.  The best previous analysis for \PRGD{} \citep{martinez2022accelerated} was limited to the case where the diameter of the domain is sufficiently small ($\zeta[\D] < 2$) and a point with $0$ gradient is inside of this set: we remove these assumptions.
Additionally, we improve the state of the art on accelerated inexact proximal point algorithms for g-convex Hadamard optimization. 
Our new method
improves the rates from $\bigotilde{\zetad /\sqrt{\lambda\mu}}$ to $\bigotilde{\sqrt{\zetad/(\lambda\mu)}+\zetad}$ for $\mu$-strongly g-convex problems, where $\lambda > 0$ is a proximal parameter and $\zeta \geq 1$ is the geometric penalty, cf. \cref{sec:preliminaries}.

To further address limitation (b), we show that, with the right choice of learning rates, the algorithms from \citep{zhang2022minimax,jordan2022first} do stay in a ball around the global saddle point, whose radius is two times the initial distance to the saddle.
\cref{table:comparisons:riemannian_min_max} provides a detailed comparison between our results and prior work.

We provide a general version of Sion's theorem which holds on Riemannian manifolds, and improves on prior  results \citep[Theorem 3.1]{zhang2022minimax} by ensuring the existence of a saddle point $f(\xast, \yast) = \max_{y\in\Y}\min_{x\in \X}f(x, y) = \min_{x\in \X}\max_{y\in\Y}f(x, y)$ under weak assumptions. In particular $\X$ and $\Y$ are not required to be compact in, for example, the strongly g-convex, strongly g-concave case.

\subsection{Related Work}\label{sec:related_work}

\paragraph{Euclidean min-max optimization.}
The extragradient (\newtarget{def:acronym_extra_gradient}{\EG{}}) algorithm is an optimal algorithm for solving min-max problems achieving optimal rates of $\bigo{\L/\mu}$ for $\L$-smooth and $\mu$-\SCSC{} functions.
The convergence of \EG{} was first introduced in \citet{korpelevich1976extragradient} under the assumption that the optimization was performed over compact sets.  It was recently proven in \citet{mokhtari2020convergence,mokhtari2020unified} that this compactness is not needed in either the \SCSC{} case or the \CC{} case.
Recently, some works \citep{lin2020near,wang20improved} have achieved accelerated rates for more fine-grained smoothness and strong-convexity assumptions, where the constants with respect to each variable can differ, among other things, cf. \cref{ass:fct}.
\citet{lin2020near} introduced an algorithm with accelerated rates of $\bigotilde{\sqrt{\frac{\L^2}{\mux\muy}}}$ for $(\mux,\muy)$-\SCSC{} and $\L$-smooth functions.
Later, \citep{wang20improved} proved accelerated rates of $\bigotilde{\sqrt{\frac{\Lx}{\mux}+\frac{L\Lxy}{\mux\muy}+\frac{\Ly}{\muy}}}$ for the more general case of $(\Lx,\Ly,\Lxy)$-smooth and $(\mux,\muy)$-\SCSC{} functions.
These accelerated methods reduce the solution of the \SCSC{} min-max problem to a sequence of better conditioned strongly convex minimization problems carried out by variants of accelerated proximal point algorithms.

\paragraph{Riemannian min-max optimization.}
There are several works studying algorithms for solving monotone variational inequalities on compact subsets of Hadamard manifolds, which encompass \CC{} min-max problems as a special case.
The algorithms presented in these works are variations of the inexact proximal point algorithm \citep{bento2021inexact,li2009monotone} and the \EG{} algorithm \citep{batista2018extragradienttype,ferreira2005singularities} and come with asymptotic convergence guarantees.
Two recent works \citep{zhang2022minimax,jordan2022first} based on a variation of the Euclidean \EG{} called Riemannian corrected extragradient (\newtarget{def:acronym_riemannian_corrected_extra_gradient}{\RCEG{}}) have shown convergence rates for smooth, unconstrained min-max problems on Riemannian manifolds of bounded sectional curvature that match the Euclidean \EG{} up to geometric constants for the \CC{} \citep{zhang2022minimax} and \SCSC{} \citep{jordan2022first} case.
The convergence guarantees of \RCEG{} assume that the iterates stay in some pre-specified, compact set without any mechanism to enforce this.
In contrast, the previously mentioned works \citep{bento2021inexact,li2009monotone,batista2018extragradienttype,ferreira2005singularities}, treat Riemannian min-max problems over compact sets and explicitly enforce that the iterates stay inside these.
\citet{jordan2022first} additionally analyze non-smooth min-max problems in the \CC{} and \SCSC{} case, as well as stochastic versions of both the smooth and non-smooth case.
\citet{han2022riemannian} introduce a second-order min-max algorithm for Riemannian manifolds aimed at minimizing the gradient norm of the $\f$.
\citet{huang2020gradient} studies min-max problems in a Euclidean-Riemannian product space that are non-convex and strongly g-concave.

\paragraph{Riemannian g-convex minimization.} Optimization of g-convex functions with rates of convergence has been studied more extensively than min-max problems. \citet{zhang2016first} provided several first-order deterministic and stochastic methods applying to smooth or non-smooth problems. A long line of works have tackled the question of acceleration on Riemannian manifolds, see \citep{martinez2022accelerated} for an overview. We improve over this work by reducing geometric penalties in the rates, among other things.
Furthermore, some works studied adaptive methods \citep{kasai2019riemannian}, as well as projection-free \citep{weber2017frank,weber2019nonconvex}, saddle-point-escaping \citep{criscitiello2019efficiently,sun2019escaping,zhou2019faster, criscitiello2022accelerated}, and variance-reduced methods \citep{sato2019riemannian,kasai2018riemannian,zhang2016fast}. Some lower bounds were provided by \citet{hamilton2021nogo,criscitiello2022negative}.
Analyses of methods for non-smooth problems that work with in-manifold constraints via projection oracles were provided in \citet{zhang2016first} and further studied in other works, such as \citet{wang2021no}. For the smooth case with in-manifold constraints, \citet{martinez2022accelerated} provided an analysis of \PRGD{} under restrictive assumptions: the diameter $\D$ of the feasible set is required to be small enough so that $\zeta[\D] < 2$ and the global solution is required to be inside of the set. To the best of our knowledge, there is no other method or analysis for the smooth setting using a metric-projection oracle. We remove both of these assumptions and show that \PRGD{} enjoys linear convergence rates for strongly g-convex and smooth problems with a projection oracle to an arbitrary g-convex constraint.

\subsection{Preliminaries: Definitions and Notations} \label{sec:preliminaries}
The following definitions in Riemannian geometry cover the concepts used in this work, cf. \citep{petersen2006riemannian, bacak2014convex}. A Riemannian manifold $(\MOnly,\mathfrak{g})$ is a real $C^{\infty}$ manifold $\MOnly$ equipped with a metric $\mathfrak{g}$, which is a smoothly varying inner product. For $x \in \MOnly$, denote by $\newtarget{def:tangent_space}{\Tansp{x}}\MOnly$ the tangent space of $\MOnly$ at $x$. For vectors $v,w  \in \Tansp{x}\MOnly$, we use $\innp{v,w}_x$ for the inner product of the metric, $\norm{v}_x \defi \sqrt{\innp{v,v}_x}$ for the corresponding norm, and we omit $x$ when it is clear from context. A geodesic of length $\ell$ is a curve $\gamma : [0,\ell] \to \MOnly$ of unit speed that is locally distance minimizing. 

A set $\XX$ is said to be g-convex if every two points are connected by a geodesic that remains in $\XX$.  The set $\XX$ is said to be uniquely geodesic if every two points are connected by one and only one geodesic. The exponential map $\newtarget{def:riemannian_exponential_map}{\expon{x}} : \Tansp{x}\MOnly\to \MOnly$ takes a point $x\in\MOnly$, and a vector $v\in \Tansp{x}\MOnly$ and returns the point $y$ we obtain from following the geodesic from $x$ in the direction $v$ for length $\norm{v}$, if this is possible. We denote its inverse by $\exponinv{x}(\cdot)$, which is well defined for uniquely geodesic manifolds, so we have $\expon{x}(v) = y$ and $\exponinv{x}(y) = v$.  Note $\dist(x, y) = \norm{\exponinv{x}(y)}$. We exclusively work in uniquely geodesic manifolds, as an open hemisphere. The manifold $\MOnly$ comes with a natural parallel transport of vectors between tangent spaces, that formally is defined from a way of identifying nearby tangent spaces, known as the Levi-Civita connection $\nabla$ \citep{levi1977absolute}. Throughout this work, we use $\newtarget{def:parallel_transport}{\Gamma{y}{x}}(v) \in \Tansp{x}\M$ to denote this parallel transport for a vector $v$ in $T_y M$ from $\Tansp{y}\M$ to $\Tansp{x}\M$ along the unique geodesic that connects $y$ to $x$.  We write $\Gamma{}{x}(v)$ if $y$ is clear from context. As for all of the related works in \cref{sec:related_work}, we assume that we can compute the $\expon{x}(\cdot)$, $\exponinv{x}(\cdot)$ and $\Gamma{}{x}(\cdot)$.
We use $\newtarget{def:injectivity_radius}{\inj}(x)$ to denote the injectivity radius at $x$, that is, the largest radius $r$ of a ball $B(0, r) \subseteq \Tansp{x}\MOnly$ for which $\expon{x}$ is a diffeomorphism.

We use $\newtarget{def:closed_ball}{\ball}(x, R)$ to denote a closed Riemannian ball with center $x$ and radius $R$, and denote by $\newtarget{def:distance}{\dist}(x,y)$ the distance between $x$ and $y$,
A map $\newtarget{def:projection_operator}{\proj}:\M \to \X$ is a metric-projection operator if it satisfies $\dist(x, \proj(x)) \leq \dist(x, z)$ for all $z \in \X$. For a uniquely geodesic g-convex set $\mathcal{Z}$ a point $z \in \mathcal{Z}$ and a closed Riemannian ball $\X \defi \ball(x, D) \subset \mathcal{Z} \subset\M$ we have $\proj(z)=z$ if $z \in \X$ and $\proj(z) = \expon{x}(D\exponinv{x}(z)/\norm{\exponinv{x}(z)})$ is a relatively cheap metric-projection operator, cf. \citep{martinez2022accelerated}.
As in the Euclidean space, a differentiable function is $\newtarget{def:L_smooth}{\L}$-smooth and $\mu$-strongly g-convex in a uniquely geodesic g-convex set $\XX$, if for any two points $x, y \in \XX$ we have, respectively:
    \[
        \f(y) \leq \f(x) + \innp{\nabla \f(x), \exponinv{x}(y)} + \frac{\L}{2}\dist^2(x,y) \text{ and }  \f(y) \geq \f(x) + \innp{\nabla \f(x), \exponinv{x}(y)} + \frac{\mu}{2}\dist^2(x,y).
    \]
    The latter property is equivalent to $$f(\expon{x}(t\cdot\exponinv{x}(x) + (1-t)\cdot\exponinv{x}(y))) \leq tf(x) + (1-t)f(y) - \frac{t(1-t)\mu}{2}\dist^2(x,y),$$ for $t\in[0,1]$ and also applies to non-differentiable functions. 
    The function is said to be g-convex if $\mu=0$ and  $(-\mu)$-weakly g-convex if $\mu<0$. Further, it is $\mu$-strongly g-concave if $-f$ is $\mu$-strongly convex. The function $f$ has $\bar{L}_x$-Lipschitz gradients in $\X$ if for all $x, y \in \X$ we have $\norm{\nabla f(x) -\Gamma{}{x} \nabla f(y)} \leq \bar{L}_x \dist(x, y)$. A function $f$ is $\newtarget{def:Lipschitz_constant}{\Lips}(f,\XX)$-Lipschitz in $\XX$ if $|f(x) - f(y)| \leq \Lips(f, \XX) \dist(x, y)$ for all $x, y \in \XX$, where $\XX$ is omitted if clear from context. 
    
    A function $f(x, y)$ is $(\mux, \muy)$-\newtarget{def:acronym_strongly_g_convex_strongly_g_concave}{\SCSC{}} in $\X\times\Y$ if it is $\mux$-strongly g-convex in $\X$ and $\muy$-strongly g-concave in $\Y$. If $\muy=0$ we say the function is $\mux$-\newtarget{def:acronym_strongly_g_convex_g_concave}{\SCC{}}, if it is $\mux=\muy=0$, then it is \newtarget{def:acronym_g_convex_g_concave}{\CC{}} and if it is not necessarily g-convex and $\muy$-strongly g-concave it is \newtarget{def:acronym_non_g_convex_strongly_g_concave}{\NCSC{}}
    For a function $f:\M\times\NN \to \R$ that is \CC{} in $\X\times\Y$, a point $(\hat{x}, \hat{y})\in\X\times\Y$ is an $\epsilon$-saddle point of $f$ in $\X \times \Y$ if $\max_{y\in\Y} f(\hat{x}, y) - \min_{x\in\X} f(x,\hat{y}) \leq \epsilon$, assuming the $\max$ and $\min$ exist. We define it to be an $\epsilon$-saddle point \textit{in distance} if $\dist^2(\hat{x}, \xast) + \dist^2(\hat{y}, \yast)  \leq \epsilon$, where $(\xast, \yast)$ is a $0$-saddle point in $\X\times\Y$ that satisfies $f(\xast, \yast) = \min_{x\in\X}\max_{y\in\Y} f(x, y) = \max_{y\in\Y}\min_{x\in\X} f(x, y)$, whose existence we can guarantee under mild assumptions, cf. \cref{thm:sion}. We do not specify the saddle point is taken with respect to $\X\times\Y$ when it is clear from context.

The sectional curvature of a manifold $\mathcal{M}$ at a point $x\in\mathcal{M}$ for a $2$-dimensional space $V\subset T_x\M$ is the Gauss curvature of $\expon{x}(V)$ at $x$. Hadamard manifolds are complete simply-connected Riemannian manifolds of non-positive sectional curvature, like the hyperbolic space or the space of $n\times n$ symmetric positive definite matrices with the metric $\innp{X, Y}_{A} \defi \operatorname{Tr}(A^{-1}XA^{-1}Y)$ where $X, Y$ are in the tangent space of $A$. They are uniquely geodesic and $\expon{x}(\cdot)$ is well defined on them for every $v\in \Tansp{x}\MOnly$. 
Given $R>0$, and a manifold of sectional curvature bounded in $[\kmin, \kmax]$, we define the geometric constants $\newtarget{def:zeta}{\zeta[R]} \defi R\sqrt{\abs{\kmin}}\coth(R\sqrt{\abs{\kmin}}) \geq 1$ if $\kmin < 0$  and $\zeta[R]\defi 1$ otherwise, and $\newtarget{def:delta}{\delta[R]} \defi R\sqrt{\kmax}\cot(R\sqrt{\kmax})\leq 1$ if $\kmax > 0$ and $\delta[R]\defi 1$ otherwise. For a set $\X$ of diameter $\D$, we use $\zeta[][\X] \defi \zeta[\D]$ or just $\newtarget{def:zeta_without_subindex}{\zetad}$ if $\X$ is clear from context. For the product $\X \times\Y$, we abuse the notation and use $\zeta \defi \max\{\zeta[][\X], \zeta[][\Y]\}$. Similarly for the notation $\delta[][\X]$ and $\newtarget{def:delta_without_subindex}{\deltad}$. These constants appear in Riemannian optimization analysis via the Riemannian cosine law \cref{lemma:cosine_law_riemannian} or other similar inequalities. 
The big-$O$ notation $\newtarget{def:big_o_tilde}{\bigotilde{\cdot}}$ omits $\log$ factors.

\subsection{Problem Setting} \label{sec:setting}

In this work, $\newtarget{def:M_manifold}{\M}$ and $\NN$ always represent two uniquely geodesic finite-dimensional Riemannian manifolds of sectional curvature bounded by $[\newtarget{def:minimum_sectional_curvature}{\kmin}, \newtarget{def:maximum_sectional_curvature}{\kmax}]$, and $\newtarget{def:geodesically_convex_feasible_region}{\X}\subset \M$, $\newtarget{def:geodesically_convex_feasible_region_Y}{\Y}\subset\NN$ are compact g-convex subsets for which we have access to metric-projection oracles.
We consider the following optimization problem
\begin{equation}\label{eq:constrained_saddle_point_problem}
  \min_{x\in \X}\max_{y\in \Y}\f(x,y),
\end{equation}
where $\newtarget{def:riemannian_function_f}{\f}:\M\times\NN \to\R$ denotes a function with a saddle point at $(\newtarget{def:x_part_of_optimum_saddle_point}{\xast},\newtarget{def:y_part_of_optimum_saddle_point}{\yast})\in \X\times \Y$, that satisfies $\nabla \f(\xast, \yast) = 0$.
Let $(\xk[0], \yk[0]) \in \X\times\Y$ be an initial point.
Define $\newtarget{def:diameter_of_geodesically_convex_feasible_region}{\D} \defi \max\{\operatorname{diam}(\X), \operatorname{diam}(\Y)\}$.
Our aim is to compute an $\newtarget{def:accuracy_epsilon}{\epsilon}$-saddle point of $\f$ over $\X\times\Y$, where $\f$ satisfies the following \cref{ass:fct}, with constants $\newtarget{def:strong_g_convexity_mu_x}{\mux}, \newtarget{def:strong_g_convexity_mu_y}{\muy}, \newtarget{def:smoothness_Lx}{\Lx}, \newtarget{def:smoothness_Ly}{\Ly}, \newtarget{def:smoothness_Lxy}{\Lxy}$.  We also assume without loss of generality that $\Lx = \Ly$, and $\muy \leq \mux$. Indeed, we can rescale the manifolds to obtain $\Lx = \Ly$ which keeps $\Lxy$, $\Lx/\mux$, $\Ly/\muy$, $\mux\muy$ constant, as well as geometric penalties depending on $\zetad$, see \cref{sec:rescaling_manifolds}. Also, if $\muy > \mux$, we can work with the function $h(x, y) = - \f(y, x)$.
We write $\newtarget{def:max_of_Lx_Ly_and_Lxy}{\L}\defi\max\{\Lx,\Ly,\Lxy\}$, $\kappa_x \defi \Lx / \mux$, and $\kappa_y \defi \Ly / \muy$.  We say a function is $(\bar{L}_x, \bar{L}_y, \bar{L}_{xy})$-smooth if it satisfies \cref{item:LxLy,item:Lxy} below.
\begin{assumption}\label{ass:fct}
    Let $\M$, $\NN$, $\X$, $\Y$ be as above, and let $g:\M\times \NN\rightarrow \mathbb{R}$ be differentiable. For any $(x, y) \in \X\times\Y$, it holds:
  \begin{enumerate}
      \item $g$ is $(\bar{\oldmu}_x, \bar{\oldmu}_y)$-\SCSC{} in $\X \times \Y$.
    \item $\nabla_x g(\cdot,y)$ is $\bar{L}_x$-Lipschitz in $\X$ and $\nabla_y g(x,\cdot)$ is $\bar{L}_y$-Lipschitz in $\Y$.\label{item:LxLy}
    \item $\nabla_y g(\cdot,y)$ and $\nabla_x g(x,\cdot)$ are $\bar{L}_{xy}$-Lipschitz in $\X$ and $\Y$, respectively. \label{item:Lxy}
  \end{enumerate}
\end{assumption}

\section{Generalized Riemannian Sion's Theorem}\label{sec:sions_thm}
We first generalize Sion's theorem \citep{sion1958general,zhang2022minimax} to Riemannian manifolds under mild assumptions, which in particular are satisfied if $f$ is \SCSC{}, ensuring the existence of a saddle point in this case.
For Riemannian manifolds, this theorem generalizes \citep{zhang2022minimax} which required the sets $\X$, $\Y$ to be compact.
\begin{definition}
    A function $f: \X\times \Y\rightarrow \mathbb{R}$ is called inf-sup-compact at $(\tilde{x},\tilde{y})\in \X\times \Y$ if the sublevel sets of $f(\cdot,\tilde{y})$  and the superlevel sets of $f(\tilde{x},\cdot)$ are compact. A function is quasi g-convex (resp. quasi g-concave) if their level sets are g-convex (resp. g-concave).
\end{definition}
\begin{theorem}[Generalized Sion's Theorem]\linktoproof{thm:sion}\label{thm:sion}
  Let $\M$ and $\NN$ be finite-dimensional Riemannian manifolds.
  Further, let $\X\subset \M$ and $\Y\subset \NN$ be g-convex and uniquely geodesic subsets.
  Let $f:\X\times \Y\rightarrow \mathbb{R}$ be a function such that $f(\cdot,y)$ is lower semicontinuous and quasi g-convex for all $y\in \Y$, $f(x,\cdot)$ is upper semicontinuous and quasi g-concave for all $x\in \X$ and that is inf-sup compact for some $(x_1,y_1)\in \X\times \Y$.
  Then we have
  \begin{equation*}
  \min_{x\in \X}\max_{y\in \Y}f(x,y)=  \max_{y\in \Y}\min_{x\in \X}f(x,y).
  \end{equation*}
\end{theorem}

\begin{corollary}
    For $\mux, \muy > 0$, a $(\mux, \muy)$-\SCSC{} function $f: \X \times \Y \to \R$ is inf-sup compact for any point in $\X\times \Y$. If we have an $f: \X \times \Y \to \R$ that is \CC{}, then if $\X$, $\Y$ are compact then $f(x, y) + \indicator{\X}(x) - \indicator{\Y}(y)$ is inf-sup compact for any point, where $\indicator{\mathcal{C}}$ denotes the indicator function of a set $\mathcal{C}$, which is $0$ if $x \in \mathcal{C}$ and it is $+\infty$ otherwise. Similarly, if the function is $(\mux, 0)$-\SCC{} and $\Y$ is compact then $f(x, y) - \indicator{\Y}(y)$ is inf-sup compact for any point.
\end{corollary}

\section{Riemannian Corrected Extra-Gradient}\label{sec:extra_gradient}
Two previous works provide rates of convergence for Riemannian smooth min-max problems \citep{zhang2022minimax,jordan2022first} by using \RCEG{}, a Riemannian adaptation of the extragradient method, see \cref{alg:rceg}. Specifically, \citet[Theorem 4.1]{zhang2022minimax} consider \CC{} functions and \citet[Theorem 3.1]{jordan2022first} consider \SCSC{} functions. They work in the setting of \cref{sec:setting} with $\mux=\muy$, and $\L=\Lx=\Ly=\Lxy$, which is equivalent to regular gradient Lipschitzness with constant $\L$. Possible different condition numbers or different Lipschitz assumptions are not exploited by \RCEG{}, in contrast to our algorithm \RAMMA{}.

The stepsize $\etaEG$ of \RCEG{} depends on the geometric constants $\zeta[\DExGr]$ and $1/\delta[\DExGr]$ arising from the Riemannian cosine law \cref{lemma:cosine_law_riemannian}, which are larger the farther the iterates are from each other and from the global saddle. Thus, the current theory does not yield a complete algorithm unless we can guarantee that given a learning rate $\etaEG$ depending on the diameter $\DExGr$ of a set $\X$ specified a priori, the iterates of \RCEG{} stay in $\X$. \citet{zhang2022minimax} assumes the above occurs, and \citep{jordan2022first} also assumes this implicitly. Because learning rates depend on the diameter of this set, this assumption is not the same as the iterates staying in some compact set a posteriori.  

The reliance on this assumption is a recurring issue in Riemannian algorithms, and multiple other works also make this assumption in other contexts, as detailed by \citet{hosseini2020recent}.
Prior to our work, no complete algorithm was given and the precise convergence rate was unknown, since the geometric constants were not shown to be bounded by problem parameters. 

In the following \cref{lem:it-rceg}, we show that for $\DExGr$ depending on the initial distance to the saddle, for both the \CC{} and \SCSC{} cases, the iterates of \RCEG{} stay in the closed ball $\ball((\xast, \yast), \DExGr/2)$ around the saddle point, satisfying the assumption. We also note that not relying on this assumption is one of the main difficulties in the design of our algorithm \RAMMA{}.

\begin{proposition}[Riemannian Corrected Extra-Gradient]\linktoproof{lem:it-rceg}\label{lem:it-rceg}
    Let $\M$ and $\NN$ be uniquely geodesic Riemannian manifolds of sectional curvature bounded by $[\kmin,\kmax]$.
    Let $f: \M\times \NN\rightarrow \mathbb{R}$ be a function with a saddle point at $(x^{*},y^{*})$, and let $\newtarget{def:distance_to_saddle_extra_grad}{\Dcal}^2\defi \dist^2(x_0,x^{*})+\dist^2(y_{0},y^{*})$, and define $\newtarget{def:diameter_of_geodesically_convex_region_for_extra_grad}{\DExGr}$ as any upper bound to $4\Dcal$. Assume $\mathcal{B} \defi \ball((\xast, \yast), 2\Dcal)\subset \M\times\NN$. If $f$ is $(\L, \L, \L)$-smooth and $(\mu, \mu)$-\SCSC{} in $\mathcal{B}$, then for \cref{alg:rceg}, the primary iterates $(x_t,y_t)$ and the secondary iterates $(w_t,z_t)$ satisfy:
  \begin{equation*}
      \dist^2(x_t,x^{*})+\dist^2(y_{t},y^{*})\le  \Dcal^2 \quad\quad \text{ and } \quad \quad \dist^2(w_t,x^{*})+\dist^2(z_{t},y^{*})\le  4\Dcal^2,
 \end{equation*}
    and we obtain an $\epsilon$-saddle point in the $\mu$-\SCSC{} and the \CC{} cases in $T$ iterations for, respectively:
\begin{equation*}
      T=\bigotildel{\frac{\L}{\mu}\sqrt{\frac{\zeta[\DExGr]}{\delta[\DExGr]}}+\frac{1}{\delta[\DExGr]}}  \quad \text{ and }\quad\quad \quad T=\bigol{\frac{\L\Dcal^2}{\epsilon}\sqrt{\frac{\zeta[\DExGr]}{\delta[\DExGr]}}}.
\end{equation*}
\end{proposition}

\cref{lem:it-rceg} requires the manifolds to be uniquely geodesic, but does not impose any restrictions on their curvature bounds. Examples of such manifolds are Hadamard manifolds, or, when $\kmax>0$, open Riemannian balls $\openball(x, \hat{R})$, where $\hat{R} < \min\{\inj(x)/2, \pi/(2\sqrt{\kmax})\}$~\citep[Thm. IX.6.1]{chavel2006riemannian}.

\section{Riemannian Accelerated Algorithms for Minimization and Min-Max}\label{sec:accelerated_method}
The min-max algorithm \RAMMA{} we introduce in \cref{sec:riemannian-min-max} works by reducing the solution of min-max problem to solving a series of g-convex problems.
In \cref{sec:riemannian-g-c}, we introduce two g-convex optimization algorithms that are essential for \RAMMA{}.
\subsection{G-Convex Algorithms}
\label{sec:riemannian-g-c}
We first present a global linear convergence analysis for Projected Riemannian Gradient Descent (\PRGD{}) in Hadamard manifolds, an algorithm for constrained optimization of $\mu$-strongly g-convex and $\Lsmooth$-smooth functions that makes use of a metric-projection oracle and a gradient oracle. Several attempts have been made towards obtaining this result, but the best analysis \citep{martinez2022accelerated} only provides  linear convergence of \PRGD{} in a compact g-convex set with diameter $D$ small enough so that $\zeta[D]< 2$,   for a function whose global minimizer is inside of this set. Our analysis provides linear global convergence without any of these assumptions.
\begin{proposition}[Projected Riemannian Gradient Descent (PRGD)]\linktoproof{lem:prgd}\label{lem:prgd}
    \newtarget{def:acronym_projected_riemannian_gradient_descent}{Let} $f: \M\rightarrow \mathbb{R}$ be a $\mu$-strongly g-convex and $\Lsmooth$-smooth function in a g-convex compact subset $\X\subset \M$ of a Hadamard manifold $\M$. For an initial point $x_0\in\X$ and $R \defi \Lips(f, \X)/\Lsmooth$, after
$$
T \geq \min\left\{2\kappa\zeta[R]\log \left( \frac{f(x_0)-f(x^{*})}{\epsilon} \right), 1+ 2\kappa\zeta[R]\log \left(\frac{\Lsmooth\zeta[R]\dist^2(x_0,x^{*})}{2\epsilon} \right)\right\}
$$
steps of \PRGD{} with update rule $x_{t+1}\gets \proj[\X]\left( \exp_{x_t}\left(-\frac{1}{\Lsmooth}\nabla f(x_t)\right) \right)$, we have $f(x_T)-f(x^\ast) \leq \epsilon$, where $\kappa = \Lsmooth/\mu$.
\end{proposition}

We note that we show that at every iteration \PRGD{} reduces the gap by a factor of $1-\mu/(4L\zeta_{R_t})$, where $R_t \defi \norm{\nabla f(x_t)}/L \leq R$, which is a quantity that does not require knowing $\Lips(f,\XX)$.

We now present our results on g-convex accelerated optimization. For a $\mu$-strongly g-convex function over a g-convex closed subset $\X$ of a finite-dimensional Hadamard manifold $\M$, we improve over the state of the art \citep[Algorithm 1]{martinez2022accelerated} and improve over their result in two important directions. First, we obtain better convergence rates: $\bigotilde{\sqrt{\zetad\kappa_{\lambda}} + \zetad}$ as opposed to %
$\bigotilde{\zetad\sqrt{\kappa_{\lambda}}}$,
where $\kappa_{\lambda} = 1/(\lambda\bar{\oldmu})$ and $\lambda$ is the learning rate of the approximate implicit gradient descent step of the algorithm. A similar improvement is obtained for the g-convex case via reductions, see \cref{rem:g-convex_reduction}. The key idea for our improved algorithm is to work directly in the strongly g-convex case, as opposed to the g-convex case and make a careful choice of learning rates. Second, we require a less restrictive condition for the subroutine that \cref{alg:riemacon_sc_absolute_criterion} ($\riemacon$) uses, namely we only require to compute a minimizer of the subproblem in Line \ref{line:subroutine},
$$\min_{y\in\X}\{\f(y) + \frac{1}{2\lambda}\dist^2(\xk, y)\},$$
with \textit{absolute} accuracy, as opposed to \textit{relative} accuracy, i.e., accuracy proportional to the distance to the minimizer of the proximal problem, which is unknown in general. This second part is a crucial modification for the application to \RAMMA{}, our min-max algorithm, and it is inspired by the analysis of the Euclidean APPA algorithm \citep{frostig2015un-regularizing}.

\begin{theorem}[Strongly g-Convex Acceleration]\linktoproof{thm:riemaconsc}\label{thm:riemaconsc}
    Using the definitions and notation of \cref{alg:riemacon_sc_absolute_criterion} to optimize a function $f$ which is $\bar{\oldmu}$-strongly g-convex in $\XX$, we have $f(y_T) - f(x^\ast) \leq \epsilon$ after a number of iterations $T \geq 2\sqrt{\xi\max\{\frac{1}{\lambda \bar{\oldmu}}, 9\xi\}}\log_2\left(\frac{2\lambda^{-1}\dist^2(x_0, x^\ast)}{\epsilon}\right) = \bigotilde{\sqrt{\zetad\kappa_{\lambda}} + \zetad}$.
\end{theorem}
We emphasize that each iteration of \cref{alg:riemacon_sc_absolute_criterion} ($\riemacon$) in Theorem~\ref{thm:riemaconsc} requires a solution to a subproblem, which we can implement with \PRGD{}, and \cref{alg:riemacon_sc_absolute_criterion} \textit{only} accesses $f$ through these subproblem solutions: no other function values or gradients are asked of $f$.  This will prove important later on. 
If for an $\Lsmooth$-smooth $\bar{\oldmu}$-strongly g-convex function we apply \cref{alg:riemacon_sc_absolute_criterion} with $\lambda = 1/L$ and \PRGD{} as subroutine, we obtain the following \cref{cor:smooth-riemaconsc}. We instantiate this corollary in the setting of \citet[Theorem 6]{martinez2022accelerated} but without assuming anything about the Hadamard manifold beyond bounded sectional curvature, and we show that in that case $\zeta_R = \bigo{1}$. We thus obtain a more general result and a better overall complexity, cf. \cref{corol:PRGD_is_free_if_global_optim_is_in_the_set}.
We note that in this setting, if $\kmax < 0$ and $\X$ is a ball, there is a lower bound for the condition number $\kappa \geq \tilde \Omega(\sqrt{\kmax/\kmin} \zeta)$, cf. \cref{remark:minimum_kappa}.

\begin{corollary}\linktoproof{cor:smooth-riemaconsc}\label{cor:smooth-riemaconsc}
  If in addition to the assumptions from \cref{thm:riemaconsc}, $f$ is also $\Lsmooth$-smooth in $\XX$, \cref{alg:riemacon_sc_absolute_criterion} with $\lambda=1/\Lsmooth$ 
    and \PRGD{} as subroutine, yields an $\epsilon$-minimizer after $\bigotilde{\zeta[R]\zetad^{3/2}  \sqrt{\kappa+\zetad}}$  gradient and metric-projection oracle calls, where $R\le (\Lips(f, \XX)/\Lsmooth+2D_{\XX})/\zetad$ and $D_{\XX}$ is the diameter of $\XX$.
\end{corollary}

\subsection{Min-Max Algorithms}
\label{sec:riemannian-min-max}

Before turning to general Riemannian min-max problems, we first consider a specific class of functions $f(x,y)$ for which the interaction between $x$ and $y$ is weak, meaning that $\Lxy$ is small relative to other function parameters, i.e., the gradient of $f(x,y)$ with respect to $x$ is only weakly dependent on $y$ and vice versa.
\begin{algorithm}
    \caption{Riemannian Alternating Best Response \RABR{}($f, (x_0, y_0), T \text{ or } \epsilon, \X\times\Y$)}
    \label{alg:crabr}
\begin{algorithmic}[1]
    \REQUIRE G-convex subsets $\X\subset \M$, $\Y\subset \NN$ of Hadamard manifolds $\M$ and $\NN$, initialization $(x_0,y_0)\in\X\times\Y$, function $f:\M\times\NN\rightarrow\mathbb{R}$ that is $(\mux, \muy)$-\SCSC{} and $(\Lx, \Ly, \Lxy)$-smooth, $T$ (if $\epsilon$ is given, compute $T$, see \cref{thm:crabr}). Define $\xi \defi 4\max \{\zeta[2D][\X],\zeta[2D][\Y]\} -3=\bigol{\zetad}$ 
    \State $T_x\gets90\xi \sqrt{\kappa_x} \log(512)$,\ $T_y\gets90\xi \sqrt{\kappa_y} \log(512) $%
    \vspace{0.1cm}
    \hrule
    \vspace{0.1cm}
    \FOR {$t = 0 \text{ \textbf{to} } T-1$}
    \State $x_{t+1}\gets\riemaconrel(f(\cdot,y_{t}),x_t, T_x, \X, \text{\PRGD{}})$
    \State $y_{t+1}\gets\riemaconrel(-f(x_{t+1},\cdot),y_t,T_y, \Y, \text{\PRGD{}})$
    \ENDFOR
    \ENSURE $(x_T,y_T)$
\end{algorithmic}
\end{algorithm}

If the interaction between $x$ and $y$ is weak enough, alternating between minimizing $x\mapsto f(x,y_{t})$ where $y_t$ is kept fixed and maximizing $y\mapsto f(x_{t+1},y)$ where $x_{t+1}$ is kept fixed is sufficient to converge to the saddle point.
The approach of computing the optimal value of $x$ for a fixed $y$, or vice versa, can be seen as the \emph{best response} of $x$ given a fixed $y$, hence the name.
In particular, for the case of $\Lxy=0$, $x$ and $y$ have no interaction and it suffices to independently compute the best response for $x$ and $y$ once to solve the min-max problem.
Our Riemannian Alternating Best Response (\newtarget{def:acronym_riemannian_alternating_best_response}{\RABR{}}) algorithm implements this approach by repeatedly applying approximate best responses using the algorithm $\riemaconrel$, cf. \cref{sec:convergence_of_ABR}.
\RABR{} applies only to a limited class of problems, but it will be used as a subroutine for \RAMMA{}.
\RABR{} is inspired by the Euclidean algorithm of \citet[Algorithm 1]{wang20improved}.

\begin{theorem}[Convergence of \RABR{}]\linktoproof{thm:crabr}\label{thm:crabr}
  Let $f$ satisfy \cref{ass:fct} with $\Lxy< \frac{1}{2}\sqrt{\mux\muy}$.
  Then \cref{alg:crabr} requires $T=\bigotilde{  \zeta[R]\zetad^{2}\sqrt{\kappa_x+\kappa_y}}$ calls to the gradient and projection oracles to ensure $\dist^2(x_T,x^{*})+\dist^2(y_T,y^{*})\le \epsilon$, where $R=\max \left\{ \Lips(f(\cdot,y), \X)/ \Lx,\Lips(f(x,\cdot), \Y)/\Ly \right\}/\zetad +\D/\zetad$.
\end{theorem}
Now we turn to explain the intuition about our \cref{alg:ramma} (\RAMMA{}) for the general $(\Lx, \Ly, \Lxy)$-smooth and $(\mux, \muy)$-\SCSC{} case. Defining $\newtarget{def:function_max_over_y}{\phi}(x)\defi \max_{y\in \Y}f(x,y)$, one can rephrase the problem $\min_{x\in \X}\max_{y\in \Y}f(x,y)$ as $\min_{x\in \X}\phi(x)$.
By \cref{lem:strong_cvxty_of_max}, if $f(\cdot, y)$ is $\mux$-strongly g-convex, then $\phi(x)$ is as well.
Hence, we can use \cref{alg:riemacon_sc_absolute_criterion} to solve this minimization problem.
This means that at each iteration, we need to solve the subroutine  $\min_{x\in \X}\{\phi(x)+1/(2\eta_x)\dist^2(\tilde{x},x)\}$ (Line \ref{line:subroutine} of \cref{alg:riemacon_sc_absolute_criterion}), which can be phrased as $\min_{x\in \X}\max_{y\in \Y}\{f(x,y)+1/(2\eta_x)\dist^2(\tilde{x},x)\}$.
By \cref{thm:sion} we can exchange the min and the max, resulting in $\max_{y\in \Y} \min_{x\in \X} \{f(x,y)+\frac{1}{2\eta_x}\dist^2(\tilde{x},x) \}$.
\begin{algorithm}
    \caption{Riemannian Accelerated Min-Max Algorithm RAMMA$(f, (x_0, y_0), \epsilon, \X\times \Y)$}
    \label{alg:ramma}
\begin{algorithmic}[1]
    \REQUIRE Sets $\X \subset \M$, $\Y \subset\NN$ that are g-convex in Hadamard manifolds $\M$ and $\NN$, $(\mux, \muy)$-strongly g-convex $(\Lx, \Ly, \Lxy)$-smooth function $f:\M\times\NN\rightarrow\mathbb{R}$, initial point $(x_0,y_0) \in \X\times\Y$, accuracy $\epsilon$. Define $\xi \defi 4\max\{\zeta[2D][\X], \zeta[2D][\Y]\} -3 = O(\zetad)$. For $T_i$, $\hat{\oldepsilon}_i$, see \cref{table:ramma-params}.
    \vspace{0.1cm}
    \hrule
    \vspace{0.1cm}
      \State $\eta_x\gets(9\xi\mux+\max\{\Lxy, \mux\})^{-1}$, $\eta_y\gets(9\xi\muy+\max\{\Lxy,\muy \})^{-1}$\label{line:loop1-start}
      \State $\lambda_y\gets(\max\{\Ly,\Lxy\}+9\xi\muy)^{-1}$
    \State $\hat{x} \gets  \riemacon(\phi(x)\defi \max_{y\in\Y} f(x, y), x_0, \TOne, \eta_x, \X,  \text{ Lines \ref{line:first_line_intermediate_loop}-\ref{line:last_line_intermediate_loop})}$ \label{line:riemacon_on_phi}
    \State $\hat{y} \gets  \riemacon(y\mapsto -f(\hat{x}, y), y_0, \Ttwo, \lambda_{y}, \Y, \text{\PRGD{}})$ \Comment{One-Gap-to-Dist}\label{line:getting_distance_to_yast_outer_loop}
    \State \return{} $\hat{x}, \hat{y}$ \label{line:loop1-end}
    \vspace{0.1cm}
    \hrule
    \vspace{0.1cm}
    \hspace{-1.4cm}Subroutine for Line \ref{line:riemacon_on_phi}: With accuracy $\hatepsilonone$, solve $\min_{x\in\X}\{\phi(x) + \frac{1}{2\eta_x} \dist^2(x_k, x)\}$ for some $x_k \in \X$.
    \State $\eta_y\gets(9\xi\muy+\max\{\Lxy,\muy \})^{-1}$, $\widehat{\lambda} \gets 1/(9\xi(\mux+\eta_x^{-1})+\Lx+\zetad\eta_x^{-1})$ \label{line:first_line_intermediate_loop}
    \State $\tilde{y}_k \gets \riemacon(\psi(y) \defi \max_{x\in\X}\{-f(x,y) - \frac{1}{2\eta_x}\dist^2(x_k, x)\}, y_0, \Tthree, \eta_y, \Y, \text{ Lines \ref{line:rabr_innermost_loop}-\ref{line:loop3-end})}$\label{line:riemacon_on_psi}%
    \State $\tilde{x}_k \gets  \riemacon(x\mapsto f(x, \tilde{y}_k)+\frac{1}{2\eta_x}\dist^2(x_k, x), x_0, \Tfour, \X, \widehat{\lambda}, \text{\PRGD{}})$ \Comment{One-Gap-to-Dist} \label{line:getting_distance_to_xkast_second_loop} 
    \State \return{} $\tilde{x}_k$ \label{line:loop2-end} \label{line:last_line_intermediate_loop}
    \vspace{0.1cm}
    \hrule
    \vspace{0.1cm}
    \hspace{-1.4cm}Subroutine for Line \ref{line:riemacon_on_psi}: With accuracy $\hatepsilonthree$, solve $\min_{y\in\Y}\{\psi(y) + \frac{1}{2\eta_y} \dist^2(y_\ell, y)\}$ for some $y_\ell\in\Y$.
    \State $\bar{x}_\ell, \bar{y}_\ell \gets \text{\RABR{}}(f(x, y) + \frac{1}{2\eta_x}\dist^2(x_k,x) - \frac{1}{2\eta_y} \dist^2(y_\ell, y), (x_0, y_0), \Tfive, \X\times\Y)$\label{line:rabr_innermost_loop}
    \State \return{} $\bar{y}_\ell$\label{line:loop3-end}
\end{algorithmic}
\end{algorithm}
Let $\psi(y) \defi-\max_{x\in\X}\{-f(x,y) - \frac{1}{2\eta_x}\dist^2(\tilde x, x)\}$, then by \cref{lem:strong_cvxty_of_max} we can interpret the latter min-max problem above as the $\muy$-strongly g-convex problem $\min_{y\in \Y}\psi(y)$. We note that an approximate solution to $\min_{y\in\Y}\psi(y)$ does not directly yield an approximate solution for the min-max problem, 
but we obtain the latter after some extra algorithmic steps, as we detail in \cref{proposition:from_one_opti_measure_to_another}.
We can solve this minimization problem via \cref{alg:riemacon_sc_absolute_criterion} again, which involves solving a proximal step $\min_{y\in \mathcal{Y}}\{\psi(y)+1/(2\eta_y)\dist^2(y,\tilde{y})\}$ at each iteration which can be phrased as the min-max problem $\min_{x\in \X}\max_{y\in \Y} \{ f(x,y)+1/(2\eta_x)\dist^2(\tilde{x},x) -1/(2\eta_y)\dist^2(y,\tilde{y}) \}$.
By choosing $\eta_{x}$ and $\eta_y$ such that $\Lxy\le (4\eta_x\eta_y)^{-1/2}$, we ensure that $x$ and $y$ have weak interaction for this last regularized min-max problem.
Hence, we reduce the original min-max problem to a series of min-max problems with weak interaction, which we can solve efficiently using \RABR{}.

The convergence guarantee of \cref{alg:riemacon_sc_absolute_criterion} holds for any proximal parameter $\lambda>0$.
However, when taking into account the computational cost of computing the proximal steps, the right choice of $\lambda$ becomes crucial in order to ensure a good overall complexity.
For \RAMMA{}, we exploit the specific structure of the min-max problem by choosing the proximal parameters small enough such that the inner proximal problem is strongly decoupled while being large enough so that the overall computational complexity achieves the accelerated rate.
Overall, the convergence rates of \RAMMA{} are the following.

\begin{theorem}[Convergence rates of \RAMMA{}]\linktoproof{thm:minmax_alg}\label{thm:minmax_alg}
    Consider a function $f:\M\times \NN\rightarrow \mathbb{R}$ as defined in \cref{sec:setting} with $\mux, \muy > 0$ and let $\M$ and $\NN$ be Hadamard manifolds. 
  Then, \cref{alg:ramma} obtains an $\epsilon$-saddle point after the following number of calls to the gradient and projection oracles:
  \begin{equation*}
    \bigotildel{\zetad^{4.5}\sqrt{\frac{\Lx}{\mux}+\frac{\Ly}{\muy} + \frac{\zetad \L \Lxy}{\mux\muy}+ \zetad^2}}.
  \end{equation*}
\end{theorem}
The central problem of achieving accelerated rates for \RAMMA{} is to show that the iterates stay in a prespecified compact set in order to bound the geometric constant $\zetad$. We design an algorithm that enforces constraints, as opposed to guaranteeing that the iterates of the algorithm naturally stay in a set, as we did for \RCEG{} in \cref{sec:extra_gradient}.

Recall that by definition of $\phi(x)$, the subproblem $\min_{x\in \X} \{\phi(x)+1/(2\eta_x)\dist^2(x_k, x)\}$ can be phrased as
    $$
    \min_{x\in \mathcal{X}}\max_{y\in \mathcal{Y}} \{f(x,y)+1/(2\eta_x)\dist^2(x,x_k)\}.
    $$
Let $(x_{k+1}^{*}, y^{*}(x_{k+1}^{*}))$  be the saddle point of this min-max problem (\cref{thm:sion}), and note that for $\xast$, the best response $y^\ast(\xast)$ for this regularized min-max problem is still $\yast$.
Intuitively, if we do not constrain the min-max problem and the subproblems to $\mathcal{X}\times \mathcal{Y}$, our iterates could get far from our initial point, and this would make us incur greater geometric penalties.
Denote by $\xastg$ and $\yastg$ the unconstrained optimizers.
Then, if we considered an unconstrained version of our algorithm, the point $\tilde{y}_k$ computed in Line \ref{line:riemacon_on_psi}, would be close to the best response $\hat{y}^\ast(\hat{x}_{k+1}^\ast)$  and for this point we can only guarantee that its distance to $\hat{y}^\ast$ is bounded as $\dist(\hat{y}^{*}(\hat{x}^{*}_{k+1}),\hat{y}^{*}(\xastg))\le(\Lxy/\muy)\dist(\xastg,\hat{x}^{*}_{k+1})$, by  using \cref{item:y_opt_lip} of \cref{lem:lip}.
This would preclude acceleration because of the added extra polynomial dependency of $\Lxy/\muy$ on the convergence rates via $\zetad$, which grows with the distances between the iterates and the minimizer. For this reason, we constrain the algorithm. In order to implement a constrained algorithm, we require a linearly convergent subroutine for strongly g-convex and \emph{constrained} problems, which did not previously exist. To this end, we use our \PRGD{} algorithm, cf. \cref{sec:riemannian-g-c}.

Another difficulty is that in order to apply the steps explained above, we require an accelerated algorithm for strongly g-convex optimization that bounds geometric penalties in our setting. Prior work relies on relative-accuracy proximal solutions which, for our setting, would rely on unknown quantities. Because of this reason, we designed \cref{alg:riemacon_sc_absolute_criterion} that accesses $f$ by solving proximal subproblems with absolute accuracy.

Lastly, there is a mismatch between the optimality criterion required for the proximal problems and the optimality criterion provided by the guarantees of the subroutines we used to solve them.
For example, Line \ref{line:first_line_intermediate_loop} requires computing an approximate minimizer $\tilde{x}_{k}$ of $\min_{x\in \mathcal{X}}\{\phi(x)+1/(2\eta_x)\dist^2(x_k,x)\}$, but after Line \ref{line:riemacon_on_psi} we only obtain an approximate minimizer $\tilde{y}_k$ of $\psi(y)$. In the next subsection, we present \cref{proposition:from_one_opti_measure_to_another} which allows to solve this problem.

\subsection{Converting Between Different Optimality Criteria}

We now give a brief overview of the different optimality criteria in g-convex and min-max problems, before we formalize how to guarantee one criterion from another, in  \cref{proposition:from_one_opti_measure_to_another}.
In $\bar{L}$-smooth and $\bar{\oldmu}$-strongly g-convex optimization there are, among others, two well-known measures of optimality: The primal gap $\bargap{x} \defi g(x) - g(x^\ast)$ and the squared distance to the solution $\dist^2(x, x^\ast)$, where $x^\ast \defi \argmin g(x)$ is the unique minimizer of $g$. By strong convexity, one can show that if a point $x$ has a small gap, then it is also close in distance to the solution, up to some  function parameters: $\dist^2(x, x^\ast) \leq \frac{2}{\bar{\oldmu}} (g(x)-g(x^\ast))$. In unconstrained optimization, a converse statement also holds, since $g(x)-g(x^\ast) \leq \frac{\Lsmooth}{2}\dist^2(x, x^\ast)$. Analogously to the Euclidean space, in optimization constrained to a g-convex closed set $\XX$ a similar result can be obtained after a relatively cheap algorithmic computation. For the point $x' \defi \proj[\XX](x-\frac{1}{\Lsmooth}\nabla f(x))$ defined as the result of one step of projected gradient descent, one can show that $\bargap{x'}\leq \frac{\zeta[R]\Lsmooth}{2}\dist^2(x, x^\ast)$, where now the gap is defined with respect to the constrained optimum $g(x)-\min_{x\in\X} g(x)$ and $R$ is defined as in \cref{lem:prgd}.

In the optimization of smooth and \SCSC{} functions $g:\X\times\Y \to \R$, where $\X, \Y$ are g-convex closed sets, we have other notions of optimality.
Also, define the functions $x^\ast(y) \defi \argmin_{x\in \X} g(x, y)$ and $y^\ast(x) \defi \argmax_{y\in\X}, g(x,y)$. We have the following notions of optimality:
\begin{itemize}
\item Duality gap $\newtarget{def:gap}{\gap[\bar{x}, \bar{y}]}  \defi \max_{y\in\Y} g(\bar{x}, y) - \min_{x\in\X} g(x, \bar{y}) = g(\bar{x}, y^{\ast}(\bar{x})) - g(x^{\ast}(\bar{y}), \bar{y})$;
\item Squared distance to the solution $\dist^2(\bar{x},x^\ast) + \dist^2(\bar{y}, y^\ast)$; 
\item Gap in each of the variables:
\[
    \newtarget{def:gap_x}{\gapx[\bar{x}]} \defi g(\bar{x}, y^\ast(\bar{x})) -g(x^\ast, y^\ast),\quad \text{ and }\quad \newtarget{def:gap_y}{\gapy[\bar{y}]} \defi  g(x^\ast, y^\ast)-g(y^\ast(\bar{y}), \bar{y}).
\]
\end{itemize}
We note that $\gap[\bar{x}, \bar{y}] = \gapx[\bar{x}] + \gapy[\bar{y}]$ and by optimality of the points involved, all gaps are non-negative. The following lemma provides how one can relate these measures. 

\begin{lemma}\linktoproof{proposition:from_one_opti_measure_to_another}\label{proposition:from_one_opti_measure_to_another}
    Let $\X\subset \M, \Y\subset\NN$ be closed g-convex subsets of the Hadamard manifolds $\M$, $\NN$, respectively. Let $g:\M\times\NN\to\R$ satisfy \cref{ass:fct} in $\X\times\Y$. The following holds:
    \begin{enumerate}
        \item\textbf{(Full Gap to One Gap)} $\gapx[\bar{x}]\leq \gap[\bar{x}, \bar{y}]$ and , $\gapy[\bar{y}] \leq \gap[\bar{x}, \bar{y}]$. \label{item:full_gap_to_individual_gap}
        \item\textbf{(One Gap to One Dist)} $\dist^2(\bar{x}, x^\ast) \leq \frac{2}{\bar{\oldmu}_x}\gapx[\bar{x}]$, and $\dist^2(\bar{y}, y^\ast) \leq \frac{2}{\bar{\oldmu}_y}\gapy[\bar{y}]$. \label{item:gap_to_dist}
        \item \textbf{(One Gap to Dist - One Variable Optimization)}\label{item:one_gap_to_dist} Suppose $\gapy[\bar{y}] \leq \epsilon$. If we compute an $\hat{\oldepsilon}$-minimizer $\bar{x}'$ of the problem $\min_{x\in\X} g(x,\bar{y})$, then
\[
            \dist^2(\bar{x}', x^\ast) + \dist^2(\bar{y}, y^\ast) \leq \frac{4\hat{\oldepsilon}}{\bar{\oldmu}_x}+ \frac{2\epsilon}{\bar{\oldmu}_y}\left( \frac{2\bar{L}_{xy}^2}{\bar{\oldmu}_x^2}+1\right).
\]

        \item \textbf{(Dist to Gap)}\label{item:dist_to_gap_lip} 
            If in addition, $x\mapsto g(x,\bar{y})$ is $\bar{L}_p^x$-Lipschitz in $\X$ and $y\mapsto g(\bar{x},y)$ is $\bar{L}_p^y$-Lipschitz in $\Y$, we have that
            \begin{equation*}
                \text{gap}(\bar{x},\bar{y})\le  \dist(y^{*},\bar{y})\left(\bar{L}_p^y+\bar{L}_p^x\frac{\bar{L}_{xy}}{\bar{\oldmu}_x}\right) +\dist(x^{*},\bar{x})\left(\bar{L}_p^x+\bar{L}_p^y\frac{\bar{L}_{xy}}{\bar{\oldmu}_y}\right).
            \end{equation*}
    \end{enumerate}
\end{lemma}

Above, we showed that we can essentially guarantee any optimality criterion with another, by only increasing the accuracy by low polynomial factors depending on the problem parameters, which translates into logarithmic factors when applied to a method like our \cref{alg:ramma}. The most expensive reduction consists of going from having a low $\gapx[x]$ or $\gapy[y]$ to bounding the other measures \cref{proposition:from_one_opti_measure_to_another}.\ref{item:one_gap_to_dist}, for which we require running an accelerated method on one variable.  
This is done in Lines \ref{line:getting_distance_to_yast_outer_loop} and \ref{line:getting_distance_to_xkast_second_loop} of \cref{alg:ramma}.
The complexity of the main routine in \cref{alg:ramma} still dominates these extra algorithmic steps. 

\subsection[The NCSC, CC and SCC Cases]{The \NCSC{}, \CC{} and \SCC{} Cases}
By means of regularization, we can reduce the \CC{} and \SCC{} cases to the \SCSC{} case and use \cref{alg:ramma} to solve such problems. Interestingly, we require regularization in both variables even if the function is strongly g-convex with respect to one of them, because regularizing in both variables guarantees $\dist((x_0, y_0), (\hat{x}_{\varepsilon}^\ast, \hat{y}_{\varepsilon}^\ast)) \leq \dist((x_0, y_0), (\xast, \yast))$ and this is a crucial property in our analysis to reduce geometric penalties, see \cref{rem:SP-ass}.
We used $(\hat{x}_{\varepsilon}^\ast, \hat{y}_{\varepsilon}^\ast)$ to denote the global saddle point of the regularized problem.

\begin{corollary}[\SCC{} or \CC{} to \SCSC{}]\linktoproof{cor:red-scc}\label{cor:red-scc}
    Let $f:\M\times \NN\rightarrow \mathbb{R}$ be a function as defined in \cref{sec:setting} and let $\M$ and $\NN$ be Hadamard manifolds. Via regularization, \cref{alg:ramma} obtains an $\epsilon$-saddle point of $f$ after the following number of calls to the gradient and metric-projected oracles, in the $(\mux,0)$-\SCSC{} case:
  \begin{equation*}
    \bigotildel{\zetad^{9/2}\sqrt{\zetad^2\frac{\L}{\mux}+ \frac{\Ly\D^2}{\epsilon} + \frac{\zetad \D^2\Lxy\L}{\epsilon\mux} }} =\bigotildel{\zetad^{11/2}\frac{LD}{\sqrt{\mux}\epsilon}}.
   \end{equation*}
    Similarly, if the function is $(0,0)$-\SCSC{}, i.e., it is \CC{}, via regularization, \cref{alg:ramma} takes
  \begin{equation*}
      \bigotildel{\zetad^{9/2}\sqrt{\frac{\Lx\D^2}{\epsilon}+\frac{\Ly\D^2}{\epsilon}+\frac{\zetad \Lxy\D^2}{\epsilon}\left(\frac{\L\D^2}{\epsilon}+\zetad \right)}} = \bigotildel{\sqrt{ \frac{\zetad^{9}\L\D^2}{\epsilon}}+\frac{\D^2\sqrt{\zetad^{11}\Lxy\L}}{\epsilon}}.
   \end{equation*}
\end{corollary}

We note that similarly to the Euclidean case, if $\Lxy = 0$ we naturally recover the accelerated convergence of optimizing one problem on each variable separately, up to geometric constants and log factors. If $\Lxy=\Lx=\Ly$, we obtain that, up to geometric penalties and log factors, the convergence rates result in the product of the accelerated rates for each problem individually. 

Further, using \RAMMAWC, a modification of \RAMMA{} described in \cref{def:ramma-wc} we can find an $\varepsilon$-stationary point of $f$.
\begin{theorem} \linktoproof{thm:ramma-wc}\label{thm:ramma-wc}
  Let $\X \subset \M, \Y \subset \NN$ be Hadamard manifolds.
  Let $f: \mathcal{X} \times \mathcal{Y} \rightarrow \mathbb{R}$ be $\mu_y$-SC in $y$ and let $f$ be $(\Lx, \Ly, \Lxy)$-smooth.
    Then the output of \RAMMAWC{}  is an $\varepsilon$-stationary point (see \cref{def:min-stat}) of $f$  with probability at least 2/3 after $\bigotildel{\zeta^{4}\frac{\L\Deltazero}{\epsilon^2}\sqrt{\zeta +\frac{\L}{\muy}}}$ calls to the gradient and projection oracle, where $\newtarget{def:initial_gap_of_phi}{\Deltazero} \defi \phi(x_0)-\min_{x\in \mathcal{X}}\phi(x)$, $\phi(x)\defi \max_{y\in\Y} f(x, y)$.
\end{theorem}

\section{Conclusion and Future Work}
\label{sec:concl-future-work}
This paper contributes to the understanding of Riemannian optimization in multiple directions.
We proposed several new and improved methods for g-convex optimization, as well as for min-max Riemannian optimization.
None of our algorithmic results assume that the iterates will stay in some pre-specified bounded set, and consequently we can ensure we bound geometric penalties.
An interesting future direction of research is whether our algorithms can avoid extra logarithmic factors or enjoy lower geometric penalties.
Finding lower bounds for the min-max Riemannian case that feature additional hardness caused by the geometry is of interest.
An important open question, even for the Euclidean setting, is achieving rates for the general $(\Lx,\Ly,\Lxy)$-smooth and $(\mux,\muy)$-\SCSC{} case matching lower bounds.
\clearpage

\acks{
This research was partially funded by the Research Campus Modal funded by the German Federal Ministry of Education and Research (fund numbers 05M14ZAM,05M20ZBM) as well as the Deutsche Forschungsgemeinschaft (DFG) through the DFG Cluster of Excellence MATH$^+$ (EXC-2046/1, project ID 390685689).
}

\printbibliography[heading=bibintoc] 

\clearpage

\appendix

\section{Geometric Auxiliary Results}

In this appendix, we use the following abuse of notation. Given points $x, y, z\in \M$, we write $y$ to mean $\exponinv{x}(y)$, if it is clear from context. For example, for $v\in \Tansp{x}\M$ we have $\innp{v, y - x} = -\innp{v, x - y} = \innp{v, \exponinv{x}(y)- \exponinv{x}(x)} = \innp{v, \exponinv{x}(y)}$; $\norm{v-y} = \norm{v-\exponinv{x}(y)}$; $\norm{z-y}_x = \norm{\exponinv{x}(z)-\exponinv{x}(y)}$; and $\norm{y-x}_x =\norm{\exponinv{x}(y)} =\dist(x, y)$. 

In this section, we provide already established useful geometric results that will be used in our proofs in the sequel. 

\begin{lemma}[Riemannian Cosine-Law Inequalities]\label{lemma:cosine_law_riemannian}
    For the vertices $x, y, p \in \M$ of a uniquely geodesic triangle of diameter $D$, we have
    \[
        \innp{\exponinv{x}(y), \exponinv{x}(p)} \geq \frac{\delta[D]}{2} \dist^2(x,y) + \frac{1}{2}\dist^2(p, x) - \frac{1}{2}\dist^2(p, y).
    \]
    and
    \[
        \innp{\exponinv{x}(y), \exponinv{x}(p)} \leq \frac{\zeta[D]}{2} \dist^2(x,y) + \frac{1}{2}\dist^2(p, x) - \frac{1}{2}\dist^2(p, y)
    \]
\end{lemma}
See \citep{martinez2022accelerated} for a proof.

\begin{remark}\label{remark:tighter_cosine_inequality}
    Actually, in spaces with lower bounded sectional curvature, if we substitute the constants $\zeta[D]$ in the previous \cref{lemma:cosine_law_riemannian} by the tighter constant and $\zeta[\dist(p,x)]$, the result also holds. See \citep{zhang2016first}.
\end{remark}

The following lemmas allow us to bound functions defined in some tangent space by other functions defined in another tangent space. See \citet{kim2022accelerated} for a proof.

\begin{lemma}\label{lemma:moving_quadratics:inexact_approachment_recession}
    Let $x, y, p \in \M$ be the vertices of a uniquely geodesic triangle $\mathcal{T}$ of diameter $D$, and let $z^x \in \Tansp{x}\M$, $z^y \defi \Gamma{x}{y}(z^x) + \exponinv{y}(x)$, such that $y = \expon{x}(rz^x)$ for some $r\in[0,1)$. If we take vectors $a^y\in \Tansp{y}\M$, $a^x \defi \Gamma{y}{x}(a^y) \in \Tansp{x}\M$, then we have the following, for all $\xi \geq \zeta[D]$:
\begin{align*}
 \begin{aligned}
     \norm{&z^y+a^y-\exponinv{y}(p)}_y^2 + (\xi-1)\norm{z^y+a^y}_y^2 \\
     &\leq \norm{z^x+a^x-\exponinv{x}(p)}_x^2 + (\xi-1)\norm{z^x+a^x}_x^2 + \frac{\xi-\delta[D]}{2}\left(\frac{r}{1-r}\right)\norm{a^x}_x^2.
 \end{aligned}
\end{align*}
\end{lemma}

\begin{corollary}\label{corol:moving_quadratics:exact_approachment_recession}
    Let $x, y, p \in \M$ be the vertices of a uniquely geodesic triangle of diameter $D$, and let $z^x \in \Tansp{x}\M$, $z^y \defi \Gamma{x}{y}(z^x) + \exponinv{y}(x)$, such that $y = \expon{x}(rz^x)$ for some $r\in[0,1]$. Then, the following holds
    \[
         \norm{z^y-\exponinv{y}(p)}^2 + (\zeta[D]-1)\norm{z^y}^2 \leq \norm{z^x-\exponinv{x}(p)}^2 + (\zeta[D]-1)\norm{z^x}^2.
    \]
\end{corollary}
The case $r=1$ above is obtained by taking the limit $r\to 1$.

\begin{lemma}\label{fact:hessian_of_riemannian_squared_distance}
    Let $\MOnly$ be a Riemannian manifold of sectional curvature bounded by $[\kmin, \kmax]$ that contains a uniquely g-convex set $\X\subset \MOnly$ of diameter $\D<\infty$. Then, given $x,y\in\X$ we have the following for the function $\Phi_x:\M \to\R$, $y \mapsto \frac{1}{2}\dist^2(x, y)$:
\[
    \nabla \Phi_x(y) = -\exponinv{y}(x)\text{\quad \quad and \quad \quad} \delta[\D]\norm{v}^2 \leq \Hess  \Phi_x(y)[v, v] \leq \zeta[\D] \norm{v}^2.
\] 
These bounds are tight for spaces of constant sectional curvature. Consequently, $\Phi_x$ is $\delta[\D]$-strongly g-convex and $\zeta[\D]$-smooth in $\X$. 
\end{lemma}
See \citet{kim2022accelerated} for a proof, for instance.  

\section[Rescaling the Metric to Obtain \texorpdfstring{$L_x = L_y$}{Lx = Ly}]{Rescaling the metric to obtain $\Lx = \Ly$}\label{sec:rescaling_manifolds}

    If given $(\M, \metric)$ we rescale the metric $\metric$ by a factor $c^2 \in \R_{>0}$, we obtain that any distance $D$ is scaled by $c$. That is, if we consider $(\M, \tilde{\metric})$, where $\tilde{\metric} = c^2 \metric$, then if we denote the distance induced by $\tilde{\metric}$ by $\dist_{\tilde{\metric}}(\cdot)$, we have $d_{\tilde{\metric}}(x, y) = c \dist(x, y)$ for all $x,y\in\M$. And similarly, if the bounds on the sectional curvature of $(\M, \metric)$ are $[\kmin, \kmax]$, we obtain that the bounds on the sectional curvature for $(\M, \tilde{\metric})$ are $[\kmint, \kmaxt] = \frac{1}{c^2}[\kmin, \kmax]$. Given a geodesically convex set $\X$, the geometric constants $\zeta[][\X]$ and $\delta[][\X]$ remain invariant under this tranformation. Indeed, let $\X$ be a set of diameter $D$ measured with $\dist(\cdot)$ and $\tilde{D}$ measured with $d_{\tilde{\metric}}(\cdot)$. Then, if $\kmin < 0$ we have $\zeta[][\X] = D\sqrt{\abs{\kmin}}\coth(D\sqrt{\abs{\kmin}}) = \tilde{D}\sqrt{\abs{\kmint}}\coth(\tilde{D}\sqrt{\abs{\kmint}})$. For $\kmin > 0$ it is $\zeta[][\X]=1$ in both cases. Similarly, we obtain the result for $\delta[][\X]$. For a function $f:\M\to\R$, we have that $L$-smoothness and $\mu$-strong convexity under $\metric$ transforms into $\tilde{L}$-smoothness and $\tilde{\oldmu}$-strong convexity under $\tilde{\metric}$, for $\tilde{L} = L/c^2$ and $\tilde{\oldmu} = \mu/c^2$ since by definition:
    \[
        \f(y) \leq \f(x) + \innp{\nabla \f(x), \exponinv{x}(y)} + \frac{\L}{2}\dist^2(x,y)  = \f(x) + \innp{\nabla \f(x), \exponinv{x}(y)} + \frac{\L}{2c^2}\dist_{\metrict}(x,y)^2,
    \] 
    and analogously for $\mu$-strong convexity. In particular, the condition number $\tilde{L}/\tilde{\oldmu} = L/\mu$ remains constant, and for any two points $x, y\in\M$ we have $\tilde{L} d_{\metrict}(x, y)^2 = L \dist^2(x, y)$. Now, if we have a function $f:\M\times\NN\to\R$ defined as in \cref{sec:setting}, and we rescale the metric of $\M$ by $c_1^2\defi(\Lx/\Ly)^{1/2}$ and rescale the metric of $\NN$ by  $c_2^2\defi(\Ly/\Lx)^{1/2}$, then we have $\tilde{L}_x = \sqrt{\Lx\Ly} = \tilde{L}_y$ and $\tilde{L}_x/\tilde{\mu}_x = \Lx/\mux$ as well as $\tilde{L}_y/\tilde{\mu}_y = \Ly/\muy$, and $\tilde{\mu}_x\tilde{\mu}_y = \mux(\Lx/\Ly)^{-1/2}\muy(\Ly/\Lx)^{-1/2} = \mux\muy$. Finally $\tilde{L}_{xy} = \Lxy$, since
\begin{align*}
\begin{aligned}
    \norm{\nabla_x f(x, y)- \nabla_x f(x, y')}_{x, \metrict} &= \frac{1}{c_1}\norm{\nabla_x f(x, y)- \nabla_x f(x, y')}_{x}  \\
        &\leq \frac{1}{c_1}\Lxy \dist(y, y') = \frac{1}{c_1c_2}\Lxy d_{\metrict}(y, y')  \\
        &= \Lxy d_{\metrict}(y, y').
\end{aligned}
\end{align*}
So indeed one can assume without loss of generality that for one such function $f$, we have $\Lx = \Ly$.

\section[Proofs of Generalized Riemannian Sion's Theorem ]{Proofs of Generalized Riemannian Sion's Theorem}

\begin{proof}\linkofproof{thm:sion}
  Let $\X_{\sup}=\{ x\in \X\ \vert\ \sup_{y\in \Y}f(x,y)\le \alpha\}$ and $\X_{\cap}={\bigcap_{y\in \Y} \{ x\in \X\ \vert\ f(x,y)\le \alpha\}}$.
Note that $\X_{\cap}=\{x\in \X\ \vert\ f(x,y)\le \alpha, \forall y\in \Y\}$ and hence $\X_{\cap}=\X_{\sup}$ by the definition of the supremum.
  We have that $\X_{\cap}\subset \{x\in \X\ \vert\ f(x,y_1)\le \alpha\}$.
  For every $y\in \Y$, $\{x\in \X \ \vert\  f(x,y)\le \alpha\}$ is closed because $f(\cdot,y)$ is lower semicontinuous for all $y \in \Y$ and by the inf-compactness of $f(\cdot,y_1)$, we have $\{x\in \X \ \vert\  f(x,y_1)\le \alpha\}$ is compact.
  It follows that $\X_{\cap}$ and $\X_{\sup}$ are also compact.
  Note that $\X_{\sup}$ is a sublevel set of $\varphi(x)=\sup_{y\in \Y}f(x,y)$.
  Define $M_1= \{ x\in \X\ \vert\  \varphi(x)\le \varphi(x_1)\}$, which is compact because the sublevel sets of $\varphi$ are of the form of $\X_{\sup}$, and it is not empty because $x_1 \in M_1$.
  We can write $\inf_{x\in \X}\varphi(x)=\inf_{x\in M_1}\varphi(x)$ and since $\varphi$ is lower semicontinuous and $M_1$ is compact and non-empty, we have $\min_{x\in \X}\varphi(x)=\min_{x\in \X}\sup_{y\in Y}f(x,y)$.
  One can analogously, show that $\max_{y\in \Y}\inf_{x\in \X}f(x,y)$ exists.
  
  The inequality $\max_{y\in \Y}\inf_{x\in \X}f(x,y)\le \min_{x\in \X}\sup_{y\in Y}f(x,y) $ holds.
  In the following, we show that the reverse inequality also holds.
  Let $\alpha<\min_{x\in \X}\sup_{y\in \Y}f(x,y)$ and $\phi_y(\alpha)\defi \{ x\in \X \ \vert \  f(x,y) \le \alpha\}$.
By definition of $\phi_{y_1}$ and inf-compactness, $\phi_{y_1}(\alpha)$ is g-convex and compact.
We have that the collection of complements
\begin{equation*}
 \phi^C\defi \left\{ \phi^C_y(\alpha)=\{ x\in \X \ \vert\  f(x,y)> \alpha \}  \right\}_{y\in \Y}
\end{equation*}
is an open cover of $\X$.
Assume for the sake of contraction that $\phi^C$ is not an open cover of $\X$.
In this case there exists an $x_0\in \X$ such that $f(x_0,y)\le \alpha$ for all $y\in \Y$.
For this $x_0$, it holds in particular that $\sup_{y\in \Y}f(x_0,y)\le \alpha$.
This contradicts the definition of $\alpha$, since
\begin{equation*}
  \alpha < \min_x\sup_yf(x,y)\le\sup_{y\in \Y}f(x_0,y)\le \alpha
\end{equation*}
cannot hold.

Since $\phi^C$ covers $\X$, it also covers $\phi_{y_1}(\alpha)\subset \X$.
The set $\phi_{y_1}(\alpha)$ is compact, so it has a finite cover $\left\{ \phi_{y_i}^C(\alpha) \right\}_{i=2,\ldots,m}$, and thus $\left\{ \phi_{y_i}^C(\alpha) \right\}_{i=1,\ldots,m}$ is a finite cover for $\X$.
    We have found a set $\Y_1=\{y_1,\ldots,y_m\}$ such that for all $x\in \X$, there exists $\tilde{y}\in \Y_1$, with $x\in \phi_{\tilde{y}}^{C}(\alpha)$.
This implies $\alpha<\max_{\tilde{y}\in \Y_1}f(x,\tilde{y})$ for all $x\in \X$.
Let $\varphi_1(x)=\max_{\tilde{y}\in \Y_1}f(x,\tilde{y})$; then $\widetilde{M}_1=\{x\in \X\ \vert\ \varphi_1(x)\le \varphi_1(x_1)\}$ is compact and non-empty.
    It follows that $\inf_{x \in \X}\varphi_1(x)=\inf_{x\in \widetilde{M}_1}\varphi_1(x)$ and hence $\min_{x \in \X}\varphi_1(x)=\min_{x \in \X}\max_{\tilde{y}\in \Y_1} f(x, y)$ exists.
 Since  $\alpha<\max_{\tilde{y}\in \Y_1}f(x,\tilde{y})$ for all $x\in \X$, it holds in particular for $\alpha< \min_{x\in \X}\max_{\tilde{y}\in \Y_1} f(x,\tilde{y})$.
    By \cref{lem:y0}, there exists a $y_0\in \Y$ with $\alpha< \min_{x\in \X}f(x,y_0)\le \sup_{y\in \Y}\min_{x\in \X}f(x,y)$.
Consider a monotonic increasing sequence $\alpha_k\rightarrow \min_{x\in \X}\sup_{y\in \Y}f(x,y)$, then since $\alpha_k< \sup_{y\in \Y}\min_{x\in \X}f(x,y)$, we have shown the reverse inequality
\begin{equation*}
  \min_{x\in \X}\sup_{y\in \Y}f(x,y)\le  \max_{y\in \Y}\inf_{x\in \X}f(x,y).
\end{equation*}
We have shown that $\min_{x\in \X}\sup_{y\in Y}f(x,y)$ and $\max_{y\in \Y}\inf_{x\in \X}f(x,y)$ exist, i.e., there exists $x_0=\argmin_{x \in \X}\sup_{y\in \Y}f(x,y)$ and $y_0=\argmax_{y\in \Y}\inf_{x\in \X}f(x,y)$ and we can write
\begin{equation*}
  f(x_0,\tilde{y})\le \sup_{y\in \Y}f(x_0,y)=\inf_{x\in \X}f(x,y_0)\le f(\tilde{x},y_0),\quad  \forall(\tilde{x},\tilde{y}) \in \X\times \Y.
\end{equation*}
Setting $\tilde{x}\gets x_{0}$ and $\tilde{y}\gets y_0$ we have that
\begin{equation*}
 f(x_0,y_0)= \min_{x\in \X}f(x,y_0)=\max_{y\in \Y}f(x_0,y).
\end{equation*}
It follows that
\begin{equation*}
  \min_{x\in \X}\max_{y\in \Y}f(x,y)=  \max_{y\in \Y}\min_{x\in \X}f(x,y),
\end{equation*}
which concludes the proof.
\end{proof}

The previous proof was inspired by ideas from three different generalizations of Sion's theorem, namely from \citet{hartung1982extension,komiya1988elementary,zhang2022minimax}. The following lemma, that we used in the proof of \cref{thm:sion} above, appeared in \citep{zhang2022minimax}.
We add the lemma with a proof for completeness.
\begin{lemma}\label{lem:y0}
Let $(\M, d_{\M})$ and $(\NN, d_{\NN})$ be finite-dimensional, unique geodesic metric spaces. Suppose $\X \subseteq \M$, $\Y \subseteq \NN$ are geodesically convex sets.
Let $f: \X \times \Y \rightarrow \mathbb{R}$ be a function such that $f(\cdot, y)$ is geodesically-quasi-convex and lower semi-continuous and $f(x, \cdot)$ is geodesically-quasi-concave and upper semi-continuous.
Then for any finite $k$ points $y_1, \ldots, y_k \in \Y$ and any real number $\alpha<\min _{x \in X} \max _{i \in [k]} f\left(x, y_i\right)$, there exists $y_0 \in \Y$ s.t. $\alpha<\min _{x \in X} f\left(x, y_0\right)$.
\end{lemma}

\begin{proof}
We prove the lemma for two points, and then the general lemma holds by induction. Suppose it does not hold, so assume that for such an $\alpha$, we have $\min _{x \in \X} f(x, y) \leq \alpha$ for any $y \in \Y$.
As a consequence, there is at least a constant $\beta$ such that
\begin{equation}\label{eq:cont}
\min _{x \in \X} f(x, y) \leq \alpha<\beta<\min _{x \in \X} \max \left\{f\left(x, y_1\right), f\left(x, y_2\right)\right\}.
\end{equation}
Consider the geodesic $\gamma:\left[0, d\left(y_1, y_2\right)\right] \rightarrow \Y$ connecting $y_1$ and $y_2$.
For any $t \in [0,d\left(y_1, y_2\right)]$ and corresponding $z=\gamma(t)$ on the geodesic, the level sets $\phi_z(\alpha), \phi_z(\beta)$ are nonempty due to \cref{eq:cont} and closed due to lower semi-continuity of $f$ regarding the first variable.
Since $f$ is geodesic quasi-concave in the second variable, we obtain
\begin{equation*}
f(x, z) \geq \min \left\{f\left(x, y_1\right), f\left(x, y_2\right)\right\}, \quad \forall x \in \X.
\end{equation*}
This is equivalent to say $\phi_z(\alpha) \subseteq \phi_z(\beta) \subseteq \phi_{y_1}(\beta) \cup \phi_{y_2}(\beta)$.
We then argue the intersection $\phi_{y_1}(\beta) \cap \phi_{y_2}(\beta)$ should be empty.
Otherwise, there exists $x \in \X$ such that $\max \left\{f\left(x, y_1\right), f\left(x, y_2\right)\right\} \leq \beta$, contradicting \cref{eq:cont}.
Next, by quasi-convexity, since level set $\phi_z(\beta)$ is geodesic convex for any $z$, it is also connected.
Consider the three facts:
\begin{itemize}
  \item $\phi_z(\alpha) \subseteq \phi_{y_1}(\beta) \cup \phi_{y_2}(\beta)$
  \item $\phi_{y_1}(\beta) \cap \phi_{y_2}(\beta)$ is empty
        \item $\phi_z(\alpha), \phi_{y_1}(\beta)$ and $\phi_{y_2}(\beta)$ are closed (due to lower semi-continuity), connected and convex
\end{itemize}
We claim that either $\phi_z(\alpha) \subseteq \phi_{y_1}(\alpha)$ or $\phi_z(\alpha) \subseteq \phi_{y_2}(\alpha)$ holds for any point $z$ on the geodesic $\gamma$.
Suppose not, then we can always find two points $w_1 \in \phi_{y_1}(\beta), w_2 \in \phi_{y_2}(\beta)$ such that $w_1, w_2 \in$ $\phi_z(\alpha)$.
Since $\phi_z(\alpha)$ is convex, then there is a geodesic $\gamma:[0,1] \rightarrow \X$ in $\phi_z(\alpha)$ connecting $w_1, w_2$.
Therefore $\gamma$ also lies in $\phi_z(\alpha) \subseteq \phi_{y_1}(\beta) \cup \phi_{y_2}(\beta)$.
Because $\phi_{y_1}(\beta) \cap \phi_{y_2}(\beta)$ is empty, $\gamma^{-1}$ induces a partition on $[0,1]$ as $J_1 \cap J_2=\varnothing$ and $J_1 \cup J_2=[0,1]$ where $\gamma\left(J_1\right) \subseteq \phi_{y_1}(\beta), \gamma\left(J_2\right) \subseteq \phi_{y_2}(\beta)$.
Therefore at least one of $J_1, J_2$ is not closed.
Since $\gamma$ is a continuous map, at least one of $\phi_{y_2}(\beta)$ or $\phi_{y_2}(\beta)$ is also not closed, contradicting known conditions.
Since either $\phi_z(\alpha) \subseteq \phi_{y_1}(\alpha)$ or $\phi_z(\alpha) \subseteq \phi_{y_2}(\alpha)$, the two sets below
\begin{align*}
I_1&\defi\left\{t \in[0,1] \mid \phi_{\gamma(t)}(\alpha) \subseteq \phi_{y_1}(\beta)\right\}, \\
I_2&\defi\left\{t \in[0,1] \mid \phi_{\gamma(t)}(\alpha) \subseteq \phi_{y_2}(\beta)\right\}
\end{align*}
form a partition of the interval $[0,1]$.
We prove $I_1$ is closed and nonempty. The latter is obvious since at least $\gamma^{-1}\left(y_1\right) \in I_1$. Now we turn to prove closedness. Let $t_k$ be an infinite sequence in $I_1$ with a limit point of $t$. We consider any $x \in \phi_{\gamma(t)}(\alpha)$. The upper semi-continuity of $f(x, \cdot)$ implies
\begin{equation*}
\limsup _{k \rightarrow \infty} f\left(x, \gamma\left(t_k\right)\right) \leq f(x, \gamma(t)) \leq \alpha<\beta.
\end{equation*}
Therefore, there exists a large enough integer $l$ such that $f\left(x, \gamma\left(t_l\right)\right)<\beta$.
This implies $x \in$ $\phi_{\gamma\left(t_l\right)}(\beta) \subseteq \phi_{y_1}(\beta)$. Therefore for any $x \in \phi_{\gamma(t)}(\alpha)$, $x \in \phi_{y_1}(\beta)$ also holds.
This is equivalent to $\phi_{\gamma(t)}(\alpha) \subseteq \phi_{y_1}(\beta)$.
Hence by the definition of level set, we know the $t \in I_1$ and $I_1$ is then closed.
By a similar argument, $I_2$ is also closed and nonempty. This contradicts the definition of partition and hence proves the lemma.
\end{proof}

\section[Proofs for Riemannian Corrected Extra-Gradient]{Proofs of Riemannian Corrected Extra-Gradient}
\begin{proof}\linkofproof{lem:it-rceg}
    We show that if $\dist(x_{t}, x^\ast)^2 + \dist(y_{t}, y^\ast)^2 \leq \Dcal^2$ then $\dist(w_{t}, x^\ast)^2 + \dist(z_{t}, y^\ast)^2 \leq 4\Dcal^2$ and $\dist(x_{t+1}, x^\ast)^2 + \dist(y_{t+1}, y^\ast)^2 \leq \Dcal^2$. Then, these two latter properties are satisfied for all $t \geq 0$, since $\dist(x_{0}, x^\ast)^2 + \dist(y_{0}, y^\ast)^2 \leq \Dcal^2$. Recall our notation $\zetad \defi \zeta[D]$ and $\deltad \defi \delta[D]$.

    We start by showing that in both the \CC{} and the \SCSC{} cases, the secondary iterates $(w_t,z_t)$ are not far from the saddle point:
\begin{align*}
  \dist^2(w_t,x^{*})+\dist^2(z_t,y^{*})&\le 2 [\dist^2(w_t,x_t)+\dist^2(x_t,x^{*})+\dist^2(z_t,y_t)+\dist^2(y_t,y^{*})]\\
                               &\circled{1}[\le](2+4\etaEG^2L^2)(\dist^2(x_t,x^{*})+\dist^2(y_t,y^{*})) \\
                               &\circled{2}[\le] (2+\frac{\deltad}{\zetad})(\dist^2(x_t,x^{*})+\dist^2(y_t,y^{*}))\\
  &\circled{3}[\le] 4\Dcal^{2}.
\end{align*}
Here, $\circled{1}$ holds since by definition of $w_t$ we have $\dist^2(w_t,x_t)=\norm{\etaEG\nabla_xf(x_t,y_t)}^2\le 2\etaEG^2L^2(\dist^2(x_t,x^{*})+\dist^2(y_t,y^{*}))$ and similarly $ \dist^2(z_t, y_t)\le 2\etaEG^2L^2(\dist^2(x_t,x^{*})+\dist^2(y_t,y^{*}))$.
Further $\circled{2}$ holds because  $\etaEG\le \sqrt{\frac{1}{4L^2\zetad}}$.
    And $\circled{3}$ holds since we have $\frac{\deltad}{\zetad}\le 1$ and by our hypothesis on $x_t, y_t$. We bounded $3\Dcal \leq 4 \Dcal$ for convenience. 
    Now, since $f$ is $\mu$-\SCSC{}, we have that
  \begin{align}\label{eq:bound}
  \begin{aligned}
    f(w_t,y^{*}) -f(x^*,z_t)&= f(w_t,y^*)-f(w_t,z_t)+f(w_t,z_t) -f(x^*,z_t)\\
                            &\circled{1}[\le]+ \langle\nabla_yf(w_t,z_t),\exponinv{z_t}(y^*)\rangle -\frac{\mu}{2}\dist^2(w_t,x^{*})\\
    &\quad-  \langle\nabla_xf(w_t,z_t),\exponinv{w_t}(x^*)\rangle  -\frac{\mu}{2}\dist^2(z_t,y^{*})\\
                            &= \frac{1}{\etaEG}\langle-\etaEG\nabla_xf(w_t,z_t) \pm \exponinv{w_t}(x_t),\exponinv{w_t}(x^*)\rangle -\frac{\mu}{2}\dist^2(w_t,x^{*})\\
    &\quad + \frac{1}{\etaEG}\langle\etaEG\nabla_yf(w_t,z_t)\pm \exponinv{z_t}(y_t),\exponinv{z_t}(y^*)\rangle -\frac{\mu}{2}\dist^2(z_t,y^{*})\\
      &\circled{2}[\le] \frac{1}{\etaEG}\langle \exponinv{w_t}(x_{t+1}),\exponinv{w_t}(x^*)\rangle -\frac{1}{\etaEG}\langle\exponinv{w_t}(x_t),\exponinv{w_t}(x^*)\rangle -\frac{\mu}{2}\dist^2(w_t,x^{*})\\
                            &\quad + \frac{1}{\etaEG}\langle \exponinv{z_t}(y_{t+1}),\exponinv{z_t}(y^*)\rangle -\frac{1}{\etaEG}\langle\exponinv{z_t}(y_{t}),\exponinv{z_t}(y^*)\rangle   -\frac{\mu}{2}\dist^2(z_t,y^{*}),
  \end{aligned}
  \end{align}
    where we used the $\mu$-\SCSC{} property in $\circled{1}$ and the definition of the iterates $x_{t+1}$, $y_{t+1}$ in $\circled{2}$. 
    We now use the Riemannian cosine inequalities  to obtain the inequalities below, cf. \cref{lemma:cosine_law_riemannian}, \citep[Lemma 1]{zhang2016first}. The first two inequalities use the fact that the diameters of the geodesic triangles with vertices $w_t, x_t, x^\ast$ and $z_t, y_t, y^\ast$, respectively, are upper bounded by $4\Dcal \leq D$, because all of those points are in $\mathcal{B} \defi \ball((x^\ast, y^\ast), 2\Dcal)$ and this ball is geodesically convex and uniquely convex and hence it contains the triangles. If $\kmin \geq 0$, we have that $\zeta[c] = 1$ for any $c > 0$, so we can just use \cref{lemma:cosine_law_riemannian} to obtain the last two inequalities. If $\kmin < 0$, we use the more fine-grained inequality \citep[Lemma 1]{zhang2016first}. This lemma establishes the cosine inequalities with constants $\zeta[\dist(w_t, x^\ast)]$ and $\zeta[\dist(z_t, y^\ast)]$, respectively. The inequalities also hold for greater values of these constants, so we use $\zetad$ because we already established that $\dist(w_t, x^\ast), \dist(z_t, y^\ast) \leq 2\Dcal \leq 4\Dcal$.
\begin{equation}\label{eq:geometric-bounds}
  \begin{aligned}
  -2\langle\exponinv{w_t}(x_t),\exponinv{w_t}(x^*)\rangle &\le - \deltad\dist^2(w_t,x_t) - \dist^2(w_t,x^*) + \dist^2(x_t,x^*)\\
  -2\langle\exponinv{z_t}(y_t),\exponinv{z_t}(y^*)\rangle& \le - \deltad\dist^2(z_t,y_t) - \dist^2(z_t,y^*) + \dist^2(y_t,y^*)\\
  2 \langle\exponinv{w_t}(x_{t+1}),\exponinv{w_t}(x^*)\rangle& \le \zetad \dist^2(w_t,x_{t+1}) + \dist^2(w_t,x^*)- \dist^2(x_{t+1},x^*)\\
  2 \langle\exponinv{z_t}(y_{t+1}),\exponinv{z_t}(y^*)\rangle& \le \zetad \dist^2(z_t,y_{t+1}) + \dist^2(z_t,y^*)- \dist^2(y_{t+1},y^*).
  \end{aligned}
\end{equation}
We can further bound the following term using the update rules and gradient Lipschitzness,
\begin{equation}
  \begin{aligned}\label{eq:x-progress}
  \dist^2(w_t,x_{t+1}) &= \norm{\exponinv{w_t}(x_{t+1})}^2 = \norm{ \exponinv{w_t}(x_t) -\etaEG\nabla_x f(w_t,z_t)}^2\\
                   &= \norm{\etaEG\Gamma{x_t}{w_t}\nabla_xf(x_t,y_t)-\etaEG\nabla_xf(w_t,z_t)}^2 \\
                   &= 2\eta^2\left(\norm{\Gamma{x_t}{w_t}\nabla_x f(x_t,y_t)-\nabla_x f(w_t,y_t)}^2 + \norm{\nabla_x f(w_t,y_t)-\nabla_xf(w_t,z_t)}^2 \right) \\
                   &\le 2\etaEG^2L^2(\dist^2(w_t,x_t)+\dist^2(z_t,y_t)).
  \end{aligned}
\end{equation}
Analogously, we obtain
\begin{equation}\label{eq:y-progress}
 \dist^2(z_t,y_{t+1})\le 2\etaEG^2L^2 (\dist^2(w_t,x_t)+\dist^2(z_t,y_t)).
\end{equation}
    Using the triangle inequality and $(a+b)^2 \leq 2a^2 + 2b^2$, we have
\begin{equation}\label{eq:triang}
  \begin{aligned}
    -\frac{\mu}{2}\dist^2(w_t,x^{*})&\le \frac{\mu}{2}\dist^{2}(w_t,x_t) - \frac{\mu}{4}\dist^2(x_t,x^{*}),\\
     -\frac{\mu}{2}\dist^2(z_t,y^{*})&\le \frac{\mu}{2}\dist^{2}(z_t,y_t) - \frac{\mu}{4}\dist^2(y_t,y^{*}).
  \end{aligned}
\end{equation}
So finally, we now bound \cref{eq:bound} using \cref{eq:geometric-bounds} in combination with \cref{eq:x-progress,eq:y-progress,eq:triang}. We will use the following inequality to study the \CC{} and \SCSC{} cases separately.
    \begin{align}\label{eq:lemma_bounding_gap_for_EG}
    \begin{aligned}
  0&\le   f(w_t,y^*) -f(x^*, z_t)\\
  &\le \frac{1}{2\etaEG} [(4\zetad \etaEG^2L^2- \deltad +\mu\etaEG )\dist^2(w_t,x_t) + (1-\frac{\mu\etaEG}{2})\dist^2(x_t,x^*)-\dist^{2}(x_{t+1},x^*)]\\
                     &\quad + \frac{1}{2\etaEG} [(4\zetad \etaEG^2L^2 - \deltad + \mu\etaEG)\dist^2(z_t,y_t) + (1-\frac{\mu\etaEG}{2})\dist^2(y_t,y^*)-\dist^{2}(y_{t+1},y^*)].
    \end{aligned}
\end{align}

\paragraph{Case \CC{}} We have $\mu=0$.
    It follows that for $\etaEG\le \sqrt{\frac{\deltad}{4\zetad \L^{2}}}$, we have by \cref{eq:lemma_bounding_gap_for_EG} that
\begin{equation*}
  \dist^{2}(x_{t+1},x^*)+ \dist^{2}(y_{t+1},y^*) \le \dist^2(y_t,y^*)+\dist^2(x_t,x^*) \leq \Dcal^2.
\end{equation*}
\paragraph{Case \SCSC{}} We have $\mu>0$.
It follows that for $\etaEG\le \min \left\{ \sqrt{\frac{\deltad}{8\L^2 \zetad}},\frac{\deltad}{2\mu} \right\}$, we have by \cref{eq:lemma_bounding_gap_for_EG} that
    \begin{equation*}
  \dist^{2}(x_{t+1},x^*)+ \dist^{2}(y_{t+1},y^*)\le \left( 1-\frac{\mu\etaEG}{2} \right) (\dist^2(y_t,y^*)+\dist^2(x_t,x^*))\le  \dist^2(y_t,y^*)+\dist^2(x_t,x^*) \leq \Dcal^2.
\end{equation*}

    Now to conclude the first statement, recall $\Dcal \defi \dist((x_0, y_0), (\xast, \yast))$ and $\DExGr \geq 4\Dcal$ by definition. Since we just showed that the iterates do not go farther than $2\Dcal \leq \DExGr/2$ to $(\xast,\yast)$, then they stay in the closed ball $\mathcal{B} \defi \ball((\xast, \yast), \DExGr/2)$, whose diameter is $\DExGr$. Therefore, our choice of $\etaEG$ in the pseudocode in \cref{alg:rceg} is a valid one. Note that the knowledge of this set $\mathcal{B}$ is not needed for the algorithm.

    For the second statement, we note that the proofs of the convergence rates stated in the proposition were provided by \citet{jordan2022first} for the \SCSC{} case and by \citet{zhang2022minimax} for the \CC{} case under the following additional assumption: when $\etaEG$ is chosen with respect to some geometric constants $\delta[\DExGr]$ and $\zeta[\DExGr]$, then the iterates do not leave a set of diameter $\DExGr$ that contains $(\xast, \yast)$. We showed that the iterates satisfy this assumption for our choice of learning rate $\etaEG$ and therefore the convergence follows.
    Note that \citet{jordan2022first} proves for the \SCSC{} case that $(x_T,y_T)$ is an $\epsilon'$-saddle point in distance.
    By \cref{item:dist_to_gap_lip} of \cref{proposition:from_one_opti_measure_to_another}, we have that $\gap[x_T,y_T]\le [ \dist(x_T,x^{*})+\dist(y_T,y^{*})] LD\left(1+\L/\mu\right)$.
    This holds, as by assumption $\nabla f(x^{*},y^{*})=0$ and hence $\Lips(f(x,y))\le LD$.
    Hence, we have that $\gap[x_T,y_T]\le 2 LD\left(1+\frac{\L}{\mu}\right)\sqrt{\epsilon'}$.
    Thus after
    \begin{equation*}
     T=\bigotildel{\frac{\L}{\mu}\sqrt{\frac{\zeta[\DExGr]}{\delta[\DExGr]}+\frac{1}{\delta[\DExGr]}}}.
    \end{equation*}
    iterations of \RCEG{}, we obtain a $\epsilon$-saddle point for the \SCSC{} case.

\end{proof}

\begin{algorithm}
    \caption{Riemannian Corrected Extragradient ({\sc rceg})}
    \label{alg:rceg}
\begin{algorithmic}[1]
    \REQUIRE Initialization $(x_0,y_0)$, $f:\M\times\NN\rightarrow\mathbb{R}$, manifolds $\M$ and $\NN$, g-strong convexity constant $\mu$ (for \SCSC{}), smoothness $\L$, bound $\DExGr \geq 4\Dcal = 4(\dist^2(x_0, \xast) + \dist^2(y_0, \yast))^{1/2}$.
    \vspace{0.1cm}
    \hrule
    \vspace{0.1cm}
    \State For \SCSC{}, choose  $\newtarget{def:learning_rate_eta_for_RCEG}{\etaEG}\gets \min \left\{\sqrt{\frac{\delta[\DExGr]}{8 \L^2\zeta[\DExGr]}},\frac{\delta[\DExGr]}{2\mu}\right\}$, for \CC{}, choose $\etaEG\gets \sqrt{\frac{\delta[\DExGr]}{4 \L^2\zeta[\DExGr]}}$
    \State $(w_0,z_0)\gets \left(\exp_{x_0}(-\etaEG\nabla_xf(x_0,y_0)),\exp_{y_0}(\etaEG\nabla_yf(x_0,y_0))\right)$
    \State  $\bar{z}_{0}\gets z_0$, $\bar{w}_{0}\gets w_{0}$
    \State $(x_{1},y_{1})\gets \left(  \exp_{w_0}(-\etaEG\nabla_xf(w_0,z_0)+\exponinv{w_0}(x_0)), \exp_{z_0}(\etaEG\nabla_yf(w_0,z_0)+\exponinv{z_0}(y_0))\right)$
    \FOR {$t = 1 \text{ \textbf{to} } T-1$}
    \State $(w_t,z_t)\gets \left(\exp_{x_t}(-\etaEG\nabla_xf(x_t,y_t)),\exp_{y_t}(\etaEG\nabla_yf(x_t,y_t))\right)$
    \State $(x_{t+1},y_{t+1})\gets \left(  \exp_{w_t}(-\etaEG\nabla_xf(w_t,z_t)+\exponinv{w_t}(x_t)), \exp_{z_t}(\etaEG\nabla_yf(w_t,z_t)+\exponinv{z_t}(y_t))\right)$
    \IF {$\mu=0$}\Comment{Geodesic averaging for \CC{}}
    \State $\left(\bar{w}_{t}, \bar{z}_{t}\right)\gets\left(\exp_{\bar{w}_{t-1}}\left(t^{-1} \log _{\bar{w}_{t-1}}\left(w_{t}\right)\right), \exp_{\bar{z}_{t-1}}\left(t^{-1} \log _{\bar{z}_{t-1}}\left(z_{t}\right)\right)  \right)$
    \ENDIF
    \ENDFOR
    \ENSURE $(x_T,y_T)$ {\bf if} $\mu>0$, {\bf else} $(\bar{w}_{T-1}, \bar{z}_{T-1})$
\end{algorithmic}
\end{algorithm}

\section[Proofs of Convergence Rates for our Accelerated Algorithms]{Proofs of Convergence Rates for our Accelerated Algorithms}

\subsection{Acceleration for G-Convex Functions}
{
\renewcommand\mu{\newlink{def:strong_g_convexity_riemacon_abs}{\oldmu}}

\begin{algorithm}[ht!]
    \caption{$\riemacon(f, x_0, T \text{ or } \epsilon, \lambda, \X,\texttt{subroutine})$. Absolute accuracy criterion.}
    \label{alg:riemacon_sc_absolute_criterion}
\begin{algorithmic}[1] 
    \REQUIRE Finite-dimensional Hadamard manifold $\M$ of bounded sectional curvature, feasible g-convex compact set $\X$ of diameter $D_{\XX}$. Initial point $\newtarget{def:initial_point}{\xInit} \in\X\subset\M$. Function $\f:\M\to\R$ that is $\bar{\oldmu}$-strongly g-convex in $\X$. Parameter $\lambda > 0$.  Final iteration $T$ or accuracy $\epsilon$. If $\epsilon$ is provided, compute the corresponding $T$ and vice versa, cf. \cref{thm:riemaconsc}, \texttt{subroutine} to solve the proximal problems of Lines \ref{line:proximal_prob_initialization_riemacon} and \ref{line:proximal_prob_in_main_loop_riemacon}.
    \vspace{0.1cm}
    \hrule
    \vspace{0.1cm}
    \State $\newtarget{def:extra_geom_penalty_xi}{\xi} \gets 4\zeta[2D_{\XX}] -3$ \Comment{$\xi$ is $\bigo{\zeta[D_{\XX}]}$}
    \State $\newtarget{def:strong_g_convexity_riemacon_abs}{\mu} \gets \min\{\bar{\oldmu}, 1/(9\xi\lambda)\}$
    \State $\kappa \gets 1/(\lambda\mu)$
    \State $\hat{\oldepsilon} \gets \epsilon \cdot (8\sqrt{\xi}\kappa^{3/2})^{-1}$
    \State $\yk[0] \gets \hat{\oldepsilon}$-minimizer of the proximal problem $\min_{y\in\X}\{\f(y) + \frac{1}{2\lambda}\dist^2(\xInit, y)\}$\label{line:proximal_prob_initialization_riemacon}
    
    \State $\zybar[0][0] \gets \zy[0][0] \gets 0 \in \Tansp{\yk[0]}\M$
    \State $\Ak[0] \gets 1$ 
    \vspace{0.1cm}
    \hrule
    \vspace{0.1cm}
    \FOR {$k = 1 \textbf{ to } T$}
        \State $\newtarget{def:integral_of_steps}{\Ak[k]} \gets (1+1/(2\sqrt{\xi\kappa}))^k$
        \State $\newtarget{def:discrete_step}{\ak[k]} \gets \xi(\Ak-\Ak[k-1])$
        \State $\newtarget{def:iterate_x}{\xk} \gets \expon{\yk[k-1]}(\frac{\ak}{\Ak[k-1] + \ak} \zybar[k-1][k-1] + \frac{\Ak[k-1]}{\Ak[k-1] + \ak}\yk[k-1])=  \expon{\yk[k-1]}(\frac{\ak}{\Ak[k-1] + \ak}\zybar[k-1][k-1]) $\Comment{Coupling}  
        \State $\newtarget{def:iterate_y}{\yk} \gets$ $\hat{\oldepsilon}$-minimizer of problem $\min_{y\in\X}\{\f(y) + \frac{1}{2\lambda}\dist^2(\xk, y)\}$ \label{line:subroutine}\Comment{Approximate implicit \RGD{}}\label{line:proximal_prob_in_main_loop_riemacon}
        \State $\newtarget{def:v_k_sup_x}{\vkx[k]} \gets -\exponinv{\xk}(\yk)/\lambda$  \Comment{Approximate subgradient}
        \State $\newtarget{def:iterate_z_x_from_z_y_bar}{\zxbar[k-1]} \gets \exponinv{\xk}(\expon{\yk[k-1]}(\zybar[k-1][k-1]))$ 
        \State $\newtarget{def:iterate_z_x}{\zx[k]} \gets \Ak[k][-1](\Ak[k-1]\zxbar[k-1] + \frac{\ak}{\xi}(-\lambda - \frac{2}{\mu})\vkx)$ \Comment{Mirror Descent step} \label{line:MD_step}
        \State $\newtarget{def:iterate_z_y}{\zy} \gets \Gamma{\xk}{\yk}(\zx) + \exponinv{\yk}(\xk)$ \Comment{Moving the dual point to $\Tansp{\yk}\M$}
        \State $\newtarget{def:iterate_z_y_bar}{\zybar} \gets \Pii{\bar{B}(0, \D_{\XX})}(\zy) \in \Tansp{\yk}\M$ \Comment{Easy projection done so the dual point is not very far} \label{line:dual_projecting_step}
    \ENDFOR
    \State \textbf{return} $\yk[T]$.
\end{algorithmic}
\end{algorithm}

We use $\newtarget{def:Riemacon_absolute_accuracy_SC}{\riemacon}$ to refer to \cref{alg:riemacon_sc_absolute_criterion} for $\mu$-strongly g-convex minimization. We write $\riemacon(f, x_0, T, \X, \texttt{subroutine})$ to specify the output of the algorithm initialized at $x_0$ for optimizing the function $f$ constrained to $\X$, run for $T$ steps and making use of the subroutine \texttt{subroutine}. The subroutine solves a proximal problem approximately, effectively implementing an approximate implicit Riemanian Gradient Descent (\newtarget{def:acronym_riemannian_gradient_descent}{\RGD{}}) step. In the pseudocode, we use the notation $\newtarget{def:euclidean_projection}{\Pii{C}}(p)$ to refer to the Euclidean projection of a point $p$ onto a closed convex set $C$.

\begin{proof}\linkofproof{thm:riemaconsc}
    We note that it is $\xi \defi 4\zeta[2D_{\XX}] -3 \leq 8\zeta[D_{\XX}] -3 =O(\zeta[D_{\XX}])$. Let $\kappa \defi \frac{1}{\lambda \mu}$ and let $c \defi \frac{1}{2\sqrt{\xi\kappa}}$, so that $\Ak = (1+c)^k$ for all $k \geq 1$, and $\ak = \xi((1+c)^k-(1+c)^{k-1}) = \xi c (1+c)^{k-1}$ for all $k \geq 1$.  We also note that in order to satisfy $\sqrt{\xi/\kappa} \leq 1/3$, to be used later, we only use $\mu$-strong g-convexity of $f$, instead of $\bar{\oldmu}$-strong convexity, where $\mu \defi \min\{\bar{\oldmu}, 1/(9\xi\lambda)\}$. We want to show the following is almost a Lyapunov function for our problem
    \[
       \Psi_k \defi \Ak\left(f(\yk) -f(\xast) + \frac{\mu}{4} \norm{\zy-\xast}^2_{\yk} + \frac{\mu(\xi-1)}{4}\norm{\zy}^2_{\yk} \right),
    \] 
    in the sense that we can show 
    \begin{equation}\label{eq:recurrence_almost_lyapunov}
    \Psi_k \leq \Psi_{k-1} + 2(\kappa+1) A_k\hat{\oldepsilon}.
    \end{equation}
    If we show \cref{eq:recurrence_almost_lyapunov}, then we can conclude the theorem since for $T \geq 2\sqrt{\xi\kappa}\log_2(\frac{2\lambda^{-1}\dist^2(\xk[0], \xast)}{\epsilon})$, we would have
\begin{align*}
\begin{aligned}
    f(y_T) - f(x^\ast) &\leq \frac{\Psi_T}{\Ak[T]}\leq \frac{\Psi_0}{\Ak[T]} + 2\hat{\oldepsilon}(\kappa + 1) \frac{\sum_{i=1}^{T} \Ak[i]}{\Ak[T]} \leq  \frac{\Psi_0}{\Ak[T]} + 2\hat{\oldepsilon}(\kappa + 1) \frac{\Ak[T+1]}{c\Ak[T]} \\
    &\circled{1}[\leq] \frac{\Psi_0}{2^{Tc}} + 4\hat{\oldepsilon}\kappa^{3/2}\sqrt{\xi} \circled{2}[\leq] \frac{\epsilon}{2} + \frac{\epsilon}{2} = \epsilon.
\end{aligned}
\end{align*}
    where in $\circled{1}$ we used $c= 1/(2\sqrt{\xi\kappa}) \leq 1/6$ (by $\sqrt{\xi/\kappa} \leq 1/3$ and $\xi \geq 1$) and so $(1+c)^{1/c} > 2$. We also bounded $1 + c \leq 2$. For $\circled{2}$, we used $\hat{\oldepsilon} \defi \epsilon / (8\sqrt{\kappa^3\xi})$ and $\Psi_0 \leq f(\yk[0])-f(\xast) \leq \lambda^{-1}\dist^2(\xk[0], \xast)$ due to \cref{lem:dist_to_gap_by_Proj_RGD} and the definition of $\yk[0]$. In short,  we find an $\epsilon$-minimizer for $T = \bigo{\sqrt{\zetad \kappa} \log(\frac{\lambda^{-1} \dist^2(\xk[0], \xast)}{\epsilon})}$.

    We now focus on proving \cref{eq:recurrence_almost_lyapunov}. We can assume without loss of generality that $\xk=0$. We work in the tangent space of $\xk$ all of the time except when applying \cref{lemma:moving_lower_bounds} that moves lower bounds, so according to our notation, points $\xast$, $\yk[k-1]$, and $\yk$ should be interpreted as $\exponinv{\xk}(\xast)$, $\exponinv{\xk}(\yk[k-1])$ and $\exponinv{\xk}(\yk)$, respectively. We note that our choice of dual point $\zx$ comes from optimizing the regularized lower bound that we have at iteration $k$:
\begin{align}\label{eq:definitoin_of_zx_k}
 \begin{aligned}
     \zx &= \argmin_{x\in T_{\xk}\M}\left\{\left( \frac{\Ak[k-1]\mu}{4}\right)\norm{\zx[k-1]-x}_{\xk}^2 + \frac{\ak}{\xi}\frac{\mu}{4}\norml{\left(1+\frac{2}{\mu\lambda}\right)\yk-\frac{2}{\mu\lambda}\xk-x}^2_{\xk} \right\} \\
        &\circled{1}[=] \frac{\Ak[k-1] \zx[k-1] + \frac{\ak}{\xi}\left[\left(1+\frac{2}{\mu\lambda}\right)\yk - \frac{2}{\lambda\mu}\xk\right]}{\Ak[k-1] + \ak /\xi} \circled{2}[=] \frac{\Ak[k-1] \zx[k-1] - \frac{\ak}{\xi}\left(\lambda+\frac{2}{\mu}\right)\vkx}{\Ak[k-1]  + \ak/ \xi}.
   \end{aligned}
\end{align}
    The equality $\circled{1}$ can be obtained by just taking a derivative and checking when we have global optimality. Equality $\circled{2}$ uses $\xk = 0$ and $\lambda\vkx =\xk-\yk = -\yk$. The definition of $\zx$ as the $\argmin$ above is derived from minimizing a convex combination of the previously computed regularized lower bound, a quadratic with minimizer at $\zx[k-1]$, plus the new bound which we obtain by \cref{lemma:appa}. Indeed, one can check that the second summand is a quadratic that has the same minimizer as the right hand side in \cref{lemma:appa}.
    In order to show \cref{eq:recurrence_almost_lyapunov}, by \cref{lemma:moving_lower_bounds} it is enough to show
\begin{align}\label{eq:moving_geometric_penalties_in_analysis}
 \begin{aligned}
     \Ak&\left(f(\yk) -f(\xast) + \frac{\mu}{4} \norm{\zx-\xast}^2_{\xk} + \frac{\mu(\xi-1)}{4}\norm{\zx}^2_{\xk} - 2(\kappa+1) \hat{\oldepsilon} \right) \\
       &\leq \Ak[k-1]\left(f(\yk[k-1]) -f(\xast) + \frac{\mu}{4} \norm{\zx[k-1]-\xast}^2_{\xk} + \frac{\mu(\xi-1)}{4}\norm{\zx[k-1]}^2_{\xk} \right).
   \end{aligned}
\end{align}
    The following identities involving our parameters will be useful in the sequel
\begin{align}
    \Ak &= \Ak[k-1] + \ak/\xi\label{eq:Ak_def} \\ 
     \Ak[k-1](\xk-\yk[k-1]) &= -\ak(\xk - \zx[k-1]) \text{ (equiv. to) }\yk[k-1] = -\frac{\ak}{\Ak[k-1]}\zx[k-1]\label{eq:yk_1_and_zk_1} \\ 
     \yk &= -\lambda \vkx  \label{eq:yk_and_vkx}
\end{align}
    We regroup the terms in \cref{eq:moving_geometric_penalties_in_analysis} with evaluations of $f$ to the left hand side to yield $\Ak[k-1](f(\yk)-f(\yk[k-1])) + (\ak/\xi) \cdot (f(\yk)-f(\xast))$ and then we apply \cref{lemma:appa} twice to show that it is enough to prove:
\begin{align*}
 \begin{aligned}
     &\Ak[k-1] \innp{\vkx, \yk-\yk[k-1]} - \Ak[k-1]\frac{\mu}{4}\norm{\yk[k-1]-\yk}^2  + \frac{\ak}{\xi}\innp{\vkx, \yk - \xast} - \frac{\ak}{\xi}\frac{\mu}{4} \norm{\xast - \yk}^2 \\
     &\leq  \frac{\mu}{4} \left[ \Ak[k-1](\norm{\zx[k-1] -\xast}^2 + (\xi -1) \norm{\zx[k-1]}^2) - \Ak\left( \norm{\zx-\xast}^2 + (\xi-1)\norm{\zx}^2\right)\right]
   \end{aligned}
\end{align*}
Note that the errors with respect to $\hat{\oldepsilon}$ cancel each other. Now, we will just check that the terms involving $\innp{\xast, \cdot}$ and $\norm{\xast}^2$ cancel each other, given our choice of $\zx$. Indeed, for $\norm{\xast}^2$ we have the following weights on each side:
\[
    -\frac{\ak}{\xi} \frac{\mu}{4} = \frac{\mu}{4} (\Ak[k-1] - \Ak),
\] 
    which holds by \eqref{eq:Ak_def}. Then, on each side of the inequality we have the following that holds by our choice of $\zx$, and \cref{eq:yk_and_vkx}
\[
    \innp{\xast, -\frac{\ak}{\xi}\vkx + \frac{\ak}{\xi}\frac{\mu}{2} \cdot (-\lambda) \vkx } = \innp{\xast, \frac{\mu}{4}\Ak[k-1] \cdot(-2\zx[k-1]) - \frac{\mu}{4}\Ak \cdot(- 2\zx)}.
\] 
We remove those terms involving $\xast$, and we use the properties \eqref{eq:yk_1_and_zk_1}  and \eqref{eq:yk_and_vkx} and the definition of $\zx$ so the only variables left are $\ak$, $\Ak[k-1]$, $\Ak$, $\vkx$ and $\zx[k-1]$:
\begin{align*}
 \begin{aligned}
     &-\lambda\Ak[k-1]\norm{\vkx}^2 + \ak\innp{\vkx, \zx[k-1]} - \frac{\mu}{4}\Ak[k-1]\lambda^2\norm{\vkx}^2  - \frac{\mu}{4}\frac{\ak[k][2]}{\Ak[k-1]}\norm{\zx[k-1]}^2 +2\innp{\vkx, \zx[k-1]}\frac{\lambda\mu\ak}{4} \\ 
     &-\frac{\ak}{\xi}\lambda\norm{\vkx}^2 - \frac{\ak}{\xi}\frac{\mu}{4} \lambda^2\norm{\vkx}^2 \\
     &\leq  \frac{\mu}{4} \Ak[k-1]\xi\norm{\zx[k-1]}^2  - \frac{\mu}{4}\frac{\xi}{\Ak}\left( \Ak[k-1][2]\norm{\zx[k-1]}^2 + \frac{\ak[k][2]}{\xi^2}(\lambda + \frac{2}{\mu})^2\norm{\vkx}^2 -2 \innp{\vkx, \zx[k-1]}\Ak[k-1]\frac{\ak}{\xi}(\lambda + \frac{2}{\mu}) \right)
   \end{aligned}
\end{align*}
The strategy now is to complete squares to make appear a factor proportional to $-\norm{a\vkx + b\zx[k-1]}^2$ on the left hand side, and show that we can prove the inequality without that term, where $a, b\in\R$. We pick $a$ so that $-a^2\norm{\vkx}^2$ is precisely the term involving $\norm{\vkx}^2$ that we have above, if we move all of those to the left hand side. In other words, after completing squares we will just need to prove that the resulting factor multiplying $\norm{\zx[k-1]}^2$ is non-positive. Let's first regroup all the coefficients  with respect to $\norm{\vkx}^2$, $\norm{\zx[k-1]}^2$ and $\innp{\vkx,\zx[k-1]}$ and place them on the left hand side:
\begin{align}\label{eq:aux:in_the_middle_of_lyapunov_proof}
 \begin{aligned}
     &\innp{\vkx, \zx[k-1]}\cdot \ak \left(1+\frac{\lambda\mu}{2} \right) \left(1-\frac{\Ak[k-1]}{\Ak}\right) + \norm{\vkx}^2\cdot\left(-\lambda \Ak- \frac{\mu}{4}\lambda^2\Ak + \frac{\mu}{4} \frac{\ak[k][2]/\xi}{\Ak} \left(\lambda+\frac{2}{\mu}\right)^2\right) \\
     &+\norm{\zx[k-1]}^2 \left(-\frac{\mu}{4}\frac{\ak[k][2]}{\Ak[k-1]} - \frac{\mu\xi}{4} \Ak[k-1] + \frac{\mu\xi }{4}\frac{\Ak[k-1][2]}{\Ak} \right) \leq 0.
   \end{aligned}
\end{align}
So now we pick $-a^2$ as the resulting factor multiplying $\norm{\vkx}^2$, i.e., $-a^2 \defi -\lambda\Ak( 1+\frac{\mu\lambda}{4}) + \frac{1}{4\mu} \frac{\ak[k][2]/\xi}{\Ak} (\lambda\mu+2)^2$ and therefore we have 
\[
    b^2 = \frac{(2ab)^2}{4a^2} =  \ak[k][2] \left(1+\frac{\lambda\mu}{2} \right)^2 \left(1-\frac{\Ak[k-1]}{\Ak}\right)^2 \frac{1}{4 a^2}. 
\] 
For this computation to be valid, we need to show our choice for $-a^2$ is non-positive. We recall it holds that $\Ak = (1+c)^k$, $\Ak[k-1] = (1+c)^{k-1}$, $\ak = \xi c(1+c)^{k-1}$ and use $\kappa = \frac{1}{\mu\lambda}$, and $c = \frac{1}{2\sqrt{\xi\kappa}}$. So $-a^2\leq 0 $ if and only if $\circled{1}$ below holds:
\begin{align*}
\begin{aligned}
  \xi \left( \frac{1 + 4\kappa + 4\kappa^2}{1 + 4\kappa}\right) = \frac{\xi}{4\lambda\mu} \left( \frac{(\lambda\mu + 2)^2}{(1+\frac{\lambda\mu}{4})}\right) \circled{1}[\leq] \frac{\Ak[k][2]\xi^2}{\ak[k][2]}  = \frac{(1+c)^2}{c^2} = 1 + \frac{2}{c} + \frac{1}{c^2} = 1 + 4\sqrt{\xi\kappa} + 4\xi\kappa.
\end{aligned}
\end{align*}
And this inequality is clearly satisfied, since we assumed $\kappa \geq 1$. Indeed, drop the two first summands on the right hand side and multiply by $1+4\kappa$. 

So now we can just add $ 2ab\innp{\vkx, \zx[k-1]} + a^2\norm{\vkx}^2 + b^2 \norm{\zx[k-1]}^2 = \norm{\vkx -\zx[k-1]}^2 \geq 0$ to the left hand side of \eqref{eq:aux:in_the_middle_of_lyapunov_proof} and show the resulting inequality, in order to prove the result. So it is enough to prove the resulting factor multiplying $\norm{\zx[k-1]}^2$ is non-positive:
\begin{align*}
\begin{aligned}
    &\ak[k][2] \left(1+\frac{\lambda\mu}{2} \right)^2 \left(1-\frac{\Ak[k-1]}{\Ak}\right)^2 \left(\lambda\Ak(4+\mu\lambda) - \frac{1}{\mu} \frac{\ak[k][2]/\xi}{\Ak} (\lambda\mu+2)^2 \right)^{-1}  \\
    &-\frac{\mu\ak[k][2]}{4\Ak[k-1]}- \frac{\mu\xi\Ak[k-1]}{4}  + \frac{\mu\xi \Ak[k-1][2]}{4\Ak} \leq 0.
\end{aligned}
\end{align*}
Now, this inequality holds by substituting the values of $\Ak$, $\Ak[k-1]$, $\ak$, using $ \kappa = (\lambda\mu)^{-1} $ and $c = \frac{1}{2\sqrt{\xi\kappa}}$ and doing some simple computations. Indeed, by substituting, combining the last two summands, dividing by $\mu(1+c)^{k-2}/4$, reducing the $4+\mu\lambda$ to $2+\mu\lambda$ and simplifying terms, we obtain it is enough to prove
\begin{align*}
\begin{aligned}
    \frac{c^4(1+c)^{k-2}\xi^2(2+\lambda\mu)}{\lambda\mu(1+c)^{k} - (1+c)^{k-2}c^2\xi (\lambda\mu+2)}   -c^2(1+c)\xi^2- c\xi \leq 0.
\end{aligned}
\end{align*}
From here, the inequality follows by operating out and comparing terms. If we just substitute the value of $c^2 =1/(2\xi\kappa)$ in the instances where there is $c^4$ or $c^2$, then $\kappa$ disappears from the equation and one can compare terms to reach the result, by using $c\in (0, 1)$, $\xi > 1$.
\end{proof}

We now prove the auxiliary lemmas that were used to prove \cref{thm:riemaconsc}.

\begin{lemma}[Approx. st. g-convexity by approx. subgradient]\label{lemma:appa}
    Let $\yk$ be an $\hat{\oldepsilon}$ minimizer of $\newtarget{def:proximal_objective}{\hk}(x) \defi \min_{x\in\X} \{ f(x) + \frac{1}{2\lambda} \dist^2(\xk,x) \}$, and let $\vkx \defi -\lambda^{-1}\exponinv{\xk}(\yk)$. Then, for all $x\in\X$, we have
\[
    f(x) \geq \f(\yk) + \innp{\vkx, x-\yk}_{\xk} + \frac{\mu}{4}\norm{x-\yk}_{\xk}^2 - \left(\frac{2}{\lambda\mu}+2 \right)\hat{\oldepsilon}.
\] 
\end{lemma}

\begin{proof}
     Let $\newtarget{def:optimal_proximal_operator}{\ykast} \defi \argmin_{x\in\X} \hk(x)$. The function $\hk$ is $(\frac{1}{\lambda} + \mu)$-strongly g-convex because by \cref{fact:hessian_of_riemannian_squared_distance} the function $\frac{1}{2}\dist^2(\xk, x)$ is $1$-strongly g-convex in a Hadamard manifold. This strong convexity and optimality of the point $\ykast$ yield $\circled{1}$ below. Besides, we have
\begin{align}
 \begin{aligned}
     \f(x) &\circled{1}[\geq] \left( \f(\ykast) + \frac{1}{2\lambda}\dist^2(\xk, \ykast)\right) - \frac{1}{2\lambda}\dist^2(\xk, x) + \left(\frac{1}{2\lambda} + \frac{\mu}{2}\right) \dist^2(\ykast, x) \\
     &\circled{2}[\geq] \left(\f(\yk) + \frac{1}{2\lambda}\norm{\xk- \yk}_{\xk}^2 - \hat{\oldepsilon} \right) - \frac{1}{2\lambda}\norm{\xk - x}_{\xk}^2 + \left(\frac{1}{2\lambda} + \frac{\mu}{2}\right)\norm{\ykast - x}_{x_k}^2 \\
     &= \f(\yk) + \innp{\vkx, x-\yk}_{\xk} + \frac{\mu}{2}\norm{x-\yk}_{\xk}^2 + \left(\frac{1}{\lambda}+\mu\right)( \innp{x-\yk, \yk-\ykast}_{\xk}+ \frac{1}{2}\norm{\ykast - \yk}_{x_k}^2 )  -\hat{\oldepsilon}\\
     &\circled{3}[\geq] \f(\yk) + \innp{\vkx, x-\yk}_{\xk} + \frac{\mu}{4}\norm{x-\yk}_{\xk}^2 - \left(\frac{1}{\lambda\mu}+\frac{1}{2} \right)\left(\frac{1}{\lambda}+\mu\right)\norm{\ykast - \yk}_{x_k}^2 -\hat{\oldepsilon}\\
     &\circled{4}[\geq] \f(\yk) + \innp{\vkx, x-\yk}_{\xk} + \frac{\mu}{4}\norm{x-\yk}_{\xk}^2 - \left(\frac{2}{\lambda\mu}+2 \right)\hat{\oldepsilon}.\\
   \end{aligned}
\end{align}
    where in $\circled{2}$ we used the $\hat{\oldepsilon}$-optimality of $\yk$ for $\hk(\cdot)$ and we used for the last summand that in a Hadamard manifold we have $\dist(x, y) \geq \norm{x-y}_{z}$ for any three points $x, y, z$.

    In $\circled{3}$, we used Young's inequality: 
    \[
        \innp{x-\yk, \left(\frac{1}{\lambda}+ \mu\right)(\yk - \ykast)}_{\xk} \geq -\frac{\mu}{4}\norm{x-\yk}^2_{\xk} - \frac{1}{\mu}\left(\frac{1}{\lambda} + \mu\right)^2 \norm{\yk - \ykast}^2,
    \] 
    and grouped some terms. Finally, in $\circled{4}$ we used $-\norm{\ykast - \yk}_{x_k}^2 \geq -\dist^2(\ykast, \yk)$ and then $(\frac{1}{\lambda} + \mu)$-strong convexity of $\hk$ along with $\hat{\oldepsilon}$-optimality of $\yk$: $(\frac{1}{\lambda} + \mu) \norm{\ykast-\yk}_{\xk}^2 \leq \hk(\yk)-\hk(\ykast) \leq \epsilon$.
\end{proof}

\begin{lemma}[Translating potentials with no geometric penalty]\label{lemma:moving_lower_bounds}
    Using the notation in \cref{alg:riemacon_sc_absolute_criterion}, we have  
\begin{align*}
 \begin{aligned}
     \Ak[k-1]&(\norm{\zx[k-1]-\xast}_{\xk}^2 + (\xi-1)\norm{\zx[k-1]}_{\xk}^2) - \Ak(\norm{\zx-\xast}_{\xk}^2  +(\xi-1)\norm{\zx}_{\xk}^2 ) \\
     &\leq \Ak[k-1](\norm{\zy[k-1][k-1]-\xast}_{\yk[k-1]}^2 + (\xi-1) \norm{\zy[k-1][k-1]}_{\yk[k-1]}) \\
     &\quad - \Ak(\norm{\zy-\xast}_{\yk}^2 +(\xi-1)\norm{\zy[k-1][k-1]}_{\yk[k-1]}^2-\norm{\yk-\zy}_{\yk}^2).
   \end{aligned}
\end{align*}
    
\end{lemma}

\begin{proof}
    Firstly, by the projection step in Line \ref{line:dual_projecting_step}, we have 
    \begin{equation}\label{eq:aux_translating_w_o_penalty_1}
        \norm{\zy[k-1][k-1] - \xast}_{\yk}^2 \geq \norm{\zybar[k-1][k-1] - \xast}_{\yk}^2 \quad \quad  \text{ and } \quad \quad  (\xi-1)\norm{\zy[k-1][k-1]}_{\yk}^2 \geq (\xi-1)\norm{\zybar[k-1][k-1]}_{\yk}^2 
    \end{equation}
since the operation is a simple Euclidean projection onto the closed ball $\ball(0,\D)$ in $\Tansp{\yk}\M$. Now, the following holds
\begin{align}\label{eq:aux_translating_w_o_penalty_2}
 \begin{aligned}
     \norm{\zybar[k-1][k-1]&-\xast}_{\yk[k-1]}^2 + (\xi-1)\norm{\zybar[k-1][k-1]}_{\yk[k-1]}^2 \circled{1}[\geq] 
     \norm{\zx[k-1]-\xast}_{\xk}^2  + (\zeta[2\D]-1)\norm{\zx[k-1]}_{\xk}^2 + (\xi-\zeta[2\D])\norm{\zybar[k-1][k-1]}_{\yk[k-1]}^2 \\
     &\circled{2}[\geq] \norm{\zx[k-1]-\xast}_{\xk}^2  + (\xi-1)\norm{\zx[k-1]}_{\xk}^2 + (\xi-\zeta[2\D])\left(\left( \frac{\Ak[k-1]+\ak}{\Ak[k-1]}\right)^2-1\right)\norm{\zx[k-1]}_{\xk}^2 \\
     &\circled{3}[\geq] \norm{\zx[k-1]-\xast}_{\xk}^2  + (\xi-1)\norm{\zx[k-1]}_{\xk}^2 + \frac{3(\xi-1)}{2}\left(\left( \frac{\Ak[k-1]+\ak}{\Ak[k-1]}\right)^2-1\right)\norm{\zx[k-1]}_{\xk}^2, 
 \end{aligned}
\end{align}

    where $\circled{1}$ is due to \cref{corol:moving_quadratics:exact_approachment_recession}, with $y\gets\xk$ and $x\gets\yk[k-1]$ and to $\dist(\xk, p) \leq \dist(\xk, \yk[k-1]) + \dist(\yk[k-1], p) \leq \norm{\zy[k-1][k-1]}_{\yk[k-1]} + \D \leq 2D$ for any $p \in \M$. Inequality $\circled{2}$ uses the definition of $\xk$. In $\circled{3}$, we used the definition of $\xi = 4\zeta[2\D] - 3$ that implies $\xi-\zeta[2\D] \geq \frac{3}{4}(\xi-1)$. Now, we use \cref{lemma:moving_quadratics:inexact_approachment_recession} with $y\gets \yk$, $x\gets \xk$ $z^x = -\vkx\cdot \frac{\ak}{\xi}(\lambda+\frac{4}{\mu})/\Ak[k]$, $a^x \gets \zx[k-1](\Ak[k-1]/\Ak)$, so that $z^x+a^x = \zx$ and $z^y + a^y = \zy$ and 
    \begin{equation*}\label{ineq:r_is_less_than_1}
        r=\frac{\norm{\exponinv{\xk}(\yk)}}{\norm{z^x}}  = \frac{\lambda\norm{\vkx}}{\norm{\vkx}\cdot\frac{\ak}{\xi}(\lambda+\frac{4}{\mu})\Ak[k][-1]} = \frac{\xi\Ak}{\ak(1+\frac{4}{\lambda\mu})} \circled{1}[\leq] \sqrt{\frac{\xi}{\kappa}} \circled{2}[\leq] \frac{1}{3} < 1. 
    \end{equation*}
    We will now explain why $\circled{1}$ holds. But first note that by the previous inequality, by the choice of parameters and the fact that $r<1$, the assumptions in \cref{lemma:moving_quadratics:inexact_approachment_recession} are satisfied. Also, note that $\circled{2}$ holds by the assumption on $\lambda$. We have $\circled{1}$ above if and only if the following holds
\begin{equation}\label{eq:restriction_1_on_ak}
    \xi\Ak[k-1] \leq \sqrt{\frac{\xi}{\kappa}}\ak\left(-\sqrt{\frac{\kappa}{\xi}}+ 1+4\kappa \right),
\end{equation}
and this is implied by $\xi \leq \sqrt{\frac{\xi}{\kappa}}c\xi\cdot3\kappa$, by substituting the values of $\Ak[k-1]$ and $\ak$ and using $3\kappa \leq 4\kappa - \sqrt{\kappa/\xi} \leq 4\kappa - \sqrt{\kappa/\xi}+1$, which comes from our assumption $\kappa \geq 9\xi \geq \xi$. Consequently, a sufficient condition is $c \geq 1/(3\sqrt{\xi\kappa})$, which is satisfied by our choice $c= 1/(2\sqrt{\xi\kappa})$.

The result in \cref{lemma:moving_quadratics:inexact_approachment_recession}, applied as above results in 
\begin{align}\label{eq:aux_translating_w_o_penalty_3}
 \begin{aligned}
     \norm{\zx-\xast}_{\xk}^2 + (\xi-1)\norm{\zx}_{\xk}^2 + \frac{\xi-1}{2}\left(\frac{r}{1-r}\right)\frac{\Ak[k-1][2]}{\Ak[k][2]}\norm{\zx[k-1]}^2 \geq \norm{\zy-\xast}_{\yk}^2 + (\xi-1)\norm{\zy}_{\yk}^2.
 \end{aligned}
\end{align}
Combining \cref{eq:aux_translating_w_o_penalty_2} multiplied by $\Ak[k-1]$ with \cref{eq:aux_translating_w_o_penalty_3} multiplied by $\Ak$, we obtain that in order to conclude, it suffices to show
\[
\Ak \frac{\xi-1}{2} \frac{r}{1-r} \frac{\Ak[k-1][2]}{\Ak[k][2]} - \Ak[k-1] \frac{3(\xi-1)}{2} \left( \left(\frac{\ak+\Ak[k-1]}{\Ak[k-1]}\right)^2-1 \right) \leq 0.
\] 
We substitute the value of $r$ and after simplifying we obtain that we want to show
\[
    \xi \Ak[k-1][3] - 3(\ak[k][2] + 2\ak\Ak[k-1])\left( \ak \left(1+\frac{4}{\lambda\mu}\right) - \xi \Ak\right) \leq 0
\] 
After substituting the value of $\Ak$, $\Ak[k-1]$, $\ak$, operating out and dividing by $\xi(1+c)^{3k-3}$, we obtain that the previous inequality is equivalent to
\[
    1 + 3\xi^2(1+c)c^2 + 6\xi c  + 6\xi c^2  \leq 3\xi^2 (1+4\kappa) c^3 +6\xi c^2(1+4\kappa) = \frac{3\xi c}{2\kappa}+ 6\xi c + \frac{3}{\kappa} + 12
\] 
where in the last equality we used $c^2 = 1/(4\xi\kappa)$. This inequality holds. After simplifying the terms $6\xi c$ we get that the left hand side is $\leq 1 + \frac{3(1+c)}{2\kappa} + \frac{3}{\kappa} \leq 7$ where we used the value of $c = 1/(2\sqrt{\xi\kappa}) \leq 1$ and $\kappa \geq 1$.

\end{proof}

\begin{proof}\linkofproof{cor:smooth-riemaconsc}
    By \cref{thm:riemaconsc}, it suffices to run \cref{alg:riemacon_sc_absolute_criterion} for
  \begin{equation*}
      T'' \geq 2\sqrt{\frac{\xi}{\lambda\bar{\oldmu}} + 9\xi^2}\log_2\left(\frac{2\lambda^{-1}\dist^2(x_0, x^\ast)}{\epsilon}\right)
  \end{equation*}
    iterations in order to obtain an $\epsilon$-minimizer. We use $\lambda=1/\L$, and recall $\xi=\bigotilde{\zetad}$. Step $k$ of \cref{alg:riemacon_sc_absolute_criterion} requires computing a $\hat{\oldepsilon}$-minimizer of $\min_{x \in \XX}\hk(x)$, where $\hk(x)\defi f(x)+\frac{1}{2\lambda}\dist^2(x_k,x)$ and
  \begin{equation*}
      \hat{\oldepsilon}=\frac{\epsilon}{8 \sqrt{\xi\left(\frac{\L}{\min\{\mu, \L/(9\xi) \}}\right)^3}}. 
\end{equation*}
    We implement the subroutine with \PRGD{} with learning rate $\frac{1}{L'}$, where $L' \defi L(1+\zeta)$ is a bound on the smoothness of $h$, cf. \cref{fact:hessian_of_riemannian_squared_distance}. We require the following number of steps
\[
    T'\ge 1+ 2\frac{L'}{\mu'}\zeta[R]\log \left(\frac{L'\zeta[R]D_{\XX}^2}{2\hat{\oldepsilon}} \right),
\] 
    where $R=\Lips(h, \XX)/L'$. We can bound
\[
    R \leq \frac{\max_{x\in\XX} \norm{\nabla f(x)} + \L \dist(x, x_k)}{\L(1+\zetad)} \circled{1}[\leq]  \frac{\max_{x \in \XX} \{ \norm{\nabla f(x)}/\L \}+2D_{\XX}}{\zetad},
\] 
    where in $\circled{1}$ we used that for all $x\in\XX$, it is 
    \[
        \dist(x, \xk) \leq \dist(\yk[k-1], \xk) + \dist(x, \yk[k-1]) \circled{2}[<] \dist(\yk[k-1], \zybar[k-1][k-1]) + D_{\XX} \circled{3}[\leq] 2D_{\XX}.
    \] 
    where $\circled{2}$ holds by definition of $\xk$ and the fact $\yk[k-1], x \in \XX$, while $\circled{3}$ is due to the projection defininig $\zybar[k-1][k-1]$.

The condition number of $h$ is bounded by
\begin{equation*}
 \frac{L'}{\mu'}\le \lambda \L+\zetad \le 1 +\zetad.
\end{equation*}
So we have
\begin{equation*}
    T'\ge 1+2\zeta[R](1+\zetad)\log \left(4\frac{\L(1+\zetad)\zeta[R]D_{\XX}^2\sqrt{(\frac{\L}{\mu}+9\xi)^3\xi}}{\epsilon} \right),
\end{equation*}
    The complete complexity of $\riemacon{}$ is
\begin{equation*}
 T=T'T''=\bigotilde{\zeta[R]\zetad^{\frac{3}{2}}  \sqrt{\kappa+\zetad}}.
\end{equation*}
\end{proof}

\begin{corollary}\label{corol:PRGD_is_free_if_global_optim_is_in_the_set}
    Under the assumptions from \cref{cor:smooth-riemaconsc}, if a global minimizer $\xast \in \argmin_{\M} f(x)$ is in $\XX$, so that $\nabla f(\xast) = 0$, then \cref{alg:riemacon_sc_absolute_criterion} with $\lambda=1/\Lsmooth$ 
    and \PRGD{} as subroutine, as in \cref{cor:smooth-riemaconsc} yields an $\epsilon$-minimizer after $\bigotilde{\zetad^{3/2}  \sqrt{\kappa+\zetad}}$  gradient and metric-projection oracle calls.
\end{corollary}

\begin{proof}
    By the assumption on $\xast$ and the smoothness assumption, we have that $\Lips(f, \XX) \leq L D_{\XX}$ and since $\zeta \geq D_{\XX}\sqrt{\abs{\kmin}}$, we obtain $R \leq (\Lips(f, \XX)/\Lsmooth+2D_{\XX})/\zetad = \bigo{\frac{1}{\sqrt{\abs{\kmin}}}}$ and thus $\zeta_R = \bigo{1}$. We obtain the result by applying \cref{cor:smooth-riemaconsc}. Note that we can also use \PRGD{} with $\riemaconrel$ and \cref{cor:riemaconrel} and similarly it is $\zeta_R = \bigo{1}$.
\end{proof}

Note that \citet[Theorem 6]{martinez2022accelerated} had to assume a mild condition on the Hadamard manifold and obtained overall query complexity $\bigotilde{\zeta^2\sqrt{\kappa}}$ whereas we obtain lower complexity in the general Hadamard case with bounded sectional curvature. Note that to compare to \citet[Theorem 6]{martinez2022accelerated} we would run our algorithm for $\X$ a ball of center $x_0$ and radius an upper bound on $\dist(x_0, x^\ast)$ and in such a case the implementation of the metric-projection oracle consist of computing a direct and an inverse exponential and simple operations, cf. \cref{sec:preliminaries}, so it does not increase the order of our computational complexity. We note that \citep{martinez2022accelerated} provided another instantiation of their algorithm for g-convex minimization but under the assumption of having access to a projection oracle that is not a metric-projected oracle.

\begin{remark}\label{rem:g-convex_reduction}
    We can obtain the accelerated result for the g-convex case with reduced geometric penalties via a reduction to the $\mu$-strongly g-convex case. Assume the existence of a global minimizer $\xast$ and let $\D/2 \geq \dist(\xInit, \xast)$. Given an $\epsilon > 0$, we optimize the regularized function $f_{\epsilon}(x) \defi f + \frac{\epsilon}{\D^2}\dist(\xInit, x)^2$. Denote by $x^\ast_{\epsilon}$ to the minimizer of $f_{\epsilon}$. We have $\dist(\xInit, x^\ast_{\epsilon}) \leq \dist(\xInit, \xast) \leq \D/2$ (see \citep[Lemma 10]{martinez2022accelerated}). 
    We run \cref{alg:riemacon_sc_absolute_criterion} on $f_{\epsilon}$ on a ball $\ball(x_0, \D/2)$. We have that $f_{\epsilon}$ is strongly g-convex with constant $\frac{2\epsilon}{\D^2}$, cf. \cref{fact:hessian_of_riemannian_squared_distance}. 
    Hence, the algorithm finds an $\epsilon/2$ minimizer $x_{T'}$ of $f_{\epsilon}$ after $T'=\bigotilde{\zetad +  \sqrt{\zetad \D^2/(\lambda\epsilon)}}$ iterations. By definition, it is $\dist(\xInit, \xast) \leq \D/2$ so the regularization at $\xast$ is $\frac{\epsilon}{\D^2}\dist(x_0, \xast)^2 \leq \frac{\epsilon}{4}$ and thus $x_{T'}$ is an $\epsilon$-minimizer of $\f$:
    \[
        \f(x_{T'}) \leq f_{\epsilon}(x_{T'}) \leq f_{\epsilon}(\xast) + \frac{\epsilon}{4} \leq \f(\xast) + \epsilon.
    \] 
\end{remark}

} 

\subsection{Convergence of Projected Gradient Descent}

\begin{proof}\linkofproof{lem:prgd}
    Below, we prove that for any point $x_t\in\X$, \PRGD{} yields
\begin{equation}\label{eq:result_for_PRGD}
    f(x_{t+1})-f(x^{*})\le (f(x_t)-f(x^{*})) \left(1-\frac{\mu}{4L\zeta[R_t]}\right),
\end{equation}
    where $R_t \defi \norm{\nabla f(x_t)}/\L$. Recall our notation $\Lips(f, \X)$ for denoting the Lipschitz constant of $f$ in $\X$. Given \eqref{eq:result_for_PRGD} above, and defining $R \defi \Lips(f, \X)/\L$, we have by applying \eqref{eq:result_for_PRGD} $T$ times from $x_0$, that the following holds
    \begin{equation*}\label{eq:aux_prgd}
        f(x_T) - f(\xast) \leq \min\left\{(f(x_0)-f(x^{*}))\left(1-\frac{\mu}{4L\zeta[R]}\right)^{T}, \frac{\L\zeta[R]}{2}\dist^2(x_0, \xast)\left(1-\frac{\mu}{4L\zeta[R]}\right)^{T-1} \right\},
    \end{equation*}
    since by \cref{lem:dist_to_gap_by_Proj_RGD} we have $f(x_1) - f(\xast) \leq \frac{\L\zeta[R]}{2} \dist^2(x_0, \xast)$. The result follows by bounding the right hand side of the expression above by $\epsilon$ and reorganizing.

    We now prove \eqref{eq:result_for_PRGD}. The following holds:
\begin{align*}
 \begin{aligned}
     \f(x_{t+1}) &\circled{1}[\leq] \min_{x\in\X} \left\{\f(x) + \frac{\L\zeta[R_t]}{2}\dist^2(x, x_t) \right\} \\
     & \circled{2}[\leq] \min_{\alpha \in [0,1]} \left\{\alpha \f(x^\ast) + (1-\alpha) \f(x_t) + \frac{\L\zeta[R_t]\alpha^2}{2}\dist^2(x^\ast, x_t) \right\} \\
     & \circled{3}[\leq] \min_{\alpha \in [0,1]} \left\{\f(x_t) - \alpha\left(1-\alpha\frac{\L\zeta[R_t]}{\mu}\right)\left(\f(x_t)-\f(x^\ast)\right) \right\} \\
     &\circled{4}[=] \f(x_t) - \frac{\mu}{4\L\zeta[R_t]}(\f(x_t)-\f(x^\ast)).
 \end{aligned}
\end{align*}
    Above, we used \cref{lem:dist_to_gap_by_Proj_RGD} to conclude $\circled{1}$, and $\circled{2}$ results from restricting the minimum to the geodesic segment between $x^\ast$ and $x_t$ so that ${x = \expon{x_t}(\alpha x^\ast + (1-\alpha)x_t)}$. We also use g-convexity of $f$. In $\circled{3}$, we used strong convexity of $f$ to bound $\frac{\mu}{2} \dist^2(x^\ast, x_t) \leq \f(x_t)-\f(x^\ast)$. Finally, in $\circled{4}$ we substituted $\alpha$ by the value that minimizes the expression, which is $\mu/(2\L\zeta[R_t])$. The result in \eqref{eq:result_for_PRGD} follows by subtracting $\f(x^\ast)$ to the inequality above. The final statement is a direct consequence of \eqref{eq:result_for_PRGD} and the definition of $R$, along with $f(x_1) - f(x^\ast) \leq \frac{\L \zeta[R]}{2}\dist^2(x_0, x^\ast)$ which holds due to \cref{lem:dist_to_gap_by_Proj_RGD}.

\end{proof}

\begin{remark}\label{remark:minimum_kappa}
    Let $\MOnly$ be a Hadamard manifold with curvatures in the interval $[\kmin, \kmax]$ where $\kmax < 0$.
\citet[Theorem 2]{criscitiello2022negative} provide the query complexity lower bound $\tilde \Omega(\sqrt{\frac{\kmax}{\kmin}} \zeta)$ for minimizing a smooth strongly g-convex function in a ball.   On the other hand, \citet[Prop 17]{martinez2022accelerated} provide an $\bigotilde{\kappa}$ gradient query complexity upper bound in this setting.  From this, we conclude that $\kappa \geq \tilde \Omega(\sqrt{\frac{\kmax}{\kmin}} \zeta)$.  This fact is also shown by \citet[Prop 28]{criscitiello2022negative}, albeit in a very different way.
\end{remark}

\begin{lemma}[Dist to Gap and Warm Start]\label{lem:dist_to_gap_by_Proj_RGD}
    Let $\M$ be a Hadamard manifold $\X\subseteq\M$ be a uniquely g-convex set of diameter $\D$, $\bar{x} \in \X$, and $g:\M\to\R$ a g-convex and $\L$-smooth function in $\X$ with a minimizer at $x^\ast \in \argmin_{x\in\X} g(x)$. Assume access to a metric-projection operator $\proj$ on $\X$ and let $x'\defi \proj(\expon{\bar{x}}(-\frac{1}{\L}\nabla g(\bar{x})))$, and $R\defi \dist(x', \bar{x}) =\norm{\nabla g (\bar{x})}/\L$. The following holds for all $p \in \X$:
    \[
        g(x')-g(p) \leq \frac{\zeta[R] \L}{2}\dist^2(\bar{x}, p).
    \]
    In particular, we have
    \[
         g(x')-g(x^\ast) \leq \frac{\zeta[R] \L}{2}\dist^2(\bar{x}, x^\ast).
    \]
\end{lemma}
See \citep[Lemma 18 (Warm start)]{martinez2022accelerated} for a proof.

\subsection{Convergence of Riemannian Alternating Best Response}\label{sec:convergence_of_ABR}

We use $\newtarget{def:Riemacon_relative_error_SC}{\riemaconrel}$ to refer to the accelerated algorithm for $\mu$-strongly convex functions presented in \citep[Theorem 4]{martinez2022accelerated}, and $\riemaconrel(f, x_0, T, \X, \texttt{subroutine})$ to specify the output of the algorithm initialized at $x_0$ for optimizing the function $f$ constrained to $\X$, run for $T$ steps and making use of the subroutine \texttt{subroutine}.

\begin{fact}[Convergence of $\riemaconrel$]\label{fact:riemaconrel}
Let $\M$ be a finite-dimensional Hadamard manifold of bounded sectional curvature,
    and consider $f: \X\subset \M \rightarrow \mathbb{R} $, a g-convex function in a compact g-convex set $\X$ of diameter $D_{\XX}$, $\lambda>0$, and $x^{*}\in \argmin_{x\in \X}f(x)$. Define $\xiOnly \defi 4\zeta[2\D]-3$.
    If $f$ is \(\mu\)-strongly g-convex then, running $\riemaconrel$ as defined in \citep[Theorem 4]{martinez2022accelerated} for $T=(90\xiOnly/\sqrt{\mu\lambda})\log (\mu \dist^2(x_0,x^{*})/\epsilon)$ iterations, returns a point $y_T$ that satisfies $f(y_T)-f(x^{*})\le \epsilon$.
\end{fact}
See \citep[Theorem 4]{martinez2022accelerated} for a proof.
\begin{corollary}\linktoproof{cor:riemaconrel}\label{cor:riemaconrel}
    If $f$ as defined in \cref{fact:riemaconrel} is in addition  $\L$-smooth, then $\riemaconrel$ with $\lambda=1/\L$ and \PRGD{} as subroutine, yields an $\epsilon$-minimizer after
   $\bigotilde{\zeta[R]\zetad^2 \sqrt{\kappa}}$ gradient and metric-projection oracle calls, where
    $R\le (\Lips(f, \XX)/\L+2D_{\XX})/\zetad$.
\end{corollary}
\begin{proof}\linkofproof{cor:riemaconrel}
  By \cref{fact:riemaconrel}, it suffices to run $\riemaconrel$ for
  \begin{equation}\label{eq:riemaconrel_outer}
T''= \frac{90\xiOnly}{\sqrt{\mu\lambda}}\log \left( \frac{\mu \dist^2(x_0,x^{*})}{\epsilon}\right)
  \end{equation}
  iterations in order to obtain an $\epsilon$-minimizer. Note $\xiOnly=\bigotilde{\zetad}$.
  We use $\lambda=1/\L$.
  Note that $\riemaconrel$ uses a series of restarts, so in the following, $k$ refers to the $k$-th iteration in one of the calls to \citep[Algorithm 1]{martinez2022accelerated}.
  This detail can be ignored in the following, as we bound $k$ uniformly by $T''$ in \cref{eq:riemaconrel_inner}.
    Step $k$ of $\riemaconrel$ requires computing a $\sigma_k$-minimizer of $\min_{x \in \X}\hk(x)$, where $\hk(x)\defi f(x)+\frac{1}{2\lambda}\dist^2(x_k,x)$ and
  $\sigma_k\defi\dist^2(x_k,x^{*}_k)/(78\lambda(k+1)^2) $.
    We solve the prox problem using \PRGD{} with learning rate $\frac{1}{L'}$, where $L' \defi L(1+\zeta)$ is a bound on the smoothness of $h$, cf. \cref{fact:hessian_of_riemannian_squared_distance}.
  By \cref{lem:prgd}, this requires
\[
    T'\ge 1+ 2\frac{L'}{\mu'}\zeta[R]\log \left(\frac{L'\zeta[R]D_{\XX}^2}{2\sigma_k} \right)
  \]
  iterations, where $L'$ and $\mu'$ are smoothness and strong g-convexity constants of $h$ respectively and $R=\norm{\nabla h(x_0)}/L'$. We can bound
\[
    R \leq \frac{\max_{x\in\X} \norm{\nabla f(x)} + \L \dist(x, x_k)}{\L(1+\zetad)} \leq  \frac{\max_{x \in \X} \norm{\nabla f(x)}/\L+2D_{\XX}}{\zetad},
\]
    where the last inequality uses $\dist(\XX, x_k) \leq 2D_{\XX}$, cf. \citep{martinez2022accelerated}.
We have that $\kappa'$, the condition number of $h$, is bounded by $\kappa'\le\frac{L'}{\mu'}\le \lambda \L+\zetad \le 1 +\zetad$.
By the definition of $\kappa'$ and bounding $k\le T''$, we obtain
\begin{equation}\label{eq:riemaconrel_inner}
    T'\ge1+4\zeta[R]\zetad\log (78\zetad\zeta[R](T''+1)^2)\ge 1+4\zeta[R]\zetad\log (78\zetad\zeta[R](k+1)^2).
\end{equation}
All in all $\riemaconrel$ requires
\begin{equation*}
 T=T'T''=\bigotilde{\zeta[R]\zetad^{2}  \sqrt{\kappa}}
\end{equation*}
calls to the gradient and metric projection oracle respectively.
\end{proof}
\begin{proof}\linkofproof{thm:crabr}
We begin by connecting the inexactness of the iterates $(x_t,y_t)$ from \cref{alg:crabr} to the number of $\riemaconrel$ iterations,
  \begin{equation}\label{eq:riemacon-x}
    \begin{aligned}[t]
        \dist^2(x_{t+1},x^{*}(y_t))&\circled{1}[\le] \frac{2}{\mux} [f(x_{t+1},y_t)-f(x^{*}(y_t),y_t)]\\
        &\circled{2}[\le] 2d^{2}(x_t,x^{*}(y_t)) \exp\left( \frac{-T_x}{90\xi \sqrt{\kappa_{x}}} \right).
   \end{aligned}
  \end{equation}
    We used strong g-convexity of $x\mapsto f(x, y_t)$ in $\circled{1}$, and $\circled{2}$ follows from running $\riemaconrel$ on $f(\cdot,y_t)$ for $T_x$ iterations starting from $x_t$ and ending up with $x_{t+1}$.
  Noting that since $-f(x_{t+1},\cdot)$ is strongly g-convex, we can repeat the arguments for $y$,
  \begin{equation} \label{eq:riemacon-y}
    \begin{aligned}[t]
    \dist^2(y_{t+1},y^{*}(x_{t+1}))& \leq  \frac{2}{\muy} [f(x_{t+1},y^{*}(x_{t+1}))-f(x_{t+1},y_{t+1})]\\
    &\le 2 \dist^{2}(y_t,y^{*}(x_{t+1})) \exp \left(  \frac{-T_y}{90\xi \sqrt{\kappa_{y}}} \right).
    \end{aligned}
  \end{equation}
  Note that we use $\riemaconrel$ instead of our $\riemacon$ in \cref{alg:crabr}, since the absolute error criterion of $\riemacon$ creates an unwanted dependence between the required precision of the proximal problems and the precision of the original problem.
Choosing $T_x=90\xi \sqrt{\kappa_x} \log(512)$ and $T_y=90\xi \sqrt{\kappa_y} \log(512) $, it follows from \cref{eq:riemacon-x,eq:riemacon-y} that,
  \begin{align}
    \dist^2(x_{t+1},x^{*}(y_t))&\le \frac{1}{256}\dist^{2}(x_t,x^{*}(y_t))\label{eq:xdescent}\\
    \dist^2(y_{t+1},y^{*}(x_{t+1}))&\le \frac{1}{256}\dist^{2}(y_t,y^{*}(x_{t+1}))
  \end{align}
 Further,
  \begin{align} \label{eq:xbound}
    \begin{aligned}[t]
        \dist(x_{t+1},x^{*})&\circled{1}[\le] \dist(x_{t+1},x^{*}(y_t))+ \dist(x^{*}(y_t),x^{*})\\
                    &\circled{2}[\le] \frac{1}{16}\dist(x_t,x^{*})+\frac{17}{16}\dist(x^{*}(y_t),x^{*})\\
                    &\circled{3}[\le] \frac{1}{16}\dist(x_t,x^{*})+\frac{17\Lxy}{16\mux}\dist(y_t,y^{*}).
    \end{aligned}
  \end{align}
    We used the triangle inequality in $\circled{1}$, and $\circled{2}$ follows from \cref{eq:xdescent} and the triangular inequality again. Finally, $\circled{3}$ uses \cref{lem:lip}, noting that $x^{*}=x^{*}(y^{*})$.
    For $y$, we follow the same argument and then use \cref{eq:xbound} in $\circled{1}$ below
   \begin{align}\label{eq:ybound}
     \begin{aligned}[t]
    \dist(y_{t+1},y^{*})&\le \dist(y_{t+1},y^{*}(x_{t+1}))+ \dist(y^{*}(x_{t+1}),y^{*})\\
                    &\le \frac{1}{16}\dist(y_t,y^{*})+\frac{17}{16}\dist(y^{*}(x_{t+1}),y^{*})\\
                    &\circled{1}[\le] \frac{1}{16}\dist(y_t,y^{*})+\frac{17\Lxy}{16\muy}\left(\frac{1}{16}\dist(x_t,x^{*})+\frac{17\Lxy}{16\mux}\dist(y_t,y^{*}) \right)\\
     &\le \left( \frac{1}{16}+ \frac{17^2\Lxy^2}{16^2\mux\muy} \right)\dist(y_t,y^{*}) +\frac{17\Lxy}{16^2\muy}\dist(x_t,x^{*}).
     \end{aligned}
   \end{align}
   Now we define $C\defi\muy/\mux$ and obtain
   \begin{align*}
       \dist^2(x_{t+1},x^{*})+Cd^2(y_{t+1},y^{*})&\circled{1}[\le]  2\left(\frac{1}{16^2}+ C \left( \frac{17\Lxy}{16^2\muy}   \right)^{2}  \right)   \dist^2(x_{t},x^{*})\\
                                           &\quad +2\left( \frac{C}{C} \left( \frac{17\Lxy}{16\mux} \right)^2  +C \left(  \frac{1}{16^2}+ \left(\frac{17^2\Lxy^2}{16^2\mux\muy}   \right)^2\right) \right)\dist^{2}(y_t,y^{*})\\
                                           &\circled{2}[\le] 2 \dist^2(x_{t},x^{*}) \left( \frac{1}{16^2} + \frac{17^2}{16^2}\frac{1}{4} \right) + 2C \dist^{2}(y_t,y^{*})  \left( \frac{17^2}{16^2}\frac{1}{4} + \left( \frac{1}{16^2}+ \frac{1}{4^2}\frac{17^2}{16^2} \right) \right)\\
                                           & \le  \frac{3}{5}( \dist^2(x_{t},x^{*})  + Cd^{2}(y_t,y^{*})).
   \end{align*}
    where $\circled{1}$ follows from $(a+b)^2 \leq 2a^2 + 2b^2$, \cref{eq:xbound,eq:ybound}.
    Inequality $\circled{2}$ is obtained by using $\Lxy<\frac{1}{2}\sqrt{\mux\muy}$, which holds by assumption, and by the definition of $C$.
By expanding the previous inequality, it follows
\begin{equation*}
  \dist^2(x_T,x^{*}) + C \dist^2(y_T,y^{*})  \le \left(\frac{3}{5}  \right)^{T} \left(\dist^2(x_0,x^{*}) +C \dist^2(y_0,y^{*})   \right).
\end{equation*}
    Now we study two cases. If $C\ge 1$, we have $\circled{1}$ below
\begin{equation*}
    \dist^2(x_T,x^{*}) + \dist^2(y_T,y^{*})  \circled{1}[\le] \left(\frac{3}{5}  \right)^{T}\frac{\muy}{\mux}\left(\dist^{2}(x_0,x^{*}) + \dist^2(y_0,y^{*})   \right)\circled{2}[\le] \left(\frac{3}{5}  \right)^{T}\frac{\Lx}{\mux}\left(\dist^2(x_0,x^{*}) + \dist^2(y_0,y^{*})   \right),
\end{equation*}
    where $\circled{2}$ is due to $\muy\le \Ly$ and $\Lx=\Ly$. Recall that we assumed the latter without loss of generality. Similarly, if $C\in (0,1)$ we have, for $C\le 1$,
\begin{equation*}
 \dist^2(x_T,x^{*}) + \dist^2(y_T,y^{*}) \le\left( \frac{3}{5}  \right)^{T}  \frac{\mux}{\muy}\left(\dist^2(x_0,x^{*}) + \dist^2(y_0,y^{*})   \right)  \le\left( \frac{3}{5}  \right)^{T}  \frac{\Ly}{\muy}\left(\dist^2(x_0,x^{*}) + \dist^2(y_0,y^{*})   \right) .
\end{equation*}
Thus, for all $C>0$, we obtain
\begin{equation}\label{eq:rabr-dist}
   \dist^2(x_T,x^{*}) + \dist^2(y_T,y^{*}) \le\left( \frac{3}{5}  \right)^{T} \max\{\kappa_x,\kappa_y\} \left(\dist^2(x_0,x^{*}) + \dist^2(y_0,y^{*})   \right).
\end{equation}
Hence, we require 
\begin{equation*}
    T'= \mathcal{O} \left(\log \left( \frac{(\dist^2(x_0,x^{*}) + \dist^{2}(y_0,y^{*}))( \kappa_x+\kappa_y)}{\epsilon} \right)  \right),
\end{equation*}
iterations of \cref{alg:crabr} to ensure $\dist^2(x_T,x^{*}) + \dist^2(y_T,y^{*})\le\epsilon$.
By \cref{cor:riemaconrel}, each $\riemaconrel$ call requires  $T_x'=\bigotilde{\zeta[R]\zetad\sqrt{\kappa_x}}$ and $T_y'=\bigotilde{\zeta[R]\zetad\sqrt{\kappa_y}}$ calls to the gradient and metric projection  oracle respectively.
In total, \RABR{} requires
\begin{equation*}
T' (T_xT_x' +T_yT_y' )=\bigotilde{ \zeta[R]\zeta[]^{2} \sqrt{\kappa_x+\kappa_y}}
\end{equation*}
calls to the gradient and metric projection oracle.
\end{proof}

\subsection[Convergence Analysis of RAMMA]{Convergence analysis of \RAMMA{}}

We first prove this important proposition, that describes how to go from one measure of convergence to another, possibly after performing some optimization steps.

\begin{proof}\linkofproof{proposition:from_one_opti_measure_to_another}
\cref{item:full_gap_to_individual_gap} follows from the non-negativity of the gaps and $\gapx[\bar{x}] + \gapy[\bar{y}]= \gap[\bar{x}, \bar{y}]$. \cref{item:gap_to_dist} follows from the strong convexity of $\phi_x$ and $\phi_y$. 

    We now prove \cref{item:one_gap_to_dist}. By strong concavity of $\phi_y$ we have $\dist^2(\bar{y}, y^\ast) \leq \frac{2}{\bar{\oldmu}_y}\gapy[y] \leq \frac{2\epsilon}{\bar{\oldmu}_y}$. The optimizer of $g(\cdot, \bar{y})$ is $x^\ast(\bar{y})$, so by $\bar{\oldmu}_x$-strong g-convexity of this function and optimality of $\bar{x}'$, we have $\dist^2(\bar{x}', x^\ast(\bar{y})) \leq \frac{2\hat{\oldepsilon}}{\bar{\oldmu}_x}$. Thus, we have
\begin{align}
\begin{aligned}
    \dist^2(\bar{x}', x^\ast) + \dist^2(\bar{y}, y^\ast) &\circled{1}[\leq] 2\dist(\bar{x}', x^\ast(\bar{y}))^2 + 2\dist(x^\ast(\bar{y}), x^\ast)^2 + \dist^2(\bar{y}, y^\ast) \\
        &\circled{2}[\leq] \frac{4\hat{\oldepsilon}}{\bar{\oldmu}_x^2}+ \left( \frac{2\bar{L}_{xy}^2}{\bar{\oldmu}_x}+1\right)\dist^2(\bar{y}, y^\ast) \leq \frac{4\hat{\oldepsilon}}{\bar{\oldmu}_x}+ \frac{2\epsilon}{\bar{\oldmu}_y}\left( \frac{2\bar{L}_{xy}^2}{\bar{\oldmu}_x^2}+1\right).
\end{aligned}
\end{align}
    We used the triangular inequality and Young's in $\circled{1}$ and for the second summand of $\circled{2}$ we used the $(\bar{L}_{xy}/\bar{\oldmu}_x)$-Lipschitzness of $x^\ast(\cdot)$, due to \cref{lem:lip}.

Under the assumption of \cref{item:dist_to_gap_lip}, we have that
\begin{align*}
  \text{gap}(\bar{x},\bar{y})&= g(\bar{x},y^{*}(\bar{x}))-g(\bar{x},\bar{y}) +g(\bar{x},\bar{y})-g(x^{*}(\bar{y}),\bar{y})\\
                             &\le \bar{L}_p^y\dist(y^{*}(\bar{x}),\bar{y}) + \bar{L}_p^x\dist(x^{*}(\bar{y}),\bar{x})\\
                             &\le  \bar{L}_p^y (\dist(y^{*}(\bar{x}),y^{*})+\dist(y^{*},\bar{y})) +\bar{L}_p^x(\dist(x^{*}(\bar{y}),x^{*})+\dist(x^{*},\bar{x}))\\
                             &\circled{1}[\le]  \bar{L}_p^y \left(\frac{\bar{L}_{xy}}{\bar{\oldmu}_y}\dist(x^{*},\bar{x})+\dist(y^{*},\bar{y})\right) +\bar{L}_p^x\left(\frac{\bar{L}_{xy}}{\bar{\oldmu}_x}\dist(\bar{y},y^{*})+\dist(x^{*},\bar{x})\right)\\
  &=  \dist(y^{*},\bar{y})\left(\bar{L}_p^y+\bar{L}_p^x\frac{\bar{L}_{xy}}{\bar{\oldmu}_x}\right) +\dist(x^{*},\bar{x})\left(\bar{L}_p^x+\bar{L}_p^y\frac{\bar{L}_{xy}}{\bar{\oldmu}_y}\right).
\end{align*}
  We used \cref{lem:lip} in $\circled{1}$ above.
\end{proof}

Before we go on to prove \cref{thm:minmax_alg}, we briefly discuss a technical detail.
\begin{remark}[Saddle point assumption]\label{rem:SP-ass}
In \cref{sec:setting} we assume for the sake of clarity that $f$ admits a saddle point $(\xast,\yast)\in \mathcal{X}\times \mathcal{Y}$ satisfying $\nabla f(\xast,\yast)=0$.
However, it is not necessary to assume that the saddle point has zero gradient and a slightly weaker assumption suffices to show our convergence result.
From now on, let $(\xast,\yast)\in \X\times \Y$ be a saddle point in $\X\times \Y$ which does not necessarily have zero gradient, and let $(\newtarget{def:x_global_saddle}{\xastg}, \newtarget{def:y_global_saddle}{\yastg}) \in \M \times\NN$, be a global saddle point such that $\nabla \f(\xastg, \yastg) = 0$.
Then, it suffices to assume that $\dist^2(x_0,\xastg)+\dist^2(y_0,\yastg)\le D^2$ .
This allows global saddle points to lie outside $\mathcal{X}\times \mathcal{Y}$.
We also note that in fact, our algorithms can work without any assumption on $\dist^2(x_0,\xastg)+\dist^2(y_0,\yastg)$ by using an upper bound of this distance in our algorithmic parameters.
\end{remark}

\begin{proof}\linkofproof{thm:minmax_alg}
The total number of gradient and metric-projection oracle calls of \cref{alg:ramma} can be calculated as follows
\begin{equation*}
 \TOne(\Tthree\Tfive+\Tfour)+\Ttwo,
\end{equation*}
where $\TOne$ to $\Tfive$ refer to the complexity of the different routines, which are provided in \cref{lem:loop1,lem:loop2,lem:loop3}.
    We provide an overview of the required $\varepsilon_i$, $\hat{\varepsilon}_i$, $T_i$ in \cref{table:ramma-params}.
\begin{table}[]
  \caption{Overview of precision and iteration parameters for
\RAMMA{}}
  \label{table:ramma-params}
\begin{tabular}{@{}ll@{}}
\toprule
 Number of iterations & Required precisions  \\ \midrule
    $\newtarget{def:T_1}{\TOne}=\bigotilde{\sqrt{ \frac{\zetad}{\mux\eta_x}}}$&$\newtarget{def:epsilon_1}{\epsilonone}= \frac{\epsilon\mux}{4C}\left( \frac{2\Lxy^2}{\muy^2}+1 \right)^{-1},\quad   \newtarget{def:hat_epsilon_1}{\hatepsilonone}=\frac{\epsilonone(\eta_x\mux)^{\frac{3}{2}}}{4 \sqrt{\xi}}$  \\
    $\newtarget{def:T_2}{\Ttwo}=\bigotilde{\zetad^{\frac{5}{2}}\sqrt{\zetad +\frac{\Lxy+\Ly}{\muy}}}$&$\newtarget{def:epsilon_2}{\epsilontwo}= \frac{\muy\epsilon}{8C}$ \\
    $\newtarget{def:T_3}{\Tthree}=\bigotildel{\sqrt{\frac{\zetad}{\muy\eta_y}}}$&$\newtarget{def:epsilon_3}{\epsilonthree}=\frac{\muy\epsilonone^2(\mux\eta_x)^3}{64C_k^2\xi \left( \frac{2\Lxy^2}{\mux+\eta\_x^{-1}}+1 \right)}, \quad \hat{\epsilon}_3=\newtarget{def:hat_epsilon_3}{\epsilonthree}(\eta_y\muy)^{-3/2}/(8 \sqrt{\xi})$ \\
    $\newtarget{def:T_4}{\Tfour}=\bigotilde{\zetad^{\frac{5}{2}}\sqrt{\kappa_x+\zetad}}$&$\newtarget{def:epsilon_4}{\epsilonfour}=\frac{(\mux+\eta_x^{-1})(\mux\eta_x)^3\epsilonone^2}{128C_k^2\xi}$ \\
    $\newtarget{def:T_5}{\Tfive}=\bigotilde{\zetad^{3} \sqrt{\Lx\eta_x+\Ly\eta_y+\zetad}}$&$\newtarget{def:epsilon_5}{\epsilonfive}=\frac{\epsilonthree^2(\muy\eta_y)}{32\xi C_{\ell}^2}$  \\\bottomrule
\end{tabular}
\end{table}
We have that
\begin{equation*}
\TOne\Tthree\Tfive=\bigotildel{\zetad^{4}\sqrt{\frac{\Lx}{\mux\muy\eta_y}+\frac{\Ly}{\mux\muy\eta_x}+\frac{\zetad}{\mux\muy\eta_x\eta_y}}}.
\end{equation*}
Recall that we assumed without loss of generality that $\muy \leq \mux$. We analyze some cases now. If $\mux\le \Lxy$, we have that $\eta_x^{-1}=\Lxy+9\xi\mux$, $\eta_y^{-1}=\Lxy+9\xi\muy$ and
\begin{equation*}
\bigol{  \frac{\Lx}{\mux\muy\eta_y}+\frac{\Ly}{\mux\muy\eta_x}+\frac{\zetad}{\mux\muy\eta_x\eta_y}}=\bigol{\frac{\zetad^2 L\Lxy}{\mux\muy}+\zetad^3}.
\end{equation*}
If $\Lxy\le \mux,\muy$, we have that $\eta_x^{-1}=(1+9\xi)\mux$, $\eta_y^{-1}=(1+9\xi)\muy$ and
\begin{equation*}
   \bigol{\frac{\Lx}{\mux\muy\eta_y}+\frac{\Ly}{\mux\muy\eta_x}+\frac{\zetad}{\mux\muy\eta_x\eta_y}}= \bigol{\zetad(\kappa_x+\kappa_y)+\zetad^3}.
\end{equation*}
If $\muy\le \Lxy\le\mux$, we have $\eta_x^{-1}=(1+9\xi)\mux$,$\eta_y^{-1}=\Lxy+9\xi\muy$ and it is
\begin{equation*}
  \bigol{\frac{\Lx}{\mux\muy\eta_y}+\frac{\Ly}{\mux\muy\eta_x}+\frac{\zetad}{\mux\muy\eta_x\eta_y}}=\bigol{ \zetad(\kappa_x+\kappa_y)+ \frac{L\Lxy}{\mux\muy} + \frac{\zetad^2\Lxy}{\min \{\mux,\muy\}} +\zetad^3}.
\end{equation*}
All in all the worst case complexity is
\begin{equation*}
  \bigol{\frac{\Lx}{\mux\muy\eta_y}+\frac{\Ly}{\mux\muy\eta_x}+\frac{\zetad}{\mux\muy\eta_x\eta_y}}= \bigol{ \frac{\zetad^2 L\Lxy}{\mux\muy}+  \zetad(\kappa_x+\kappa_y) + \zetad^3}.
\end{equation*}
Hence
\begin{equation}\label{eq:t1t3t5}
\TOne\Tthree\Tfive=\bigotilde{\zetad^{9/2}\sqrt{\frac{\zetad L\Lxy}{\mux\muy}+ \kappa_x+\kappa_y + \zetad^2}}.
\end{equation}
Further, we have
\begin{equation*}
 \TOne\Tfour = \bigotildel{\zetad^3 \sqrt{\frac{\kappa_x+\zetad}{\mux\eta_x}}}.
\end{equation*}
It holds that
\begin{equation*}
\bigol{  \frac{\kappa_x+\zetad}{\mux\eta_x}}=\bigol{ \zetad^2+\kappa_x\zetad+ \frac{\Lx\Lxy}{\mux^2}+\frac{\zetad \Lxy}{\mux}},
\end{equation*}
and hence
\begin{equation}\label{eq:t1t4}
 \TOne\Tfour = \bigotilde{\zetad^3 \sqrt{ \zetad^2+\kappa_x\zetad+ \frac{\Lx\Lxy}{\mux^2}+\frac{\zetad \Lxy}{\mux}}}.
\end{equation}
Finally
\begin{equation}\label{eq:t2}
\Ttwo=\bigotilde{\zetad^{\frac{5}{2}}\sqrt{\zetad +\frac{\Lxy+\Ly}{\muy}}}.
\end{equation}
Using \cref{eq:t1t3t5,eq:t1t4,eq:t2}, we conclude that
\begin{equation*}
 T = \bigotilde{\zetad^{9/2}\sqrt{\frac{\zetad L\Lxy}{\mux\muy}+ \kappa_x+\kappa_y + \zetad^2}},
\end{equation*}
 where we used $\muy\le \mux$, which we recall that was assumed to hold without loss of generality. We note that the dependence on $\epsilon$ is $\log^3(\epsilon^{-1})$ and the $\log$ contains a polylog expression on $\D$ and the smoothness and strong convexity constants of $f$.
\end{proof}

\begin{lemma}[Guarantees of Lines \ref{line:loop1-start}-\ref{line:loop1-end}]\label{lem:loop1}
Running Lines \ref{line:loop1-start}-\ref{line:loop1-end} of \cref{alg:ramma} with $\TOne=\bigotilde{\sqrt{ \frac{\zetad}{\mux\eta_x}}}$ and $\Ttwo=\bigotilde{\zetad^{\frac{5}{2}}\sqrt{\zetad +\frac{\Lxy+\Ly}{\muy}}}$ ensures that $\text{gap}(\hat{x},\hat{y})\le \epsilon$.
\end{lemma}
\begin{proof}
We show the lemma by first finding sufficient error criteria $\epsilonone$ and $\epsilontwo$ for obtaining $\text{gap}(\hat{x},\hat{y})\le \epsilon$ and then we compute the number of iterations $\TOne$ and $\Ttwo$  required to achieve these error criteria.
\paragraph{Error criterion}
Let $f_x\defi f(\cdot,y)$ and $f_y=f(x,\cdot)$, then using \cref{item:dist_to_gap_lip} of \cref{proposition:from_one_opti_measure_to_another}, we have that
\begin{equation}\label{eq:gap-loop1}
  \begin{aligned}
    \text{gap}(\hat{x},\hat{y})&\le \dist(\hat{x},x^{*}) (\Lips(f_x)+ \frac{\Lxy}{\muy}\Lips(f_y)) + \dist(\hat{y},y^{*}) (\Lips(f_y)+ \frac{\Lxy}{\mux}\Lips(f_x))\\
    &\le C(\dist(\hat{x},x^{*}) +\dist(\hat{y},y^{*}))
  \end{aligned}
\end{equation}
where $\Lips(f_x)$ and $\Lips(f_y)$ denote the Lipschitz constant of $f_x$ and $f_y$ respectively and
\begin{equation*}
  C=\max\{\Lips(f_x)+ \frac{\Lxy}{\muy}\Lips(f_y), \Lips(f_y)+ \frac{\Lxy}{\mux}\Lips(f_x)\}.
\end{equation*}
Recall that by assumption, we have that $\nabla_xf(\xastg,\yastg)=\nabla_yf(\xastg,\yastg)=0$.
We leverage this fact in order to bound the Lipschitz constants.
We have for some $x\in \mathcal{X}$
\begin{equation}\label{eq:grad_norm_y}
\begin{aligned}
\Lips(f_y)= \max_{y\in \Y}  \norm{\nabla_yf_y(y)}&=\max_{y\in  \Y}\norm{\nabla_yf(x,y)\pm \nabla_yf(\xastg,y )-\Gamma{\yastg}{y}\nabla_yf(\xastg,\yastg)}\\
  &\le \max_{y\in \mathcal{Y}}\Ly \dist(\yastg,y)+ \Lxy \dist(x,\xastg)\le D(\Ly+\Lxy)
\end{aligned}
\end{equation}
and similarly
\begin{equation}\label{eq:grad_norm_x}
\Lips(f_y)= \max_{x\in \X}  \norm{\nabla_xf_x(x)}\le D( \Lx+\Lxy).
\end{equation}
    Since $\hat{x}$ is an $\epsilonone$-minimizer of the problem $\min_{x\in \X}\phi(x)$ and $\hat{y}$ is an \(\epsilontwo\)-minimizer of the problem $\min_{y\in \Y}-f(\hat{x},y)$,  \cref{item:one_gap_to_dist} of \cref{proposition:from_one_opti_measure_to_another} implies that
  \begin{equation}\label{eq:dist-loop1}
  \dist^2(\hat{x},x^{*})+ \dist^2(\hat{y},y^{*})\le \frac{4\epsilontwo}{\muy}+\frac{2\epsilonone}{\mux} \left( \frac{2\Lxy^2}{\muy^2}+1 \right).
\end{equation}
Using \cref{eq:dist-loop1}, we obtain
\begin{equation*}
   \text{gap}(\hat{x},\hat{y})\le C\left( \frac{4\epsilontwo}{\muy}+\frac{2\epsilonone}{\mux} \left( \frac{2\Lxy^2}{\muy^2}+1 \right)\right).
 \end{equation*}
It suffices to choose
\begin{equation}\label{eq:eps12}
  \epsilonone\le \frac{\epsilon\mux}{4C}\left( \frac{2\Lxy^2}{\muy^2}+1 \right)^{-1}, \quad \epsilontwo\le \frac{\muy\epsilon}{8C},
\end{equation}
in order to ensure that $\text{gap}(\hat{x},\hat{y})\le \epsilon$.
\paragraph{Complexity}
By \cref{thm:riemaconsc}, computing $\hat{x}$ with $\riemacon$ takes, by the choice of the corresponding $\epsilonone$, computed in \cref{eq:eps12}:
\begin{equation*}
 \TOne= \bigotildel{\sqrt{ \frac{\zetad}{\mux\eta_x}}}.
\end{equation*}
Using \cref{thm:riemaconsc} and by definition of $\lambda_y$, we have that running $\riemacon$ in Line \ref{line:getting_distance_to_yast_outer_loop} requires
\begin{equation*}
 \Ttwo''=\bigotildel{\sqrt{\zetad\frac{\Lxy+\Ly}{\muy} + \zetad^2}}
\end{equation*}
iterations.
Further, computing a $\sigma$-minimizer of $\min_{y \in \Y} \hat{F}_y(y)$, where $\hat{F}_y(y)\defi -f(\hat{x},y)+\frac{1}{2\lambda_y}\dist^2(y,\bar{y})$ using \PRGD{} costs $ \Ttwo'=\bigotilde{\tilde{\kappa}\zetad_{R}}$, where
\begin{equation*}
  \tilde{\kappa}\defi\frac{\Ly}{\muy+\lambda_y^{-1}} +\frac{\zetad/\lambda_y}{\muy+\lambda_y^{-1}}\le 1+\zetad
\end{equation*}
is the condition number of $\hat{F}_y$.
Note we used $\lambda_y^{-1} = (\max \{\Lxy,\Ly\} + 9\xi\mu_y) \geq \Ly$.
We now show that $R\le 2D$ in order to bound $\zetad_{R}\le \zeta[2D] = \bigo{\zetad}$.
We bound $R\le \Lips(\hat{F}_y)/(\Ly+\lambda_y^{-1})$, where $\Lips(\hat{F}_y)$ is the Lipschitz constant of $\hat{F}_y(y)$ for $y\in \Y$ and $\bar{y}\in \Y$.
Note that $\nabla_yf(\xastg,\yastg)=0$ by assumption.
Hence, for all $x \in \X$
\begin{align*}
 \Lips(\hat{F}_y)&\le\max_{y\in \Y}\norm{\nabla_y\hat{F}(y)}\\
                           &=\max_{y\in \Y}  \norm{-\nabla_yf(\hat{x},y)-\lambda_y^{-1}\log_y(\bar{y})+\Gamma{\yastg}{y}\nabla_yf(\xastg,\yastg)}\\
                      &\le \max_{y\in \Y} \norm{-\nabla_yf(\hat{x},y)\pm\nabla_yf(\xastg,y)+\Gamma{\yastg}{y}\nabla_yf(\xastg,\yastg)} + \max_{y\in \Y}\lambda_y^{-1}\dist(y,\bar{y})\\
                &\le( \Ly+ \Lxy+\lambda_{y}^{-1})\D.
\end{align*}
And thus it holds that $R\le 2D$.
Therefore, the total complexity of computing $\hat{y}$ is
\begin{equation*}
  \Ttwo=\Ttwo'\Ttwo'' = \bigotildel{\zetad^{\frac{5}{2}}\sqrt{\zetad +\frac{\Lxy+\Ly}{\muy}}}.
\end{equation*}
\end{proof}

\begin{lemma}[Guarantees of Lines \ref{line:riemacon_on_psi}-\ref{line:getting_distance_to_xkast_second_loop}]\label{lem:loop2}
  Running Lines \ref{line:riemacon_on_psi}-\ref{line:getting_distance_to_xkast_second_loop} of \cref{alg:ramma} with  $\Tthree=\bigotildel{\sqrt{\frac{\zetad}{\muy\eta_y}}}$ and $\Tfour=\bigotilde{\zetad^{\frac{5}{2}}\sqrt{\kappa_x+\zetad}}$ ensures that $\text{gap}_k(\tilde{x}_k)\le \hatepsilonone$.
\end{lemma}
\begin{proof}
We show the lemma by first finding sufficient error criteria $\epsilonthree$ and $\epsilonfour$ for obtaining  $\text{gap}_k(\tilde{x}_k)\le \hatepsilonone$ and then computing the number of iterations $\Tthree$ an $\Tfour$  required to achieve these error criteria.
\paragraph{Error criterion}
  Let $G_x(x)\defi f(x,y)+\frac{1}{2\eta_x}\dist^2(x_k,x)$ and $G_y(y)\defi f(x,y)+\frac{1}{2\eta_x}\dist^2(x_k,x)$, then we have
\begin{align}\label{eq:aux:gap_bounding}
 \begin{aligned}
   \text{gap}_k(\tilde{x}_k)&\circled{1}[\le]\text{gap}_k(\tilde{x}_k,\tilde{y}_{k})\\
                            &\circled{2}[\le] \dist(\hat{x},x^{*}_k) \left(  \Lips(G_x)+\frac{\Lxy}{\mux}\Lips(G_y)\right)+ \dist(\hat{y},y^{*}_k) \left( \Lips(G_x)\frac{\Lxy}{\muy}+\Lips(G_{y}) \right)\\
                            & \circled{3}[\le] C_k (\dist(\hat{x},x^{*}_k)+\dist(\hat{y},y^{*}_k)) \\
                            &\circled{4}[\le] C_k\sqrt{2(\dist^2(\hat{x},x^{*}_k)+\dist^2(\hat{y},y^{*}_k))}.
   \end{aligned}
 \end{align}
    Here $\circled{1}$ and $\circled{2}$ hold by \cref{item:full_gap_to_individual_gap,item:dist_to_gap_lip} in \cref{proposition:from_one_opti_measure_to_another}, respectively, and $\circled{3} $ holds with
 \begin{equation*}
  C_k=\max \left\{ \Lips(G_x)+\frac{\Lxy}{\mux}\Lips(G_y),\Lips(G_x)\frac{\Lxy}{\muy}+\Lips(G_{y}),  \right\}.
\end{equation*}
Finally, $\circled{4}$ follows from $a+b\le \sqrt{2(a^2+b^2)}$.
 We bound the Lipschitz constant of $G_x$ by bounding the following, for all $x\in \X$
 \begin{align*}
   \begin{aligned}
       \norm{\nabla_x F(x,y)} =  \norm{\nabla_x f(x,y) \pm \nabla_x f(x,\hat{y}^\ast) -\frac{1}{\eta_x}\log_x(x_k)-\Gamma{\xastg}{x}\nabla_{x}f(\xastg,\yastg)}\le D(\Lx+\Lxy+\eta_x^{-1}),
   \end{aligned}
\end{align*}
    and thus $\Lips(G_x)\le D(\Lx+\Lxy+\eta_x^{-1})$. Similarly, we obtain $\Lips(G_y)\le D(\Ly+\Lxy)$. Since $\tilde{y}_k$ is an $\epsilonthree$-minimizer of the problem $\min_{y\in \Y}\psi(y)$ and $\tilde{x}_k$ is an $\epsilonfour$-minimizer of the problem $\min_{x\in \X}\{ f(x,\tilde{y}_k) +\frac{1}{2\eta_x}\dist^2(x_k,x)\}$, \cref{item:one_gap_to_dist} in \cref{proposition:from_one_opti_measure_to_another} implies that
  \begin{equation}\label{eq:loop2-dist}
 \dist^{2}(\tilde{x}_k,x_k^{*})+\dist^2(\tilde{y}_k,y_k^{*})\le \frac{4\epsilonfour}{\mux+\eta_x^{-1}}+\frac{2\epsilonthree}{\muy}\left( \frac{2\Lxy^2}{(\mux+\eta_x^{-1})^2}+1 \right).
  \end{equation}
    Then, using \cref{eq:loop2-dist} and \cref{eq:aux:gap_bounding}, we have that
 \begin{equation*}
 \text{gap}_k(\tilde{x}_k)\le  \sqrt{2}C_k\sqrt{\frac{4\epsilonfour}{\mux+\eta_x^{-1}}+\frac{2\epsilonthree}{\muy}\left( \frac{2\Lxy^2}{(\mux+\eta_x^{-1})^2}+1 \right)}.
 \end{equation*}
 We have that $\hatepsilonone=\frac{\epsilonone(\eta_x\mux)^{\frac{3}{2}}}{4 \sqrt{\xi}}$
 Hence, choosing
 \begin{equation}\label{eq:eps34}
  \epsilonthree=\frac{\muy\epsilonone^2(\mux\eta_x)^3}{64C_k^2\xi \left( \frac{2\Lxy^2}{\mux+\eta_x^{-1}}+1 \right)},\quad \epsilonfour  = \frac{(\mux+\eta_x^{-1})(\mux\eta_x)^3\epsilonone^2}{128C_k^2\xi}.
 \end{equation}
 suffices to satisfy $\text{gap}_k(\tilde{x}_k)\le \hatepsilonone$.
\paragraph{Complexity}
    By \cref{thm:riemaconsc}, and using $\epsilonthree$ computed in \cref{eq:eps34}, we have that computing $\tilde{y}_k$ takes
 \begin{equation*}
  \Tthree=\bigotildel{\sqrt{\frac{\zetad}{\muy\eta_y}}}
 \end{equation*}
 iterations.
 Further, by \cref{cor:smooth-riemaconsc}, we have
 \begin{equation*}
  \Tfour = \bigotilde{\zetad_{R}\zetad^{\frac{3}{2}}\sqrt{\tilde{\kappa}_x+\zetad}},
 \end{equation*}
 where $\tilde{\kappa}_x$ is the condition number of $G_x(x)$.
 Further, let $\hat{F}_x(x)\defi f(x,\hat{y})+\frac{1}{2\eta_x}\dist^2(x_k,x)+\frac{1}{2\lambda_x}\dist^2(x,\bar{x})$, where $\lambda_x=(\Lx+\zetad\eta_x^{-1}+9\xi\mux)$.
  We bound the Lipschitz constant $\Lips(\hat{F}_x)$ for all $x\in \X$ by bounding
  \begin{equation}
    \begin{aligned}
        \norm{\nabla\hat{F}_x(x)} &\leq\norm{\nabla_xf(x,\hat{y}) -\eta_x^{-1}\log_x(x_k)-\lambda_x^{-1}\log_x(\bar{x})-\Gamma{\xastg}{x}\nabla_xf(\xastg,\yastg)}\\
      & \le D(\Lx+\Lxy+\eta_x^{-1}+ \lambda_x^{-1}).
    \end{aligned}
  \end{equation}
Thus $\Lips(\hat{F}_x)\le D(\Lx+\Lxy+\eta_x^{-1}\lambda_x^{-1})$.
  Hence,
  \begin{equation*}
   R  = \frac{\Lips(\hat{F}_x)}{\Lx+\zetad\eta_x^{-1}+\zetad\lambda_x^{-1}}\le \frac{\dist	(\Lx+\Lxy+\eta_x^{-1}+\lambda_x^{-1})}{\Lx+\zetad\eta_x^{-1}+\zetad\lambda_x^{-1}}\le 2D.
  \end{equation*}
  And so $\zeta[R] \leq \zeta[2D] = \bigo{\zetad}$. Further, the condition number of $G_x(x)$ can be bounded by
  \begin{equation*}
  \tilde{\kappa}_x= \frac{\Lx}{\mux+\eta_x^{-1}+\lambda_x^{-1}} + \frac{\zetad(\eta_x^{-1}+\lambda_x^{-1})}{\mux+\eta_x^{-1}+\lambda_x^{-1}}\le 1+  \zetad .
\end{equation*}
Finally, we obtain
 \begin{equation*}
  \Tfour = \bigotilde{\zetad^{5/2}\sqrt{\kappa_x+\zetad}}.
 \end{equation*}
\end{proof}

\begin{lemma}[Guarantees of Lines \ref{line:rabr_innermost_loop}-\ref{line:loop3-end}]\label{lem:loop3}
    Let $\operatorname{gap}_{\ell}$ refers to the gap of the problem $\min \max\{ f(x, y) + \frac{1}{2\eta_x}\dist^2(x_k, x)-\frac{1}{2\eta_y} \dist^2(y_\ell, y) \}$. Running Lines \ref{line:rabr_innermost_loop}-\ref{line:loop3-end} of \cref{alg:ramma} with $\Tfive=\bigotilde{\zetad^{3} \sqrt{\Lx\eta_x+\Ly\eta_y+\zetad} }$ ensures that $\gap[\bar{y}_{\ell}][\ell]\leq  \hatepsilonthree$.
\end{lemma}

\begin{proof}
    We show the lemma by first finding a sufficient error criterion $\epsilonfive$ for obtaining $\gap[\bar{x}_{\ell}, \bar{y}_{\ell}][\ell]\le \hatepsilonthree$, and then we computing the number of iterations $\Tfive$ required to achieve it. This bound implies the result since by \cref{item:full_gap_to_individual_gap} of \cref{proposition:from_one_opti_measure_to_another}, it is  $\gap[\bar{y}_{\ell}][\ell]\leq \gap[\bar{x}_{\ell}, \bar{y}_{\ell}][\ell]$.
\paragraph{Error criterion}
Write $h(x, y)\defi f(x,y)+\frac{1}{2\eta_x}\dist^2(x,x_k)-\frac{1}{2\eta_y}\dist^2(y,y_{\ell})$.
    Then, we can bound the Lipschitz constant of $h(\cdot, \bar{y}_\ell)$ as
\begin{equation}
  \label{eq:lip_hx}
\begin{aligned}
    \Lips(h(\cdot, \bar{y}_\ell))&=\max_{x\in \X} \norm{\nabla_x h(x, \bar{y}_\ell)}\\
  &=\max_{x\in \X}\norm{\nabla_x f(x,\bar{y}_\ell)-\eta_x^{-1}\log_x(x_k) \pm \nabla_xf(x,\yastg)-\Gamma{\xastg}{x}\nabla_xf(\xastg,\yastg)}\\
  &\le D(\eta_x^{-1}+\Lx+\Lxy ).
\end{aligned}
\end{equation}
    and similarly, for any $x\in \X$, we bound the Lipschitz constant of $h(\bar{x}_\ell, \cdot)$ as follows
\begin{equation}\label{eq:lip_hy}
    \Lips(h(\bar{x}_\ell, \cdot)) \le \max_{y\in \Y} \norm{\nabla_y h(\bar{x}_\ell, y)}\le D	(\eta_y^{-1}+\Ly+\Lxy ).
\end{equation}
we have that for 
\[
    C_{\ell}=\D\max \left\{ \frac{\Lxy(\eta_x^{-1}+\Lx+\Lxy)}{\mux} +\eta_y^{-1}+\Ly+\Lxy,\frac{\Lxy(\eta_y^{-1}+\Ly+\Lxy)}{\muy}+\eta_x^{-1}+\Lx+\Lxy \right\}, \\
\] 
and defining
\[
\epsilonfive\defi\frac{\epsilonthree^2(\muy\eta_y)}{32\xi C_{\ell}^2},
\] 
the following holds, as desired:
\begin{equation}
  \label{eq:gap_ell}
\begin{aligned}
    &\gap[\bar{x}_{\ell},\bar{y}_{\ell}][\ell]\\
                             &\ \ \circled{1}[\le] \dist(\bar{x}_{\ell},x^{*}_{\ell}) \left( \frac{\Lxy}{\mux} \Lips(h(\cdot, \bar{y}_\ell)) +\Lips(h(\bar{x}_\ell, \cdot)) \right)+\dist(\bar{y}_{\ell},y^{*}_{\ell}) \left( \frac{\Lxy}{\muy}\Lips(h(\cdot, \bar{x}_\ell))+\Lips(h(\bar{y}_\ell, \cdot)) \right)\\
                             & \ \ \circled{2}[\le] C_{\ell}[\dist(\bar{x}_{\ell},x^{*}_{\ell})+\dist(\bar{y}_{\ell},y^{*}_{\ell})] \circled{3}[\leq] C_{\ell}\sqrt{2\epsilonfive} \circled{4}[\leq]  \epsilonthree(\eta_y\muy)^{-3/2}/(8 \sqrt{\xi})\circled{5}[=] \hatepsilonthree.
\end{aligned}
\end{equation}
    We used \cref{item:dist_to_gap_lip} in \cref{proposition:from_one_opti_measure_to_another} for $\circled{1}$ and the definition of $C_{\ell}$ in $\circled{2}$.
    \cref{thm:crabr} implies $\dist^2(\bar{x}_{\ell},x^{*}_{\ell})+\dist^2(\bar{y}_{\ell},y^{*}_{\ell})\le \epsilonfive$ which was used for $\circled{3}$ along with $(a+b)^2 \leq 2a^b +2b^2$. We defined $\epsilonfive$ in order to satisfy $\circled{4}$. The definition of the accuracy $\hatepsilonthree$ that we require in $\riemacon$ was used in $\circled{5}$.
\paragraph{Complexity}
By \cref{thm:crabr}, and the definition of $\epsilonfive$, computing $\bar{y}_{\ell}$ takes
\begin{equation*}
 \Tfive=\bigotilde{\zetad_{R}\zetad^{2} \sqrt{\tilde{\kappa}_x+\tilde{\kappa}_y+\zetad}},
\end{equation*}
    iterations of \RABR{}, where $\tilde{\kappa}_x$ is the condition number of $h(\cdot, \bar{y}_\ell)$ and $\tilde{\kappa}_y$ is the condition number of $h_y$.
Using \cref{eq:lip_hx} and the definition of $\eta_x$, we have
\begin{equation*}
    R\le \frac{\max_{y\in\Y}\Lips(h(\cdot, y))}{\Lx+\frac{\zetad}{\eta_x}}\le \frac{D	(\eta_x^{-1}+\Lx+\Lxy )}{\Lx+\frac{\zetad}{\eta_x}} \le 2D
\end{equation*}
and similarly, by \cref{eq:lip_hy} and the definition of $\eta_y$ it holds
\begin{equation*}
R\le  \frac{\max_{x\in\X}\Lips(h(x, \cdot))}{\Ly+\frac{\zetad}{\eta_y}}\le 2D.
\end{equation*}
    Hence, we have $\zeta[R]\le \zeta[2D] = \bigotilde{\zetad}$.
Given that $\tilde{\kappa}_x\le \eta_x\Lx+\zetad$ and $\tilde{\kappa}_y\le \Ly\eta_y+\zetad$, we conclude that
\begin{equation*}
 \Tfive=\bigotilde{\zetad^{3} \sqrt{\Lx\eta_x+\Ly\eta_y+\zetad}}.
\end{equation*}
\end{proof}

\begin{proof}\linkofproof{cor:red-scc}
    Let $(\bar{x}, \bar{y})\in \X \times \Y$ be the initial point of our algorithm, and define the following regularized function
\begin{equation*}
  f_{\epsilon}(x,y)\defi f(x,y) + \frac{\epsilon}{4\D^2} \dist^2(\bar{x},x)- \frac{\epsilon}{4\D^2}\dist^2(\bar{y},y).
\end{equation*}
Interestingly, we require the use of this function for both the \CC{} and the \SCC{} case. That is, we require regularizing both variables even when the function is strongly g-convex with respect to one.
This is done in order to show that the global saddle point of the regularized problem $(\hat{x}^{*}_{\varepsilon},\hat{y}_{\varepsilon}^{*})$ is not further away from the initial point $(\bar{x}, \bar{y})$ than the global saddle point $(\xastg,\yastg)$ of the unregularized problem , i.e. $\dist^2(\bar{x},\hat{x}^{*}_{\varepsilon})+\dist^2(\bar{y},\hat{y}^{*}_{\varepsilon})\le \dist^2(\bar{x},\xastg)+\dist^2(\bar{y},\yastg)=D^2$ which is required in order bound the geometric penalties.
Now let $(\hat{x},\hat{y})$ be an $\epsilon/2$ saddle point of $f_{\epsilon}$, i.e.
\begin{equation*}
 \max_{y\in \Y}f_{\epsilon}(\hat{x},y)-\min_{x\in \X}f(x,\hat{y})\le \frac{\epsilon}{2}.
\end{equation*}
Let $y^{*}(\hat{x})=\argmax_{y\in \Y}f(\hat{x},y)$, then
\begin{align*}
  \max_{y\in \Y}f_{\epsilon}(\hat{x},y)&\ge f_{\epsilon}(\hat{x},y^{*}(\hat{x}))=f(\hat{x},y^{*}(\hat{x})) + \frac{\epsilon}{4\D^2} \dist^2(\bar{x},\hat{x})- \frac{\epsilon}{4\D^2}\dist^2(\bar{y},\yastg) \ge f(\hat{x},y^{*}(\hat{x})) -\frac{\epsilon}{4}.
\end{align*}
Similarly, for $x^{*}(\hat{y})=\argmin_{x\in \X}f(x,\hat{y})$, we have $\min_{x\in \X} f_{\epsilon}(x,\hat{y})\le f(x^{*}(\hat{y}),\hat{y})+ \frac{\epsilon}{4}$.
Combining these inequalities, we conclude
\begin{equation*}
 \max_{y\in \Y}f(\hat{x},y) -\min_{x\in \X}f(x,\hat{y})= f(\hat{x},y^{*}(\hat{x}))-f(x^{*}(\hat{y}),\hat{y})\le \max_{y\in \Y}f_{\epsilon}(\hat{x},y)- \min_{x\in \X} f_{\epsilon}(x,\hat{y}) + \frac{\epsilon}{4} + \frac{\epsilon}{4} \leq \epsilon.
\end{equation*}
Hence if $(\hat{x},\hat{y})$ is an $\epsilon/2$-saddle point of $f_{\epsilon}$ it is an $\epsilon$-saddle point of $f$.
By \cref{lem:dist-bound} and by the definition of $\X$ and $\Y$, we have that the saddle point of $f_{\epsilon}$ satisfies $\dist^2(\bar{x},\hat{x}^{*}_{\varepsilon})+\dist^2(\bar{y},\hat{y}^{*}_{\varepsilon})\le D^2$ which is required to use \cref{alg:ramma} on $f_{\epsilon}$.
Recall that the complexity of the algorithm is
\begin{equation*}
\bigotilde{ \zetad^{9/2}\sqrt{\frac{\zetad \tilde{L}\Lxy}{\tilde{\mu}_x\tilde{\mu}_y}+ \tilde{\kappa}_x+\tilde{\kappa}_y + \zetad^2}},
\end{equation*}
    where the variables noted with a tilde are the constants of $f_{\epsilon}$. 

   We first analyze the \SCC{} case. We have $\tilde{\mu}_x = \mux + \epsilon/(2\D^2)$,  $\tilde{\mu}_y = \epsilon/(2\D^2)$, $\tilde{L}_x \leq \Lx + \zetad \epsilon/(2\D^2)$ and therefore the condition numbers are
\begin{equation*}
    \tilde{\kappa}_{x}= \frac{\Lx+\zetad \frac{\epsilon}{2\D^2}}{\mux+\frac{\epsilon}{2\D^2}}\le \frac{\Lx}{\mux} + \zetad \quad \quad \text{ and } \quad \quad \tilde{\kappa}_{y}= \frac{\Ly+\zetad \frac{\epsilon}{2\D^2}}{\muy+\frac{\epsilon}{2\D^2}}\le \frac{2\Ly\D^2}{\epsilon}+\zetad.
\end{equation*}
Note that $\Lxy$ is not influenced by regularization.
First, assume that $\Lxy\ge \tilde{L}_x$, then $\tilde{L}=\Lxy$ (recall that $\Lx=\Ly$ without loss of generality)
\begin{equation*}
  \frac{\zetad \tilde{L}\Lxy}{\tilde{\mu}_x\tilde{\mu}_y} \le \frac{2\zetad \Lxy^2\D^2}{\mux\epsilon}= \frac{2\zetad L\Lxy\D^2}{\mux\epsilon}.
\end{equation*}
Now assume that $\Lxy\le \tilde{L}_x$, then $\tilde{L}=\Lx+\zetad \frac{\epsilon}{2\D^2}$ and 
\begin{equation*}
  \frac{\zetad \tilde{L}\Lxy}{\tilde{\mu}_x\tilde{\mu}_y} =  \frac{\zetad \Lxy(\Lx+\zetad \frac{\epsilon}{2\D^2})}{\tilde{\mu}_x\tilde{\mu}_y} \le \frac{2\zetad \Lxy\L\D^2}{\mux\epsilon} +\frac{\zetad^2 \Lxy}{\tilde{\mu}_x} \leq  \frac{2\zetad \Lxy\L\D^2}{\mux\epsilon} + \frac{\zetad^2 \L}{\mux}.
\end{equation*}
All in all, we have
\begin{equation*}
    \bigotildel{\frac{\zetad \tilde{L}\Lxy}{\tilde{\mu}_x\tilde{\mu}_y}+ \tilde{\kappa}_x+\tilde{\kappa}_y + \zetad^2} = \bigotildel{\frac{\zetad^2L}{\mux}+ \frac{\D^2}{\epsilon}\left(\frac{\zetad \Lxy\L}{\mux}+ \Ly\right) } =\bigotildel{\zetad\frac{\L}{\mux}\left(\frac{\L\D^2}{\epsilon}+\zetad\right)}.
\end{equation*}
The resulting complexity is
\begin{equation*}
    \bigotildel{\zetad^{9/2}\sqrt{\frac{\zetad^2L}{\mux}+ \frac{\zetad \Lxy\L\D^2}{\mux\epsilon} }} =\bigotildel{\zetad^{11/2}\frac{LD}{\sqrt{\mux\epsilon}}}.
\end{equation*}

Now we proceed to analyze the \CC{} case. We have $\tilde{\mu}_x = \mux + \epsilon/(2\D^2)$,  $\tilde{\mu}_y = \epsilon/(2\D^2)$, $\tilde{L}_x \leq \Lx + \zetad \epsilon/(2\D^2)$ and therefore the condition numbers are
\begin{equation*}
    \tilde{\kappa}_{x}= \frac{\Lx+\zetad \frac{\epsilon}{2\D^2}}{\mux+\frac{\epsilon}{2\D^2}}\le \frac{2\Lx\D^2}{\epsilon}+\zetad \quad \quad \text{ and } \quad \quad \tilde{\kappa}_{y}= \frac{\Ly+\zetad \frac{\epsilon}{2\D^2}}{\muy+\frac{\epsilon}{2\D^2}}\le \frac{2\Ly\D^2}{\epsilon}+\zetad.
\end{equation*}
First, assume that $\Lxy\ge \tilde{L}_x$, then $\tilde{L}=\Lxy$ (recall that $\Lx=\Ly$ without loss of generality)
\begin{equation*}
  \frac{\zetad \tilde{L}\Lxy}{\tilde{\mu}_x\tilde{\mu}_y} \le  \frac{4\zetad L\Lxy\D^4}{\epsilon^2}
\end{equation*}
Now assume that $\Lxy\le \tilde{L}_x$, then $\tilde{L}=\Lx+\zetad \frac{\epsilon}{2\D^2}$ and
\begin{equation*}
  \frac{\zetad \tilde{L}\Lxy}{\tilde{\mu}_x\tilde{\mu}_y} \le \frac{4\zetad \Lxy\Lx\D^4}{\epsilon^2} + \frac{2\zetad^2\Lxy\D^2}{\epsilon}.
\end{equation*}
Together, we have that
\begin{equation*}
 \bigotildel{  \frac{\zetad \tilde{L}\Lxy}{\tilde{\mu}_x\tilde{\mu}_y}+ \tilde{\kappa}_x+\tilde{\kappa}_y + \zetad^2} =\bigotildel{\frac{\zetad L\Lxy\D^4}{\epsilon^2}+\frac{\zetad^2\Lxy\D^2}{\epsilon}+ \frac{\D^2(\Lx+\Ly)}{\epsilon}}.
 \end{equation*}
 Thus, the complexity is bounded by
 \begin{equation*}
      \bigotildel{\zetad^{9/2}\left(\sqrt{\frac{\Lx\D^2}{\epsilon} + \frac{\Ly\D^2}{\epsilon} + \frac{\zetad \Lxy\D^2}{\epsilon}\left(\frac{\L\D^2}{\epsilon}+\zetad \right)}\right) } = \bigotildel{\zetad^{11/2} \frac{\L\D^2}{\epsilon}}.
 \end{equation*}
\end{proof}

\subsection[Convergence of RAMMA-WC]{Convergence of \RAMMAWC}
\label{sec:weakly-g-convex}
When $f(x,y)$ is not g-convex with respect to $x$, finding a saddle point can be intractable. Even if $f(x,\cdot)$ is a constant function, the problem reduces to an non-$g$-convex problem.
However, we can still approximate a stationary point of $f$ through approximating a stationary point of $\phi(x)\defi \max_{y\in\Y} f(x, y)$.
We consider a notion of stationarity based on the gradient of the Moreau envelope of $\phi$ as defined in the following.
\begin{definition}[Moreau envelope]\label{def:moreau-env}
  Let $f:\M\rightarrow \R\cup \{+\infty\}$ be a g-convex, proper and lower semicontinuous function,  where $\M$ is a uniquely geodesic Riemannian manifold of sectional curvature in $\left[\kmin, \kmax \right]$.
  Then the Moreau envelope of $f$ at $x$ with parameter $\eta$ is $M_{\eta f}:\M\rightarrow \R$ defined as
  \[
    M_{\eta f}(x)\defi \inf_{z} \left\{ f(z)+ \frac{1}{2\eta}\dist^2(x,z) \right\}.
  \]
\end{definition}

We note that since $f(\cdot,y)$ is $\Lx$-smooth in $\X$, it is also $\Lx$-weakly g-convex, and thereby $\phi$ is also $\Lx$-weakly g-convex in $\X$ by essentially the same argument in \cref{lem:strong_cvxty_of_max}. The following notion of stationarity is widely used in non-convex optimization, cf. \citep{lin2020near,davis2018stochastic,lin2020on}.

\begin{definition}\label{def:min-stat}
  Consider a $\rho$-weakly g-convex, strongly g-concave function $f:\mathcal{M}\times \mathcal{N} \rightarrow \mathbb{R}$ and g-convex and compact sets $\mathcal{X}\subset \mathcal{M}$, $\mathcal{Y}\subset \mathcal{N}$.
  Further, let $\phi(x)\defi \max_{y\in \mathcal{Y}}f(x,y)$.
  Then we call $\hat{x}\in \mathcal{X}$ an $\varepsilon$-stationary point of $f$ in $\mathcal{X}$, if
  \begin{equation*}
 \norm{\nabla M_{\eta \phi}(\hat{x})}\leqslant \varepsilon,
\end{equation*}
where $\nabla M_{\eta \phi}(x)$ is the gradient of the Moreau envelope of $\phi$ with parameter $\eta$.
\end{definition}
In the following, we introduce our algorithm \newtarget{def:acronym_riemannian_inexact_proximal_point_algorithm}{\RIPPAWC}, which converges to a stationary point of a weakly g-convex function with probability $\geq 2/3$.
\begin{algorithm}
  \caption{Riemannian Inexact Proximal Point Algorithm  \RIPPAWC($f$, $x_0$, $T$ or $\varepsilon$, $\eta$, $\mathcal{X}$, \texttt{subroutine})}
    \label{alg:wc-rppa}
\begin{algorithmic}[1]
    \REQUIRE Uniquely geodesic compact set $\X\subset \M$, function $f:\M \rightarrow\mathbb{R}$ that is $\rho$-weakly g-convex in $\X$, initialization $x_0\in \X\subset \M$, $T$ (or if $\epsilon$ is provided, compute $T$ with the convergence rates), $\eta\in [0,\rho^{-1})$.
    \vspace{0.1cm}
    \hrule
    \vspace{0.1cm}
    \State $\sigma\gets \eta\varepsilon^2 /(24(1+(\mu\eta)^{-1}))$, with $\mu=-\rho+\eta^{-1}$
    \FOR {$t = 1 \text{ \textbf{to} } T$}
    \State $x_{t+1}\gets \sigma\text{-minimizer of } x\mapsto \{f(x)+\frac{1}{2\eta}\dist^2(x,x_t)\}$
    \ENDFOR
    \State Sample $\tau$ from $\{1,\ldots, T\}$ uniformly
    \ENSURE $x_{\tau}$
\end{algorithmic}
\end{algorithm}

\begin{theorem}\label{thm:wc-rppa}
  Consider a function $f: \mathcal{M} \rightarrow \mathbb{R}$, where $\mathcal{M}$ is a Riemannian manifold of sectional curvature in $\left[\kappa_{\min }, \kappa_{\max }\right]$. Let $\mathcal{X} \subset \mathcal{M}$ be a uniquely geodesic subset of $\mathcal{M}$ and let $f$ be $\rho$-weakly g-convex in $\mathcal{X}$. Then, after  
  \[
    T=\bigol{\frac{f(x_0)-\min_{x\in \mathcal{X}}f(x)}{\varepsilon^2 \eta}}
  \]
  iterations, \cref{alg:wc-rppa} outputs $\hat{x}\in \mathcal{X}$ such that an $\norm{\nabla M_{\eta f}(\hat{x})}\le \varepsilon$ with probability at least $2/3$.
\end{theorem}

\begin{proof}
We write $x_{t+1}^*$ for the exact optimizer of $h_t(x) \defi f(x)+\frac{1}{2 \eta} \dist^2\left(x, x_t\right)$, hence
\begin{equation}
  \label{eq:9}
\begin{aligned}
f(x_{t+1}^*)+\frac{1}{2 \eta} \dist^2(x_{t+1}^*, x_t) = \min _{x \in \mathcal{X}}\{f(x)+\frac{1}{2 \eta} \dist^2(x, x_t)\} & \leqslant f(x_t) \\
    \left(f(x_{t+1}^*)+\frac{1}{2 \eta} \dist^2(x_{t+1}^*, x_t) \right)-\left(f(x_{t+1})-\frac{1}{2 \eta} \dist^2(x_{t+1}, x_t) \right)& \leqslant f(x_t)-\frac{1}{2 \eta} \dist^2(x_{t+1}, x_t) -f(x_{t+1}) \\
\frac{1}{2 \eta} \dist^2(x_{t+1}, x_t) & \circled{1}[\leqslant] f(x_t)-f(x_{t+1})+\sigma
\end{aligned}
\end{equation}
where $\circled{1}$ is derived from the inequality above by using the inexactness criterion in the definition of $x_{t+1}$.
    Note that $h_t(x)$ is $\mu$-strongly convex with $\mu \defi-\rho+\eta^{-1}$, hence we have
\begin{equation}
  \label{eq:16}
\dist^2\left(x_{t+1}, x_{t+1}^*\right) \leqslant \frac{2}{\mu}\left(h_t(x_{t+1})-h_t(x_{t+1}^*)\right) \leqslant \frac{2 \sigma}{\mu} .
\end{equation}
Using the fact that $(a+b)^2\le 2a^2+2b^2$ as well as \cref{eq:16,eq:9} we conclude
\begin{align*}
  \dist^2(x_t,x_{t+1}^{*}) &\le 2 \dist^2(x_t,x_{t+1})+ 2d^2(x_{t+1},x_{t+1}^{*})\\
  &\le 4\eta [ f(x_t)-f(x_{t+1})+\sigma] + \frac{4\sigma}{\mu}.
\end{align*}
Summing the previous inequality from $t=0, \ldots, T-1$ and dividing by $T\eta^2$, we obtain
\begin{align*}
  \frac{1}{\eta^2T}\sum_{t=0}^{T-1}\dist^2(x_t,x_{t+1}^{*})&\le \frac{4 (f(x_0)-f(x_T))}{\eta T} + \frac{4\sigma}{\eta}(1+\frac{1}{\eta\mu})\\
  &\circled{1}[\le] \frac{4(f(x_0)-f(x^{*}))}{\eta T} + \frac{\varepsilon^2}{6},
\end{align*}
Here $\circled{1}$ follows by the definition of $\sigma$ and by $f(x^{*})\le f(x_T)$ with  $x^{*}\in\argmin_{x\in \mathcal{X}}f(x)$.
Choosing $x_\tau$ uniformly from $\left\{x_t\right\}_{t \in[T]}$, we can write
\begin{equation*}
\frac{1}{T} \sum_{t=0}^{T-1} \dist^2(x_t, x^{*}_{t+1})=\mathbb{E}\left[\dist^2(x_\tau, x_{\tau+1}^{*})\right] .
\end{equation*}
By Markov's inequality we have that
\[
  \mathbb{P}\left(\dist^2(x_\tau, x_{\tau+1}^{*})\leq 3 \mathbb{E}[\dist^2(x_\tau, x_{\tau+1}^{*})]\right)\ge \frac{2}{3}.
\]
Hence, for $T\ge\frac{24(f(x_0)-f(x^{*}))}{\varepsilon^2\eta}$, we have with probability at least $2/3$ that,
\begin{equation*}
  \frac{1}{\eta^2}\dist^2(x_\tau, x_{\tau+1}^{*}) \leqslant \frac{3}{\eta^2T}\sum_{t=0}^{T-1}\dist^2(x_t,x_{t+1}^{*})  \leq \frac{12(f(x_0)-f(x^{*}))}{\eta T} + \frac{\varepsilon^2}{2} \le \varepsilon^2,
\end{equation*}
and hence $ \frac{1}{\eta}\dist(x_\tau, x_{\tau+1}^{*}) \le \varepsilon$.
Note that $\nabla M_{\eta f}(x_{\tau})=-\frac{1}{\eta}\exponinv{x_{\tau}}(x^{*}_{\tau+1})$.
Hence, since $\dist(x_\tau, x_{\tau+1}^{*})= \norm{\exponinv{x_{\tau}}(x^{*}_{\tau+1})}=\eta \norm{\nabla M_{\eta f}(x_{\tau})}$ we have that $\norm{\nabla M_{\eta f}(x_{\tau})}\le \varepsilon$ with probability at least $2/3$, which concludes the proof.
\end{proof}
Using \RIPPAWC{}, we approximate a stationary point of $\phi(x) \defi \max_{y\in\Y} f(x,y)$. However, solving the resulting proximal problems can be challenging.
In the following, we introduce \RAMMAWC{}, which specifies how to implement prox subroutines similarly to \RAMMA.
\begin{definition}[RAMMA-WC]\label{def:ramma-wc}
  Consider an algorithm which consists of running RAMMA with the following modifications:
    Set $\eta_x \defi 1 /(3 \max \{\Lxy, \Lx\})$, $T_1\defi\bigotilde{\frac{\Deltazero}{\eta_x\epsilon_1^2}}$ with $\Deltazero\defi\phi(x_0)-\min_{x\in \mathcal{X}}\phi(x)$. Replace Line \ref{line:riemacon_on_phi} in RAMMA with $\text{\RIPPAWC}(\phi(x), x_0, T_1, \eta_x,\text{Lines \ref{line:first_line_intermediate_loop}-\ref{line:last_line_intermediate_loop}})$. Note we are still using the same subroutines as for \RAMMA{}.
    Remove Line \ref{line:getting_distance_to_yast_outer_loop}, we only return the $x$ point. 
  We refer to this algorithm as \newtarget{def:acronym_riemannian_accelerated_min_max_algorithm-wc}{\RAMMAWC}.
\end{definition}

\begin{proof}\linkofproof{thm:ramma-wc}
Note that any $\bar{L}$-smooth function is also $\bar{L}$-weakly g-convex, hence, we have that  $\phi$ is at most $\Lx$-weakly g-convex.
While it does not improve the final complexity, we use $\rho\le \Lx$ as the weak convexity constant of $\phi$ for the sake of generality.
  In RAMMA-WC, we run \cref{alg:wc-rppa} on $\phi(x)=\max _{y \in \mathcal{Y}} f(x, y)=f(x, y^*(x))$ for $T_1$ iterations.
    Hence, we have from \cref{thm:wc-rppa} that the output $\hat{x} \in \mathcal{X}$ of \cref{alg:wc-rppa} satisfies the following, with probability at least $2/3$:
\begin{equation*}\label{eq:exact-stat}
\norm{\nabla M_{\eta\phi}(\hat{x})}\le \varepsilon_1.
\end{equation*}
    In the following, we discuss the complexity of solving the prox subroutine of \cref{alg:wc-rppa} accounting for the changes to the inner loops of RAMMA.
    First, note that the total complexity of RAMMA-WC is $T= T_1(T_3 T_5+T_4)$.
    Here $T_3$ is the complexity of running $\riemacon$ on $\Psi(y)=\max_{x\in \mathcal{X}}-\{f(x,y)-\frac{1}{2\eta_x}\dist^2(x,x_k)\}$, $T_4$ is the complexity of running $\riemacon$ on $x\mapsto f(x,\tilde{y}_k)+\frac{1}{2\eta_x}\dist^2(x_k,x)$ and $T_5$ is the complexity of running \RABR{} on $f(x, y)+\frac{1}{2 \eta_x} \dist^2(x, x_k)-\frac{1}{2 \eta_y} \dist^2(y, y_{\ell})$.
    We have by \cref{thm:wc-rppa} that $T_1\defi\bigol{\frac{\Deltazero}{\epsilon_1^2 \eta_x}}$.
Since $\min_{y\in \mathcal{Y}}\Psi(y)$ is an optimization problem in $y$, its complexity is not affected by changes in $x$. Hence we have $T_3=\bigotilde{\sqrt{\frac{\zeta}{\mu_y \eta_y}}}$ as in the \SCSC{}  case, see \cref{lem:loop2}.
Now consider the min-max problem we want to solve using \RABR{}, i.e.,
\begin{equation*}\label{eq:NC-h}
    \min_{x\in \mathcal{X}} \max_{y\in \mathcal{Y}} \left\{ h(x,y)\defi f(x, y)+\frac{1}{2 \eta_x} \dist^2(x, \bar{x})-\frac{1}{2 \eta_y} \dist^2(y, \bar{y})\right\}.
\end{equation*}
    First note that while $f(\cdot,y)$ is no longer strongly g-convex, $h(\cdot, y)$ is, and the choice of $\eta_x$ ensures that $\bar{L}_{x y} \leqslant \frac{1}{2} \sqrt{\bar{\mu}_x \bar{\mu}_y}$ is still satisfied for $h(x,y)$.
In order to compute the value of $T_5$, we need to know the condition numbers $\kappa_x(h)$, $\kappa_y(h)$ of $h$ with respect to $x$ and $y$, respectively.
    We have $\kappa_{y}(h)=\bigol{\Ly\eta_y+\zeta}$ as in the analysis of the \SCSC{} case, see \cref{lem:loop3}.
    In the following, we show that by our choice of $\eta_x$, $\kappa_x(h)$ and also stays as in the \SCSC{} case, see \cref{lem:loop3},
\begin{equation*}
\kappa_x(h)=\frac{\Lx+\zeta \eta_x^{-1}}{-\rho+\eta_x^{-1}} \leqslant \frac{3}{2}\left(\Lx \eta_x+\zeta\right)=\bigol{\Lx \eta_x+\zeta}.
\end{equation*}
Hence we still have $T_5=\bigotilde{\zeta^3\sqrt{\Lx \eta_x+\Ly \eta_y+\zeta}}$.
In order to obtain the complexity of $T_4$, we consider the following optimization problem,
\begin{equation*}
  \min_{x\in \mathcal{X}} \left\{ F_k(x) \defi f\left(x, \tilde{y}_k\right)+\frac{1}{2 \eta_x} \dist^2\left(x_k, x\right) \right\}.
\end{equation*}
First, note that due to the regularization term, $F_k$ is $(-\rho+\eta_x^{-1})$-strongly convex and can be solved via $\riemacon$ and we have by \cref{lem:loop2} that $T_4=\bigotilde{\zeta^{3.5} \sqrt{\kappa(F_k)+\zeta}}$, where $\kappa(F_k)$ is the condition number of $F_k$.
Noting that
\begin{equation*}
\kappa(F_k)=\frac{\Lx+\eta_x^{-1} \zeta}{-\rho+\eta_x^{-1}}=\bigol{\frac{\Lx}{\max \left\{\rho, \Lxy\right\}}+\zeta},
\end{equation*}
it follows that $T_4=\bigotilde{\zeta^{3.5} \sqrt{\frac{\Lx}{\max \left\{\rho, \Lxy\right\}}+\zeta}}$.
We now go on to compute $T_1T_3T_5$:
\begin{equation*}
    T_1 T_3 T_5=\bigotildel{\frac{\zeta^{3.5} \Deltazero}{\varepsilon^2} \sqrt{\frac{\Lx \eta_x+\Ly \eta_y+\zeta}{\eta_x^2 \eta_y \mu_y}}}=\bigotildel{\frac{\zeta^{3.5} \Deltazero}{\varepsilon^2} \sqrt{\frac{\Lx}{\eta_x \eta_y \mu_y}+\frac{\Ly}{\eta_x^2 \mu_y}+\frac{\zeta}{\eta_x^2 \eta_y \mu_y}}}.
\end{equation*}
Recall that we assume wlog that $\Lx=\Ly$.
\paragraph{Case 1} $\rho \leqslant \Lxy$, and $\mu_y \leqslant \Lxy$, so we have $\eta_y=\frac{1}{\zeta \mu_y+\Lxy}$ and $\eta_x=\Lxy^{-1}$. Then
\begin{equation*}
    \frac{\Lx}{\eta_x \eta_y \mu_y} + \frac{\Ly}{\eta_x^2 \mu_y} + \frac{\zeta}{\eta_x^2 \eta_y \mu_y}  =\left(\zeta \Lx \Lxy+\frac{\Lx \Lxy^2}{\mu_y}\right)  + \left ( \frac{\Lxy^2 \Ly}{\mu_y}\right) + \left( \zeta^2 \Lxy^2+\frac{\zeta \Lxy^3}{\mu_y}\right).
\end{equation*}

\paragraph{Case 2} $\rho \leqslant \Lxy$, and $\mu_y >  \Lxy$, so we have $\eta_y=\frac{1}{\zeta \mu_y}$ and $\eta_x=\Lxy^{-1}$. Then
\begin{equation*}
    \frac{\Lx}{\eta_x \eta_y \mu_y} + \frac{\Ly}{\eta_x^2 \mu_y} + \frac{\zeta}{\eta_x^2 \eta_y \mu_y}  =\left(\zeta \Lx \Lxy\right)  + \left ( \frac{\Lxy^2 \Ly}{\mu_y}\right) + \left( \zeta^2 \Lxy^2\right).
\end{equation*}

\paragraph{Case 3} $\rho>\Lxy, \mu_y \leqslant \Lxy$, so we have $\eta_y=\frac{1}{\zeta \mu_y+\Lxy}$ and $\eta_x=\rho^{-1}$. Then
\begin{equation*}
    \frac{\Lx}{\eta_x \eta_y \mu_y} + \frac{\Ly}{\eta_x^2 \mu_y} + \frac{\zeta}{\eta_x^2 \eta_y \mu_y}  =\left(\frac{\rho \Lx \Lxy}{\mu_y}+\rho \Lx \zeta\right)  + \left (\frac{\Ly \rho^2}{\mu_y} \right) + \left( \zeta^2 \rho^2+\frac{\zeta \Lxy \rho^2}{\mu_y}\right).
\end{equation*}

\paragraph{Case 4} $\rho>\Lxy, \mu_y>\Lxy$, so we have $\eta_y=\frac{1}{\zeta \mu_y}$ and $\eta_x=\rho^{-1}$. Then
\begin{equation*}
    \frac{\Lx}{\eta_x \eta_y \mu_y} + \frac{\Ly}{\eta_x^2 \mu_y} + \frac{\zeta}{\eta_x^2 \eta_y \mu_y}  =\left(\zeta \rho \Lx\right)  + \left (\frac{\rho^2 \Ly}{\mu_y} \right) + \left(\zeta^2 \rho^2 \right).
\end{equation*}
For all cases combined, setting $L=\max \{\Lx, \Ly, \Lxy\}$ and using $\rho \leqslant L$ we have that
\begin{equation*}
\frac{\Lx}{\eta_x \eta_y \mu_y}+\frac{\Ly}{\eta_x^2 \mu_y}+\frac{\zeta}{\eta_x^2 \eta_y \mu_y}=\bigol{\zeta L^2+\frac{\zeta L^3}{\mu_y}} .
\end{equation*}
Which yields
\begin{equation*}
T_1 T_3 T_5=\bigotildel{\frac{\zeta^{4} \Deltazero L}{\varepsilon^2} \sqrt{\zeta+ \frac{L}{\muy}}} .
\end{equation*}
We have shown that the complexity is dominated by $T_1 T_3 T_5$ and the resulting complexity is $T=\bigotilde{\frac{\zeta^{4} \Deltazero L}{\varepsilon^2} \sqrt{\zeta+\frac{L}{\mu_y}}}$, which concludes the proof.
\end{proof}

\subsection{Technical Results}

\begin{lemma}\label{lem:strong_cvxty_of_max}
Let $\M, \NN$ be Riemannian manifolds and let $\X\subset\M$, $\Y \subset\NN$ be g-convex subsets that are uniquely geodesic. Let $f: \X \times \Y \to \R$ be such that $f(\cdot,y)$ is lower semicontinuous, $f(x,\cdot)$ is upper semicontinuous, and $f(x, y)$ is $(\mux, 0)$-\SCSC{} in $\X\times\Y$. Then $\phi(x) \defi \sup_{y\in\Y} f(x, y)$ is $\mux$-strongly g-convex in its domain. Also, if $f$ is sup-compact, $\phi(x)$ is well defined for all $x\in \X$ and it holds $\phi(x)= \max_{y\in\Y} f(x, y)$.
\end{lemma}

\begin{proof}
     Let $x_1, x_2$ be two points in the domain of $\phi$ and let $\gamma$ be the geodesic joining $\gamma(0)=x_1$ and $\gamma(t)=x_2$ with $t\in [0,1]$. Then, we have for all $y\in \Y$ that
     \begin{align*}
       f(\gamma(t),y)& \circled{1}[\le] tf(x_1,y)+(1-t)f(x_2,y) -\frac{t(1-t)\mux}{2}\dist^2(x_1,x_2)\\
         &\circled{2}[\leq] t\phi(x_1)+(1-t)\phi(x_2)-\frac{t(1-t)\mux}{2}\dist^2(x_1,x_2).
     \end{align*}
       Here, $\circled{1}$ holds by $\mux$-strong g-convexity of $f(\cdot,y)$ and $\circled{2}$ uses the definition of $\phi$. Since the inequality holds for all $y$, it also holds for the supremum, proving that $\phi$ is $\mux$ strongly g-convex.

       Now we show that if $f$ is sup-compact for some $\tilde{x} \in \X$, then $\phi(x) = \max_{y\in\Y} f(x, y)$ for all $x\in\X$. To that aim, we show that the superlevel sets of $f(x,\cdot)$ are compact for all $x\in \X$.
     We have that $\{y\in \Y \vert f(\tilde{x},y)\ge \alpha\}$ is compact because $f(\tilde{x},\cdot)$ is sup-compact.
     We have that $\Y_{\cap}=\{y\in \Y\vert f(x,y)\ge \alpha, \forall x\in \X\}=\bigcap_{x\in \X}\{y\in \Y\vert f(x,y)\ge \alpha\}$ is closed because $f(x,\cdot)$ is upper semicontinuous.
     Further, $\Y_{\cap}\subset \{y\in \Y \vert f(\tilde{x},y)\ge \alpha\}$, hence $\Y_{\cap}$ is compact since it is the intersection of a closed and a compact set.
     By the extreme value theorem, an upper semicontinuous function reaches its maximum over a compact set, hence $\phi(x)=\max_{y\in\Y} f(x, y)$.
\end{proof}

\begin{lemma}\linktoproof{lem:dist-bound}\label{lem:dist-bound}
Consider a function $f:\M\times \NN\rightarrow \mathbb{R}$ as described in \cref{sec:setting}.
  Further, let $h(x,y) = f(x,y) + \frac{1}{2\eta}\dist^2(\tilde{x},x) - \frac{1}{2\eta}\dist^2(\tilde{y},y)$  with
  \begin{equation*}
      (\tilde{x}^{*},\tilde{y}^{*}) \defi \argmin_{x\in \M}\argmax_{y\in \NN} h(x,y).
\end{equation*}
Then, $\dist^2(\tilde{x},\tilde{x}^{*})   +\dist^2(\tilde{y},\tilde{y}^{*})\le \dist^2(\tilde{y},\yastg) +\dist^2(\tilde{x},\xastg)$, where $(\xastg,\yastg)$ is the unconstrained saddle point of $f$.
\end{lemma}
\begin{proof}\linkofproof{lem:dist-bound}
Note that by \cref{thm:sion}, $h$ admits an unconstrained saddle point $(\tilde{x}^{*},\tilde{y}^{*})$.
We have that,
\begin{align*}
 & \frac{1}{2\eta}\dist^2(\tilde{x},\tilde{x}^{*})   -\frac{1}{2\eta}\dist^2(\tilde{y},\yastg)+\frac{1}{2\eta}\dist^2(\tilde{y},\tilde{y}^{*}) -\frac{1}{2\eta}\dist^2(\tilde{x},\xastg)  \\
  \le& f(\tilde{x}^{*},\yastg) + \frac{1}{2\eta}\dist^2(\tilde{x},\tilde{x}^{*})   -\frac{1}{2\eta}\dist^2(\tilde{y},\yastg) - f(\xastg,\tilde{y}^{*}) +\frac{1}{2\eta}\dist^2(\tilde{y},\tilde{y}^{*}) -\frac{1}{2\eta}\dist^2(\tilde{x},\xastg)\\
  =& h(\tilde{x}^{*},\yastg)- h(\xastg,\tilde{y}^{*}) \le 0.
\end{align*}
It follows that
\begin{equation*}
  \dist^2(\tilde{x},\tilde{x}^{*})   +\dist^2(\tilde{y},\tilde{y}^{*})\le \dist^2(\tilde{y},\yastg) +\dist^2(\tilde{x},\xastg).
\end{equation*}
\end{proof}
\begin{proposition}\linktoproof{prop:vi}\label{prop:vi}
  Consider a g-convex function $f: \X\rightarrow \mathbb{R}$, where $\X\subset\M$ is a compact convex subset of a Riemannian Manifold $\M$.
  Then for $x^{*}\in\argmin_{x\in \X}f(x)$, it holds that
  \begin{equation}
    \label{eq:vimin}
   \langle\nabla f(x^{*}),\exponinv{x^{*}}(y)\rangle\ge 0, \forall y\in \X.
 \end{equation}
\end{proposition}
 Note that this also directly implies that if we have a g-concave function $f$ and $x^{*}\in\argmax_{x\in \X}f(x)$, then $\langle\nabla f(x^{*}),\exponinv{x^{*}}(y)\rangle\le 0,  \forall y\in \X$.
\begin{proof}\linkofproof{prop:vi}
  Let $f$ be g-convex and $x^{*}\in\argmin_{x\in \X}f(x)$.
  Let $F(t)\defi f(\gamma (t))$, where $\gamma (t)$ is a geodesic such that $\gamma (0)=x^{*}$ and $\gamma(1)=x$.
  Then $F$ reaches its minimum at $t=0$ and we have that $0\le F'(0)=\langle \nabla f(x^{*}),\exponinv{x^{*}}(x)\rangle$.
\end{proof}

\begin{lemma}\label{lemma:strong_g_convexity_consequences}
  For a $\mu$-strongly g-convex function $f:\X\rightarrow \mathbb{R}$ where $\X\subset \M$ is a compact g-convex subset we have
  \begin{align}\label{eq:strcvx1}
  \begin{aligned}
    \frac{\mu}{2}\dist^{2}(x^{*},y)&\le f(y)-f(x^{*}) \\
      \mu \dist^{2}(x,y)&\le \langle\exponinv{x}(y),\Gamma{y}{x}\nabla f(y)-\nabla f(x)\rangle\\
      &=\langle\log_y(x),\Gamma{x}{y}\nabla f(x)-\nabla f(y)\rangle
  \end{aligned}
  \end{align}
  where $x^{*}= \argmin_{x\in \X}f(x)$.
  Equivalently if $f$ is $\mu$-strongly g-concave then
  \begin{align}\label{eq:strccv1}
  \begin{aligned}
    \frac{\mu}{2}\dist^{2}(x^{*},y)&\le f(x^{*})-f(y) \\
      \mu \dist^{2}(x,y)&\le \langle -\exponinv{x}(y),\Gamma{y}{x}\nabla f(y)-\nabla f(x)\rangle\\
      &=\langle -\log_y(x),\nabla f(x)-\Gamma{y}{x}\nabla f(y)\rangle
  \end{aligned}
  \end{align}
where $x^{*}=\argmax_{x\in \X}f(x)$.
\end{lemma}
\begin{proof}
 By strong convexity, we have
\begin{equation*}
   f(x^{*})\le f(y) + \langle\nabla f(x^{*}),-\exponinv{x^{*}}(y)\rangle -\frac{\mu}{2} \dist^{2}(x^{*},y).
\end{equation*}
Using \eqref{eq:vimin}, we conclude that
\begin{equation*}
 \frac{\mu}{2}\dist^2(x^{*},y) \le f(y)-f(x^{*}).
\end{equation*}
Further, by strong convexity, we can write
\begin{align*}
   f(x)&\le f(y) + \langle\nabla f(x),-\exponinv{x}(y)\rangle -\frac{\mu}{2} \dist^{2}(x,y)\\
   f(y)&\le f(x) + \langle\nabla f(y),-\exponinv{y}(x)\rangle -\frac{\mu}{2} \dist^{2}(x,y).
\end{align*}
Adding both inequalities, we get
\begin{align*}
 \mu \dist^{2}(x,y)  &\le  \langle\nabla f(y),-\exponinv{y}(x)\rangle+ \langle\nabla f(x),-\exponinv{x}(y)\rangle\\
  &= \langle \Gamma{y}{x}\nabla f(y)-\nabla f(x),\log_x(y)\rangle\\
  &= \langle\Gamma{x}{y}\nabla f(x)- \nabla f(y),\log_y(x)\rangle
\end{align*}
where the last two equalities follow from transporting the scalar products to $\mathcal{T}_x$ and $\mathcal{T}_y$ respectively.
The proof for the strongly g-concave follows by the same argument.
\end{proof}

\begin{lemma} \label{lem:lip}
  Assume $f$ satisfies \cref{ass:fct}, define $y^{*}(x)\defi\argmax_{y\in \Y}f(x,y)$, $x^{*}(y)\defi\argmax_{x\in \X}f(x,y)$, $\phi(x)\defi\max_{y\in \Y}f(x,y)$ and $\Psi(y)\defi \min_{x\in \X}f(x,y)$  then
  \begin{enumerate}
    \item $y^{*}(\cdot)$ is $\frac{\Lxy}{\muy}$-Lipschitz. \label{item:y_opt_lip}
    \item $x^{*}(\cdot)$ is $\frac{\Lxy}{\mux}$-Lipschitz. \label{item:x_opt_lip}
    \item $\phi(x)$ is $\mux$-strongly g-convex. \label{item:phi_sc}
    \item $\Psi(y)$ is $\muy$-strongly g-concave. \label{item:Psi_sc}
  \end{enumerate}
\end{lemma}
\begin{proof}
  By \cref{prop:vi}, we have
 \begin{align}
   \langle\exponinv{y^{*}(x)}(y),\nabla_yf(x,y^{*}(x))\rangle\le 0,\quad \forall y\in \Y \label{eq:opt1}\\
   \langle \exponinv{y^{*}(z)}(y),\nabla_yf(z,y^{*}(z))\rangle\le 0,\quad  \forall y\in \Y\label{eq:opt2}
 \end{align}
Sum up \cref{eq:opt1} with $y=y^{*}(z)$ transporting the scalar product to $\mathcal{T}_{y^{*}(z)}$ and \cref{eq:opt2} with $y=y^{*}(x)$,
\begin{align*}
    \langle \exponinv{y^{*}(z)}(y^{*}(x)),\nabla_yf(z,y^{*}(z))-\Gamma{}{y^{*}(z)}\nabla_yf(x,y^{*}(x))\rangle\le 0.
\end{align*}
    Then since $f(x,\cdot)$ is $\muy$-strongly g-concave, we have by \cref{lemma:strong_g_convexity_consequences}
\begin{equation*}
    \muy \dist^2(y^{*}(x),y^{*}(z)) + \langle\exponinv{y^{*}(z)}(y^{*}(x)), \Gamma{}{y^{*}(z)}\nabla_yf(x,y^{*}(x)) - \nabla_{y}f(x,y^{*}(z)\rangle\le 0.
\end{equation*}
Summing these two equations we get
\begin{equation*}
  \muy \dist^2(y^{*}(x),y^{*}(z))   + \langle\exponinv{y^{*}(z)}(y^{*}(x)),\nabla_y f(z,y^{*}(z)) - \nabla_{y}f(x,y^{*}(z)\rangle\le 0.
\end{equation*}
Further,
 \begin{align*}
\begin{aligned}
     \dist(y^{*}(x),y^{*}(z))&\le \dist^{-1}(y^{*}(x),y^{*}(z)) \muy^{-1} \langle\exponinv{y^{*}(z)}(y^{*}(x)), \nabla_{y}f(x,y^{*}(z)-\nabla_yf(z,y^{*}(z)) \rangle\\
     &\leq \muy^{-1} \norm{\nabla_{y}f(x,y^{*}(z)-\nabla_y f(z,y^{*}(z))}\\
     &\le \frac{\Lxy}{\muy} \dist(x,z),
\end{aligned}
\end{align*}   
where we used gradient Lipschitzness. This concludes the proof for \cref{item:y_opt_lip}.
The proof of \cref{item:x_opt_lip} works in the same way using strong g-convexity instead of strong g-concavity.
The proofs of \cref{item:phi_sc,item:Psi_sc} follow directly by \cref{lem:strong_cvxty_of_max}.
\end{proof}

\end{document}

%% file: make_cleveref_work.tex
\usepackage[capitalise,nameinlink,noabbrev]{cleveref} 

\makeatletter
  \newcommand{\jmlrBlackBox}{\rule{1.5ex}{1.5ex}}
  
  \newcommand{\jmlrQED}{\hfill\jmlrBlackBox\par\bigskip}
\providecommand{\proofname}{Proof}
  \newenvironment{proof}%
  {%
   \par\noindent{\bfseries\upshape \proofname\ }%
  }%
  {\jmlrQED}
  \newcommand*{\theorembodyfont}[1]{%
    \renewcommand*{\@theorembodyfont}{#1}%
  }
  \newcommand*{\@theorembodyfont}{\normalfont\itshape}%
  \newcommand*{\theoremheaderfont}[1]{%
    \renewcommand*{\@theoremheaderfont}{#1}%
  }
  \newcommand*{\@theoremheaderfont}{\normalfont\bfseries }%
  \newcommand*{\theoremsep}[1]{%
    \renewcommand*{\@theoremsep}{#1}%
  }
  \newcommand*{\@theoremsep}{}%
  \newcommand*{\theorempostheader}[1]{%
    \renewcommand*{\@theorempostheader}{#1}%
  }
  \newcommand*{\@theorempostheader}{}%
  \let\jmlr@org@newtheorem\newtheorem
  \renewcommand*{\newtheorem}{\@ifstar\jmlr@snewtheorem\jmlr@newtheorem}
  \newcommand*{\jmlr@snewtheorem}[2]{%
    \cslet{jmlr@thm@#1@body@font}{\@theorembodyfont}%
    \cslet{jmlr@thm@#1@header@font}{\@theoremheaderfont}%
    \cslet{jmlr@thm@#1@sep}{\@theoremsep}%
    \cslet{jmlr@thm@#1@postheader}{\@theorempostheader}%
    \newenvironment{#1}%
    {%
      \trivlist
        \item
        [%
          \hskip\labelsep{\csuse{jmlr@thm@#1@header@font}#2%
            \csuse{jmlr@thm@#1@postheader}%
          }%
        ]%
        \mbox{}\csuse{jmlr@thm@#1@sep}%
        \csuse{jmlr@thm@#1@body@font}%
    }%
    {%
      \endtrivlist
    }%
  }
  \newcommand{\jmlr@newtheorem}[1]{%
    \cslet{jmlr@thm@#1@body@font}{\@theorembodyfont}%
    \cslet{jmlr@thm@#1@header@font}{\@theoremheaderfont}%
    \cslet{jmlr@thm@#1@sep}{\@theoremsep}%
    \cslet{jmlr@thm@#1@postheader}{\@theorempostheader}%
    \jmlr@org@newtheorem{#1}%
  }
  \renewcommand*{\@xthm}[2]{%
    \def\@jmlr@currentthm{#1}%
    \@begintheorem{#2}{\csname the#1\endcsname}%
    \ignorespaces
  }
  \def\@ythm#1#2[#3]{%
    \def\@jmlr@currentthm{#1}%
    \@opargbegintheorem{#2}{\csname the#1\endcsname}{#3}%
    \ignorespaces
  }
  \renewcommand*{\@begintheorem}[2]{%
    \ifdef{\@jmlr@currentthm}%
    {%
      \letcs{\jmlr@this@theoremheader}{jmlr@thm@\@jmlr@currentthm @header@font}%
      \letcs{\jmlr@this@theorembody}{jmlr@thm@\@jmlr@currentthm @body@font}%
      \letcs{\jmlr@this@theoremsep}{jmlr@thm@\@jmlr@currentthm @sep}%
      \letcs{\jmlr@this@theorempostheader}%
         {jmlr@thm@\@jmlr@currentthm @postheader}%
    }%
    {%
      \let\jmlr@this@theorembody\@theorembodyfont
      \let\jmlr@this@theoremheader\@theoremheaderfont
      \let\jmlr@this@theoremsep\@theoremsep
      \let\jmlr@this@theorempostheader\@theorempostheader
    }%
    \trivlist
      \item
       [%
        \hskip\labelsep{\jmlr@this@theoremheader #1\ #2%
           \jmlr@this@theorempostheader}%
       ]%
      \mbox{}\jmlr@this@theoremsep
      \jmlr@this@theorembody
  }
  \renewcommand*{\@opargbegintheorem}[3]{%
    \ifdef{\@jmlr@currentthm}%
    {%
      \letcs{\jmlr@this@theoremheader}{jmlr@thm@\@jmlr@currentthm @header@font}%
      \letcs{\jmlr@this@theorembody}{jmlr@thm@\@jmlr@currentthm @body@font}%
      \letcs{\jmlr@this@theoremsep}{jmlr@thm@\@jmlr@currentthm @sep}%
      \letcs{\jmlr@this@theorempostheader}%
         {jmlr@thm@\@jmlr@currentthm @postheader}%
    }%
    {%
      \let\jmlr@this@theorembody\@theorembodyfont
      \let\jmlr@this@theoremheader\@theoremheaderfont
      \let\jmlr@this@theoremsep\@theoremsep
      \let\jmlr@this@theorempostheader\@theorempostheader
    }%
    \trivlist
     \item[\hskip\labelsep{\jmlr@this@theoremheader #1\ #2\ (#3)%
       \jmlr@this@theorempostheader}]%
     \mbox{}\jmlr@this@theoremsep
     \jmlr@this@theorembody
  }
\makeatother

%% file: algorithm_config.tex
\usepackage{algorithm}
\usepackage{algcompatible}
\algnewcommand{\lst}{\texttt{lst}}
\algnewcommand{\slst}{\texttt{slst}}
\algnewcommand{\SEND}{\textbf{send}}

\newsavebox{\algleft}
\newsavebox{\algright}

\makeatletter
\newcounter{algorithmicH}
\let\oldalgorithmic\algorithmic
\renewcommand{\algorithmic}{%
  \stepcounter{algorithmicH}
  \oldalgorithmic}
\renewcommand{\theHALG@line}{ALG@line.\thealgorithmicH.\arabic{ALG@line}}
\makeatother

%% file: bibliography_config.tex
\usepackage[natbib, backend=biber, maxcitenames=3, minalphanames=3, maxbibnames=99, style=alphabetic, hyperref, backref, useprefix=true, uniquename=false, doi=false,url=false,eprint=false]{biblatex}

\usepackage{csquotes}               

\bibliography{refs}

\newbibmacro{string+doiurlisbn}[1]{%
  \iffieldundef{doi}{%
    \iffieldundef{url}{%
      \iffieldundef{isbn}{%
        \iffieldundef{issn}{%
          #1%
        }{%
          \href{http://books.google.com/books?vid=ISSN\thefield{issn}}{#1}%
        }%
      }{%
        \href{http://books.google.com/books?vid=ISBN\thefield{isbn}}{#1}%
      }%
    }{%
      \href{\thefield{url}}{#1}%
    }%
  }{%
    \href{https://doi.org/\thefield{doi}}{#1}%
  }%
}

\DeclareFieldFormat{title}{\usebibmacro{string+doiurlisbn}{\mkbibemph{#1}}}
\DeclareFieldFormat[article,incollection,inproceedings]{title}%
    {\usebibmacro{string+doiurlisbn}{\mkbibquote{#1}}}